\documentclass[12pt,leqno]{article}
\usepackage{amssymb}
\usepackage{mathtools}
\usepackage{url}%
\usepackage{amsrefs}
\usepackage{rotating}
\usepackage{amsmath}
\usepackage{tabularx}
\usepackage{verbatim}
\usepackage{theorem}
\usepackage[matrix,tips,frame,color,line,poly]{xy}
\newcommand{\Gext}{\negthinspace\negthinspace\phantom{a}^\delta G}
\newcommand\kappaarrow[2]{#1\overset\kappa\rightarrow#2}
\newtheorem{theorem}[equation]{Theorem}
\newtheorem{corollary}[equation]{Corollary}
\newtheorem{definition}[equation]{Definition}
\newtheorem{lemma}[equation]{Lemma}
\newtheorem{desideratum}[equation]{Desideratum}
\newtheorem{conjecture}[equation]{Conjecture}
\newtheorem{proposition}[equation]{Proposition}
\newtheorem{remark}[equation]{Remark}
{\theorembodyfont{\rmfamily}
\newtheorem{remarkplain}[equation]{Remark}

\newtheorem{exampleplain}[equation]{Example}

\newtheorem{mytable}[equation]{Table}
}

\font\temporary=manfnt
\def\dbend{{\temporary\char127}} 
\def\danger{\begin{trivlist}\item[]\noindent%
\begingroup\hangindent=3pc\hangafter=-2
\def\par{\endgraf\endgroup}%
\hbox to0pt{\hskip-\hangindent\dbend\hfill}\ignorespaces}
\def\enddanger{\par\end{trivlist}}

\newcommand{\qed}{\hfill $\square$ \medskip}
\newenvironment{proof}[1][Proof]{\noindent\textbf{#1.} }{\qed}

\newcommand{\Aut}{\mathrm{Aut}}

\newcommand{\Out}{\mathrm{Out}}

\renewcommand{\int}{\mathrm{int}}

\newcommand{\defect}{\mathrm{def}}

\newcommand{\Norm}{\mathrm{Norm}}

\newcommand{\X}{\mathcal X}
\newcommand{\M}{\mathcal M}
\newcommand{\B}{\mathcal B}
\newcommand{\D}{\mathcal D}

\newcommand{\caP}{\mathcal P}

\newcommand{\R}{\mathbb R}
\newcommand{\C}{\mathbb C}
\newcommand{\Z}{\mathbb Z}

\newcommand{\T}{\mathbb T}

\renewcommand{\H}{\mathbb H}
\newcommand{\h}{\mathfrak h}
\renewcommand{\a}{\mathfrak a}

\newcommand{\bD}{\overline D}

\newcommand{\ch}[1]{#1^\vee}

\newcommand{\bH}{\mathbf H}
\renewcommand{\bD}{\mathbf D}
\newcommand\Chat{\wh C}
\renewcommand\P{\wh P^\sigma}
\newcommand\U{\wh U_\kappa}

\newcommand\Mu{\wh\mu^\sigma}
\renewcommand\T{\wh T}
\renewcommand\a{\hat a}
\renewcommand{\sec}[1]{\section{#1}
\renewcommand{\theequation}{\thesection.\arabic{equation}}
  \setcounter{equation}{0}}
\newcommand{\subsec}[1]{\subsection{#1}
\renewcommand{\theequation}{\thesubsection.\arabic{equation}}
  \setcounter{equation}{0}}

\newcommand\inv{^{-1}}

\newcommand\wt{\widetilde}
\newcommand\wh{\widehat}

\newcommand{\kappaless}{\overset\kappa\prec}

\newcommand{\oneH}{\phantom{a}^1H}
\newcommand{\oneLambda}{\phantom{a}^1\Lambda}
\newcommand{\Hext}{\negthinspace\negthinspace\phantom{a}^\delta H}
\newcommand{\mypmod}{\hskip-6pt\pmod}
\begin{document}
\title{Computing twisted KLV polynomials}
\author{Jeffrey Adams}
\maketitle

This is an attempt to convert  \cite{lv2012b} to the atlas setting, and
write down explicit recursion formulas  for the Kazhdan-Lusztig-Vogan
polynomials which arise.

\sec{The Setup}
\label{s:setup}

The starting point is: a group $G$, a Cartan involution $\theta$, and
another involution $\sigma$ of finite order, commuting with $\theta$.

It is natural to consider the coset $\sigma K=\{\sigma\circ
\int(k)\mid k\in K\}\subset \Aut(G)$. Every element of this coset commutes with $\theta$.

We're mainly interested when
 $\sigma$ is an involution, especially the case $\sigma=\theta$.

Now fix a pinning $\caP=(H,B,\{X_\alpha\})$, and write $\delta,\epsilon\in \Aut(G)$ for the images
via the embedding $\Out(G)\hookrightarrow \Aut(G)$ (the image consists of $\caP$-distinguished automorphisms). 
Then $\delta,\epsilon\in\Aut(G)$  commute.

We now introduce the usual atlas structure. 
See \cite{algorithms} for details.
Let $\Gext=G\rtimes\langle\delta\rangle$ be the usual extended group;
$\sigma$ acts on it, trivially on $\delta$.

Recall $\X=\{\xi\in \Norm_{G\delta}(H)\,|\, \xi^2\in Z(G)\}/H$, and
$\wt\X$ is the numerator.  We'll write $x$ for elements of $\X$, 
$\xi$ for elements of $\wt\X$, and $p:\wt\X\rightarrow\X$. 
It is important to distinguish between elements of $\X$ and $\wt\X$.

For $\xi\in \wt\X$, let $\theta_{\xi}=\int(\xi)\in\Aut(G)$,
$K_{\xi}=G^{\theta_{\xi}}$.The restriction of $\theta_{\xi}$ to $H$ only depends on
$p(\xi)\in\X$, and is denoted $\theta_x$. It is important to remember 
$\theta_x$ is only an involution of $H$, not of $G$,  and  
$K_x$ is not well defined.

It is immediate that $\sigma(\X)=\X$.

After conjugating we can assume $\theta=\int(\xi_0)$ for 
some $\xi_0\in\wt\X$. Let $x_0=p(\xi_0)$, and 
define $\X_\theta=\{x\in \X\,|\, x\text{ is $G$-conjugate to }x_0\}$
($G$-conjugacy of elements of $\X$ is well defined).
Let $K=G^\theta=K_{\xi_0}$. 
Then there is a canonical bijection $\X_\theta\leftrightarrow K\backslash G/B$.

Since $\{\sigma,\theta\}=1$, $\sigma$ acts on $K\backslash G/B$.
Write $x\rightarrow\sigma^\dagger(x)$ for the automorphism of
$X_\theta$, corresponding to the action of $\sigma$ on $K\backslash
G/B$. We need to compute $\sigma^\dagger$.

The condition $\{\sigma,\theta\}=1$ holds if and only if
$\sigma(\xi_0)\in \xi_0Z$. 
It is convenient to define
$z_0=\sigma(\xi_0\inv)\xi_0\in Z^{-\sigma}$,
so 
$$
\sigma(\xi_0)=\xi_0z_0\inv.
$$
(The choice of inverse on $z_0$ is so that it goes away later.)
It makes sense to write:
\begin{equation}
\label{e:z_0}
\sigma(x_0)=x_0z_0\inv.
\end{equation}
(both sides being defined up to conjugacy by $H$).

\begin{proposition}
\label{p:kgbaction}
After replacing $\sigma$ by another element of $\sigma K$, we
may assume $\sigma$ normalizes $H$.  Define $v\in\Norm_G(H)$ by 
$\sigma(B)=vBv\inv$. Then
\begin{subequations}
\renewcommand{\theequation}{\theparentequation)(\alph{equation}}  
\begin{equation}
\sigma^\dagger(x)=v\inv \sigma(x)v z_0 \quad(x\in X_\theta).
\end{equation}
If
 sigma preserves the base orbit, i.e. $\sigma(K\cdot
B)=K\cdot B$, then after replacing $\sigma$ with another element of $\sigma
K$ we may assume $\sigma$ normalizes $(B,H)$. 
In this case
\begin{equation}
\sigma^\dagger(x)=\sigma(x)z_0.
\end{equation}
\end{subequations}
\end{proposition}

The coset $vH$ of $v$ in $W=\Norm_G(H)/H$ is well defined, and it makes no
difference if we view $v$ as an element of $\Norm_G(H)$ or $W$.

\begin{remark}
After replacing $\sigma$ with another element of $\sigma K$ we may
assume $\sigma(H)=H$, and $\sigma(B\cap K)=B\cap K$. Having done
this, we conclude $v$ normalizes $B\cap K$. The normalizer in $W$ of
$B\cap K$, equivalently $\rho_K$, is very small; in particular 
it is a product of $A_1$ factors.
\end{remark}

Assume $\sigma(K\cdot B)=K\cdot B$, i.e. $v=1$. Then we are in the
following setting. We have an involution $\theta$ and an automorphism
$\sigma$, preserving $(B,H)$, and commuting with $\theta$.  The induced automorphism
of $(W,S)$
also written $\sigma$, has finite order ($S$ is the set of simple reflections).
Finally the corresponding  automorphism $\sigma^\dagger$ of $X_\theta$ 
is $\sigma^\dagger(x)=\sigma(x)z_0$.

The main case of interest is:

\begin{corollary}
\label{c:kgbaction}
Suppose $\sigma=\theta$. After replacing $\theta$ with a
$G$-conjugate, we have
$\sigma^\dagger(x)=\delta(x)$, and $\theta,\delta$ commute. 
\end{corollary}

\begin{proof}
After replacing $\theta$ with a $G$-conjugate we may assume $\xi_0\in
H\delta$. Then $\sigma(B)=\theta(B)=B$, also $z_0=1$, so
$\sigma^\dagger(x)=\sigma(x)=\theta(x)$. But $\theta=\int(\xi_0)$ and $\delta$
differ by an element of $H$, so $\theta(x)=\delta(x)$ (since $X$ has
conjugation by $H$ built in). 
\end{proof}

\begin{proof}[Proof of the Proposition]

We recall a few details of the bijection $X_\theta\leftrightarrow K\backslash G/B$ \cite[Sections 8 and 9]{algorithms}.

Let $\P=\{(x,B')\mid x\in X_\theta,B'\in\B\}/G$.
There are bijections:
\begin{equation}
\begin{aligned}
X_\theta&\longleftrightarrow&\P&\longleftrightarrow K\backslash\B\\
x&\longrightarrow& (x,B)\\
&&(x_0,B')&\longleftarrow K\cdot B'
\end{aligned}
\end{equation}
Since $H$ is a fundamental Cartan subgroup with respect to $\theta$, 
and all such Cartan subgroups are $K$-conjugate, we can modify $\sigma$ by $k$ 
so that $\sigma$ normalizes $H$.

Here is the computation.
$$
\begin{aligned}
X_\theta\ni x&\rightarrow (x,B)\in\P\\
&=(gx_0g\inv,B)\quad (x=gx_0g\inv)\\
&=(x_0,g\inv Bg)\\
&\rightarrow g\inv Bg\in K\backslash\B\\
&\rightarrow \sigma(g\inv Bg)\text{ by the action of }\sigma\text{ on }K\backslash\B\\
&=\sigma(g\inv)vBv\inv \sigma(g)\text{ where }\sigma(B)=vBv\inv\\
&\rightarrow(x_0,\sigma(g\inv)vBv\inv \sigma(g))\in\P\\
&=(v\inv\sigma(g)x_0\sigma(g\inv)v,B)\\
&\rightarrow(v\inv\sigma(g)x_0\sigma(g\inv)v\in X_\theta\\
&=v\inv\sigma(g\sigma(x_0)g\inv)v,\\
&=v\inv\sigma(gx_0z_0\inv g\inv)v\quad(\text{by }\eqref{e:z_0})\\
&=v\inv\sigma(gx_0g\inv)\sigma(z_0\inv)v\\
&=v\inv\sigma(x)\sigma(z_0\inv)v\quad(gx_0g\inv=x)\\
&=v\inv\sigma(x)vz_0\quad(\sigma(z_0\inv)=z_0)\\
\end{aligned}
$$

\end{proof}

\sec{The group $W^\sigma$ and the twisted Hecke algebra $\bH$}
\label{s:hecke}

We continue in the  setting of Section \ref{s:setup},
and we now assume $\sigma$ is an involution.
Let $K=G^\theta$. 

For simplicity let's assume $\sigma(K\cdot B)=K\cdot B$, so $v=1$
(see Proposition \ref{p:kgbaction}).
Then, after replacing $\sigma$ with
an element of $\sigma K$,
we may assume $\sigma$ commutes with $\theta$, and satisfies
$\sigma(B,H)=(B,H)$. Although $\sigma$ may not have finite order,
$\sigma^2=\int(h)$ for some $h\in H$, so $\sigma$ induces
an involution, also denoted $\sigma$, of 
$(W,S)$.
Let $\overline S$ be the set of orbits of the action of $\sigma$ on $S$.

We are primarily interested in the case $\sigma=\theta$.
In this case, 
after conjugating  by $G$ we may assume 
the induced automorphism of $(W,S)$ is $\delta$
(Corollary \ref{c:kgbaction}), and $\{\theta,\delta\}=1$.

If $\kappa\in\overline S$ let $W(\kappa)$ be the subgroup of $G$
generated by $\kappa$. Write $\kappa=\{s_\alpha\}$ or $\{s_\alpha,s_\beta\}$,
with $\alpha,\beta$ simple.
In each case there is a unique long element  $w_\kappa\in W(\kappa)$.
Define $\ell(\kappa)=\ell(w_\kappa)$.
\begin{equation}
W(\kappa)=
\begin{cases}
\ell(\kappa)=1\quad S_2&\sigma(s_\alpha)=s_\alpha\\
\ell(\kappa)=2\quad S_2\times S_2&\sigma(s_\alpha)=s_\beta, \langle\alpha,\ch\beta\rangle=0\\
\ell(\kappa)=3\quad S_3&\sigma(s_\alpha)=s_\beta, \langle\alpha,\ch\beta\rangle=-1\\
\end{cases}
\end{equation}

Lusztig and Vogan define a Hecke algebra $\bH$ over $\Z[u,u\inv]$ ($u$ is an
indeterminate). See the end of \cite[Section 3.1]{lv2012b}.
It has generators $T_w$ ($w\in W^\sigma$) and relations
\begin{equation}
\begin{aligned}
&T_wT_{w'}=T_{ww'}\quad w,w'\in W^\sigma,
\ell(ww')=\ell(w)+\ell(w')\\
&(T_{w_\kappa}+1)(T_{w_\kappa}-u^{\ell(w_\kappa)})=0\quad(\kappa\in
\overline S)
\end{aligned}
\end{equation}

The quotient of the root system by $\sigma$ is itself a root
system (nonreduced if length 3 occurs), with simple roots parametrized
by $\overline S$, and $W^\sigma$ is the Weyl group of this root
system.  In particular $W^\sigma$ is generated by
$\{w_\kappa\mid\kappa\in\overline S\}$, and $\bH$ is generated by
$\{T_{w_\kappa}\mid\kappa\in\overline S\}$.  So in fact $\bH$ has
generators and relations

\begin{equation}
\label{e:quadratic}
\begin{aligned}
&T_{w_\kappa}T_{w'}=T_{w_\kappa w'}\quad
\kappa\in \overline S, \quad
\ell(w_\kappa w')=\ell(w_\kappa)+\ell(w')\\
&(T_{w_\kappa}+1)(T_{w_\kappa}-u^{\ell(w_\kappa)})=0\quad(\kappa\in\overline S)
\end{aligned}
\end{equation}
This makes $\bH$ a quasisplit Hecke algebra \cite[\S4.7]{lv2012b}.

Let $\D=\mathcal Z[x]$ be the subset of $\mathcal Z$ having to do with
$x$. That is $\mathcal Z[x]\subset\X[x]\times\ch\X$ ($\ch\X$ is the
dual KGB space), 
where $\mathcal Z[x]\simeq K_\xi\backslash G/B$, and 
$\D$ is parametrized by the  $K_{\xi}$-invariant local systems on
$K_{\xi}\backslash G/B$. 
Then that $\sigma$ acts on $\D$, and let $\D^\sigma$ be the fixed points.

Lusztig and Vogan define a $\bH$ module $M$, with basis
$\{a_\gamma\mid\gamma\in\D^\sigma\}$. 
We are going to write down formulas for the action of $\bH$ on $M$.


\sec{Extended Cartans and Parameters}
\label{s:extended}

If $\gamma\in\mathcal D^\sigma$  there is an isomorphism
between the representations
parametrized by $\gamma$ and $\sigma(\gamma)$. 
  It is possible to
normalize this isomorphism to have ``square 1'', i.e. in such a way
that there are two choices, $\pm\alpha_\gamma$.  This leads to {\it
  extended parameters}: for each $\gamma\in\D^\sigma$ there are two
extended parameters corresponding to the two choices of
$\alpha_\gamma$.  Write $\wh\gamma$ for an extended parameter
corresponding to $\gamma$.  

Strictly speaking, the module $M$ is spanned by vectors
$a_{\wh\gamma}$ as $\wh\gamma$ runs over extended parameters, and the
Hecke algebra action is naturally defined in these terms.  
If $\wh\gamma^{\pm}$ are the two choices of extension, 
in the module $M$ we have
$a_{\wh\gamma^-}=-a_{\wh\gamma^+}$, and the dimension of $M$ is
$|\D^\sigma|$.

\begin{desideratum}
\label{desideratum}
For each parameter $\gamma\in\D^\sigma$, it is possible to choose one
extended parameter, denoted  $\wh\gamma_+$
so that the formulas of \cite{lv2012b} hold with $\wh\gamma_+$ and
$a_{\wh\gamma_+}$ 
everywhere.
\end{desideratum}

\subsec{Extended Cartans}
We probably don't need this subsection and the next one. They are
vestiges of a version in which we worked in terms of extended
parameters. But it might be helpful to include a few basic facts.

In this section and the next we assume $\sigma=\delta$. Probably this
isn't serious, but in any event in the rest of the paper we only
assume $\sigma$ is an involution.




We are interested in
KGB elements $x\in\X$ which are fixed by $\delta$. A key point is that
if $\delta(x)=x$, and $\xi\in p\inv(x)\in\wt\X$, then
$\delta(\xi)=h\xi h\inv$ for some $h\in H$. We cannot 
assume we can choose $\xi$ so that   $\delta(\xi)=\xi$. 

\begin{lemma}
\label{l:onlyfixedx}
Suppose $\xi\in\wt\X$, and let $x=p(\xi)\in\X$.
The following conditions are equivalent. 
\begin{enumerate}
\item $\theta_\xi$ normalizes $\Hext$, and $(\Hext)^{\theta_\xi}$
  meets both components of $\Hext$;
\item $\delta(x)=x$.
\end{enumerate}
\end{lemma}

\begin{proof}
Suppose (1) holds.
The second part of (1) says that $\xi(t\delta)\xi\inv=t\delta$ for
some $t\in H$, i.e.
\begin{subequations}
\renewcommand{\theequation}{\theparentequation)(\alph{equation}}  
\begin{equation}
\theta_x(t)(\xi\delta\xi\inv)=t\delta. 
\end{equation}

The first part of (1) says $\xi\delta\xi\inv=h\delta$ for some $h\in
H$, i.e.
\begin{equation}
\delta(\xi)=h\inv\xi
\end{equation}
Plug in $\xi\delta\xi\inv=h\delta$ to (a):
$\theta_x(t)h\delta=t\delta$, so $h\inv=t\inv\theta_x(t)$. 
Then by (b):
\begin{equation}
\delta(\xi)=h\inv\xi=t\inv\theta_x(t)\xi=t\inv\xi t.
\end{equation}
\end{subequations}
Projecting to $\X$ this says $\delta(x)=x$.

Conversely, suppose $\delta(x)=x$. By definition this means
$\delta(\xi)=h\inv\xi h$ for some $h\in H$.
Note that
\begin{equation}
 \delta(\xi)=h\inv\xi h\Leftrightarrow\xi(h\delta)\xi\inv=h\delta.
\end{equation}
In other words the second condition of (1) holds. Also the first
condition holds:
the right hand side gives $\theta_x(h)\xi\delta\xi\inv=h\delta$, i.e.
$\xi\delta\xi\inv\in H\delta$. 
\end{proof}

From now on we will usually assume $\delta(x)=x$.

\begin{definition}
Suppose $\xi\in\wt\X$. Let $x=p(\xi)\in\X$, and assume $\delta(x)=x$.
The extended Cartan defined by $\xi$ is
$\oneH_\xi=(\Hext)^{\theta_\xi}$. It contains $H^{\theta_x}$ as a subgroup
of index $2$, and meets both components of $\Hext$.
\end{definition}
In other words 
\begin{equation}
\oneH_\xi=\langle H^{\theta_x},h\delta\rangle
\end{equation}
where  $\xi(h\delta)\xi\inv=h\delta$, equivalently $h\inv\xi h=\delta(\xi)$.

This is related to 
\cite[Definition 13.5]{unitaryDual}. Note that $h\delta$
normalizes $\Delta^+$.

\subsec{Extended Parameters}
\label{s:extended}

We work only at a fixed regular infinitesimal character, so we fix
$\lambda\in\h^*$, dominant for $\Delta^+$.

By a {\it character} we mean a pair
$(x,\Lambda)$, where $\Lambda$ is 
an $(\h,H^{\theta_x})$-module. We're ignoring the $\rho$-cover; this
isn't hard to fix.
The differential of $\Lambda$  (the $\h$ part) is $\lambda\in\h^*$, so we usually
identify $\Lambda$ with a character of $H^{\theta_x}$.

\begin{definition}
Given $\xi\in\wt\X$, an extended character is a pair $(\xi,\oneLambda)$
where $\oneLambda$ is an
$(\h,\negthinspace\negthinspace\negthinspace\negthinspace\oneH_\xi)$ module. Equivalence of extended characters is by
conjugation by $H$.
\end{definition}

Recall given a parameter $(x,y)$ as usual, compatible with $\lambda$ we obtain a
character $(x,\Lambda)$, as above we think of $\Lambda$ as a character
of $H^{\theta_x}$. 

\begin{definition}
An extended parameter, at infinitesimal character $\lambda$, is a
quadruple $(\xi,y,h\delta,z)$ satisfying the following conditions. Set
$x=p(\xi)\in\X$, and we assume $\delta(x)=x$.
\begin{enumerate}
\item $(x,y)$ is a parameter at $\lambda$, defining a character
$\Lambda$ of $H^{\theta_x}$.
\item $\Lambda^\delta=\Lambda$ (i.e. $\Lambda$ is fixed by $\delta$),
\item $h\delta$ is in the extended group $(\Hext)^{\theta_\xi}$, i.e. $h\delta$ commutes with $\xi$,
\item $z\in\C^*, z^2=\Lambda(h\delta(h))$
\end{enumerate}
Equivalence of extended parameters is generated by conjugation by $H$,
and
\begin{equation}
(\xi,y,h\delta,z)\simeq
(\xi,y,th\delta,\Lambda(t)z)\quad (t\in H^{\theta_x}).
\end{equation}
\end{definition}

\begin{remarkplain}
  Condition (2) implies (but is not equivalent to):
  $\delta^t(y)=y$. So we may as well assume this holds as well as $\delta(x)=x$.
\end{remarkplain}
\begin{proposition}
There is a bijection between equivalence classes of extended
characters and equivalence classes of extended parameters. 
\end{proposition}

The bijection is
\begin{equation}
(\xi,y,h\delta,z)\leftrightarrow(\xi,\oneLambda)
\end{equation}
From left to right, take $\oneLambda|_{H^{\theta_x}}$ to be the
character defined by $(x=p(\xi),y)$, and $\oneLambda(h\delta)=z$. 
Conversely, given $\oneLambda$, choose $y$ so that $(x,y)$ corresponds
to $\oneLambda|_{H^{\theta_x}}$. Choose any
$h\delta\in(\Hext)^{\theta_\xi}$, and let $z=\oneLambda(h\delta)$. 
There are a few straightforward checks that this works.
One of the main points is that if we choose
$h_i\delta\in(\Hext)^{\theta_\xi}$ ($i=1,2$), then $h_2=h_1t$ for
$t\in H^{\theta_x}$, which is taken care of by the equivalence. 

\medskip





\sec{Cayley transforms and cross actions}
\label{s:cayley}

Cayley transforms and cross actions can naturally be defined in terms
of extended parameters.
(This was done in an earlier version of these notes.)
As discussed at the beginning of Section \ref{s:extended}, implicit in \cite{lv2012b} is the assertion that,
for each parameter $\gamma$, there is a choice of extended parameter, 
which we'll label $\wh\gamma^+$, so that the following formulas hold 
with $\wh\gamma^+$ in place of $\gamma$ everywhere {\it except} in cases
{\tt 2i12/2r21}. For these see Section \ref{s:2i12}.

Some of these new  ``Cayley transforms'' are iterated Cayley transforms, but some involve 
a combination of cross actions and Cayley transforms.

\subsec{Length 1}

In length $1$, these are essentially the usual definitions, except in the {\tt 1i2s} case, when the Cayley transform is not defined.

Suppose $\ell(\kappa)=1$, so $\kappa=\{s_\alpha\}$ and
$w_\kappa=s_\alpha$, 
where $\sigma(\alpha)=\alpha$.

In the classical case  $\alpha$ has type
{\tt C+, C-, i, i2, ic, r, r2} or {\tt rn}.
We write these {\tt 1C+,\dots,1rn} to emphasize the length of
$\kappa$. 

Suppose 
$\alpha$ is of  type {\tt 1i2}, so the Cayley transform is double
valued: $\gamma^\alpha_1,\gamma^\alpha_2$. 
Then $\sigma(\alpha)=\alpha$ implies
$\sigma$ preserves the set $\{\gamma^\alpha_1,\gamma^\alpha_2\}$.
This yields two sub-cases in the new setting:
denote these 
{\tt 1i2f} (``fixed'') or {\tt 1i2s} (``switched''), depending on
whether $\sigma$ acts trivially on this set, or interchanges the two members.

Type {\tt 1r1} is similar; the double-valued Cayley transform is
written $\{\gamma_\alpha^1,\gamma_\alpha^2\}$.

\bigskip
\begin{tabular}{|c|c|c|}
\hline 
type & definition & Cayley transform\\
\hline
{\tt 1C+}&$\alpha$ complex, $\theta\alpha>0$&\\
\hline
{\tt 1C-}&$\alpha$ complex, $\theta\alpha<0$&\\
\hline
{\tt 1i1}&$\alpha$ imaginary, noncompact, type 1&$\gamma^\kappa=\gamma^\alpha$\\
\hline
{\tt 1i2f}&
\begin{tabular}{cc}$\alpha$ imaginary, noncompact, type 2\\
$\sigma$ fixes both terms of $\gamma^\alpha$
\end{tabular}
&$\gamma^\kappa=\gamma^\alpha=\{\gamma^\kappa_1,\gamma^\kappa_2\}$
\\\hline
{\tt 1i2s}&
\begin{tabular}{cc}$\alpha$ imaginary, noncompact, type 2\\ 
$\sigma$ switches the two terms of $\gamma^\alpha$
\end{tabular}
&\\
\hline
{\tt 1ic}&$\alpha$ compact imaginary&\\
\hline
{\tt 1r1f}&
\begin{tabular}{cc}$\alpha$ real, parity, type 1\\ 
$\sigma$ switches the two terms of $\gamma^\alpha$
\end{tabular}
&$\gamma_\kappa=\gamma_\alpha=\{\gamma_\kappa^1,\gamma_\kappa^2\}$
\\
\hline
{\tt 1r1s}&
\begin{tabular}{cc}$\alpha$ real, parity, type 1\\ 
$\sigma$ switches the two terms of $\gamma_\alpha$
\end{tabular}&
\\
\hline
{\tt 1r2}&
$\alpha$ real, parity, type 2&$\gamma_\kappa=\gamma_\alpha$\\
\hline
{\tt 1rn}&
$\alpha$ real, non-parity&\\
\hline
\end{tabular}

\medskip

\subsec{Length 2}

Suppose $\alpha\in S, \beta=\sigma(\alpha)\in S$, and
$\langle\alpha,\ch\beta\rangle=0$.
Let $\kappa=\{s_\alpha,s_\beta\}$, so $w_\kappa=s_\alpha s_\beta\in W^\sigma$. 
It is easy to see that $\alpha,\beta$ have the same type with respect
to $\theta$.
Here are the twelve cases as listed in \cite[Section
7.5]{lv2012b}.

In the length $2$ and $3$ cases we include the terminology from \cite{lv2012b} in a separate column.

\medskip
\hskip-3cm
\begin{tabular}{|c|c|c|c|}
\hline 
type &LV terminology& definition & Cayley transform\\
\hline
{\tt 2C+}&two-complex ascent&
\begin{tabular}{c}
$\alpha,\beta$ complex $\theta\alpha>0$\\ $\theta\alpha\ne\beta$  
\end{tabular}
&\\
\hline
{\tt 2C-}&two-complex ascent&
\begin{tabular}{c}
$\alpha,\beta$ complex $\theta\alpha<0$\\$\theta\alpha\ne\beta$
\end{tabular}&\\
\hline
{\tt 2Ci}\,$^*$&two-semiimaginary ascent&$\alpha,\beta$ complex, $\theta\alpha=\beta$&
$\gamma^\kappa=s_\alpha\times\gamma=s_\beta\times\gamma$\\
\hline
{\tt 2Cr}\,$^*$&two-semireal descent&$\alpha,\beta$ complex,$\theta\alpha=-\beta$&
$\gamma_\kappa=s_\alpha\times\gamma=s_\beta\times\gamma$
\\
\hline
{\tt 2i11}&
\begin{tabular}{c}
two-imaginary noncpt \\type I-I ascent
\end{tabular}&
\begin{tabular}{c}
$\alpha,\beta$ noncpt imaginary, type 1\\
$(\gamma^\alpha)^\beta$ single valued
\end{tabular}&$\gamma^\kappa=(\gamma^\alpha)^\beta$
\\
\hline
{\tt 2i12}\,$^\dagger$&
\begin{tabular}{c}
two-imaginary noncpt \\type I-II ascent
\end{tabular}&
\begin{tabular}{c}
$\alpha,\beta$ noncpt imaginary, type 1\\
$(\gamma^\alpha)^\beta$ double valued
\end{tabular}&
$\gamma^\kappa=\{\gamma^{\kappa}_1,\gamma^{\kappa}_2\}=(\gamma^\alpha)^\beta$\\
\hline
{\tt 2i22}&
\begin{tabular}{c}
two-imaginary noncpt \\type II-II ascent
\end{tabular}&
\begin{tabular}{c}
$\alpha,\beta$ noncpt imaginary, type 1\\
$(\gamma^\alpha)^\beta$ has $4$ values
\end{tabular}&
$\gamma^\kappa=\{\gamma^\kappa_1,\gamma^\kappa_2\}=\{\gamma^{\alpha,\beta}\}^\sigma$\\
\hline
{\tt 2r22}&
\begin{tabular}{c}
two-real \\type II-II descent
\end{tabular}&
\begin{tabular}{c}
$\alpha,\beta$ real, parity, type 2\\
$(\gamma_\alpha)_\beta$ single valued
\end{tabular}&
$\gamma_\kappa=(\gamma_{\alpha})_\beta$
\\
\hline
{\tt 2r21}&
\begin{tabular}{c}
two-real \\type II-I descent
\end{tabular}&
\begin{tabular}{c}
$\alpha,\beta$ real, parity, type 2\\
$(\gamma_\alpha)_\beta$ double valued
\end{tabular}&
$\gamma_\kappa=\{\gamma_{\kappa}^1,\gamma_{\kappa}^2\}=(\gamma_\alpha)_\beta$\\
\hline
{\tt 2r11}&
\begin{tabular}{c}
two-real \\type I-I descent
\end{tabular}&
\begin{tabular}{c}
$\alpha,\beta$ real, parity, type 2\\
$(\gamma_\alpha)_\beta$ has $4$ values
\end{tabular}&
$\gamma_\kappa=\{\gamma_\kappa^1,\gamma_\kappa^2\}=\{(\gamma_{\alpha})_\beta\}^\sigma$\\
\hline
{\tt 2rn}&two-real nonparity ascent&$\alpha,\beta$ real, nonparity&\\
\hline
{\tt 2ic}&two-imaginary compact  descent&$\alpha,\beta$ compact imaginary&\\
\hline
\end{tabular}

\medskip

\hskip-3cm $*$: defect=1 (see Definition \ref{d:defect}).

\hskip-3cm $\dagger$: See Section \ref{s:2i12}

\subsec{Length 3}

Suppose $\alpha\in S, \beta=\sigma(\alpha)\in S$, and
$\langle\alpha,\ch\beta\rangle\ne0$ (equivalently $\pm1$). 
In this case 
$w_\kappa=s_\alpha s_\beta s_\alpha=s_\beta s_\alpha s_\beta\in W^\sigma$. 

Again it is easy to see that $\alpha,\beta$ have the same type with respect
to $\theta$. 
Here are the cases.

\centerline{\bf\large Length 3}

\medskip

\hskip-4cm
\begin{tabular}{|c|c|c|c|}
\hline 
type &LV terminology& definition & Cayley transform\\
\hline
{\tt 3C+}&three-complex ascent&$\alpha,\beta$ complex $\theta\alpha>0$, $\theta\alpha\ne\beta$&\\
\hline
{\tt 3C-}&three-complex descent&$\alpha,\beta$ complex $\theta\alpha<0$, $\theta\alpha\ne\beta$&\\
\hline
{\tt 3Ci}\,$^*$&three-semiimaginary ascent&$\alpha,\beta$ complex, $\theta\alpha=\beta$&
$\gamma^\kappa=(s_\alpha\times\gamma)^\beta\cap(s_\beta\times\gamma)^\alpha$\\
\hline
{\tt 3Cr}\,$^*$&three-semireal descent&$\alpha,\beta$ complex, $\theta\alpha=-\beta$&
$\gamma_\kappa=(s_\alpha\times\gamma)_\beta\cap(s_\beta\times\gamma)_\alpha$\\
\hline
{\tt 3i}\,$^*$&
three imaginary noncompact ascent&
$\alpha,\beta$ noncpt imaginary, type 1
&$\gamma^\kappa=s_\alpha\times\gamma^\beta=s_\beta\times\gamma^\alpha$
\\
\hline
{\tt 3r}\,$^*$&
three-real descent&
$\alpha,\beta$ real, parity, type 2
&$\gamma_\kappa=s_\alpha\times\gamma_\beta=s_\beta\times\gamma_\alpha$
\\
\hline
{\tt 3rn}&three-real non-parity ascent&
$\alpha,\beta$ real, nonparity&\\
\hline
{\tt 3ic}&three-imaginary compact descent&
$\alpha,\beta$ noncompact imaginary&\\
\hline
\end{tabular}

\medskip

\hskip-4cm $*$: defect=1 (see Definition \ref{d:defect}).

p
\bigskip

The type of $\kappa$ depends on a parameter $\gamma\in\D^\sigma$. 
We say $\kappa$ is of a given type with respect to $\gamma$.

\begin{definition}
\label{d:type}
If $\kappa\in\overline S$ and $\gamma\in\D^\sigma$ write
$t_\gamma(\kappa)$ for the type of $\kappa$ with respect to $\gamma$.  
\end{definition}

For the notion of ascent/descent in these tables see Lemma
\ref{l:tauinv}.

\begin{definition}
\label{d:tau}
The $\tau$-invariant of $\gamma\in\D^\sigma$ is

$$\tau(\gamma)=\{\kappa\in\overline S\mid \kappa \text{ is a descent for }\gamma\}.
$$
  
\end{definition}

Here is a list of the 
$10+12+8=30$ types:

\begin{mytable}\
\hfil\break\newline
\label{table:types}
\begin{tabular}{|l|l|l|}
\hline
$\ell(\kappa)$ & ascent $(\kappa\not\in\tau(\gamma))$ & descent $(\kappa\in\tau(\gamma))$ \\
\hline
$1$ & 1C+, 1i1, 1i2f, 1i2s, 1rn& 1C-, 1r1f, 1r1s, 1r2, 1ic\\
\hline
$2$ & 2C+, 2Ci, 2i11, 2i12, 2i22, 2rn & 2C-, 2Cr, 2r11, 2r21, 2r22, 2ic\\
\hline
$3$ & 3C+, 3Ci, 3i, 3rn & 3C--, 3Cr, 3r, 3ic\\
\hline
\end{tabular}
\end{mytable}
\noindent $*$: defect=1 (see Definition \ref{d:defect}).

\sec{Cases {\tt 2i12} and {\tt2r21}}
\label{s:2i12}

Recall we need to address the issue, discussed at the beginning of
Section \ref{s:cayley}, of types {\tt 2i12/2r21}.

Suppose $\gamma\in\D^\sigma$, and $\kappa=\{\alpha,\beta\}$ is of type
{\tt 2i12} with respect to $\gamma$.
Then $s_\alpha\times\gamma=s_\beta\times\gamma$ is also of type {\tt
2i12}. Label these two parameters $\{\gamma_1,\gamma_2\}$. 
Then, on the level of non-extended parameters,
$(\gamma_1)^\kappa=(\gamma_2)^\kappa$ is double-valued, label 
these two parameters $\lambda_1,\lambda_2$. 

Thus we are given two {\it unordered} pairs $\{\gamma_1,\gamma_2\}$ and
$\{\lambda_1,\lambda_2\}$; $\kappa$ is of type {\tt 2i12} and {\tt 2r21}
respectively. 

We want to define the extensions inductively, starting on the
fundamental Cartan. 
Assume that we have already chosen an  extension
$\wh\gamma^+_1$ of $\gamma_1$.

The  Cayley transform of $\wh\gamma_1^+$ by $\kappa$ is a well defined
pair of extended parameters (see the old version of these notes). 
Use this to define the $+$ labelling on the extended parameters for $\lambda_i$:
\begin{subequations}
\renewcommand{\theequation}{\theparentequation)(\alph{equation}}  
\begin{equation}
(\wh\gamma_1^+)^\kappa=\{\wh\lambda_1^+,\wh\lambda_2^+\}
\end{equation}
Now fix an extension $\wh\gamma_2^+$ of $\gamma_2$. 
Then  $(\wh\gamma_2^+)^\kappa$ is either $\{\wh\lambda_1^+,\wh\lambda_2^-\}$ or 
$\{\wh\lambda_1^-,\wh\lambda_2^+\}$. 
(This computation was done in the old notes, in $SL(4,\R)$). 
After switching $\lambda_1,\lambda_2$ if necessary,
we can assume 
\begin{equation}
(\wh\gamma_2^+)^\kappa=\{\wh\lambda_1^+,\wh\lambda_2^-\}
\end{equation}
Alternatively, define $\wh\gamma^+_2$ by the requirement:
$(\wh\gamma_2^+)^\kappa=\{\wh\lambda_1^+,\wh\lambda_2^-\}$
(then $(\wh\gamma_2^-)^\kappa=\{\wh\lambda_1^-,\wh\lambda_2^+\}$).

Clearly the chosen extensions of $\lambda_1,\lambda_2$ depend on the
extensions of $\gamma_1,\gamma_2$, and also the fact that we've chosen 
an order of each pair
$(\gamma_1,\gamma_2)$ and $(\lambda_1,\lambda_2)$. 
For example,
suppose we switch $\gamma_1,\gamma_2$, but keep the same extensions of
these two parameters.  This would induce new definitions of
$\wh\lambda_1^+,\wh\lambda_2^+$: $\wh\lambda_1^+$ wouldn't change, but what we labelled
$\wh\lambda_2^+$ before would now be labelled $\wh\lambda_2^-$. See the table at the end of this section. 

Conclusion: some additional information is needed to determine unique
preferred extensions for $\lambda_1,\lambda_2$. 
\end{subequations}

\noindent{\it What was here before was incorrect, and I don't know 
how to fix it at the moment. So I'm leaving this as a conjecture.}





\begin{conjecture}
\label{c:distinguish}
Assume $\kappa=\{\alpha,\beta\}$, where $\alpha,\beta$ are orthogonal
and interchanged by $\sigma$. 
Suppose $\kappa$ is of type {\tt 2i12} or {\tt 2r21} for parameters $\gamma$ and 
$\gamma'=s_\alpha\times\gamma=s_\beta\times\gamma$.
There is a canonical way to distinguish $\gamma,\gamma'$, and so
write them as an ordered pair $(\gamma_1,\gamma_2)$. 
\end{conjecture}

Assuming this, start with the ordered pair $(\gamma_1,\gamma_2)$, and 
assume we have chosen $\wh\gamma_1^+$. 
Then the Cayley transform $(\gamma_1)^\kappa=(\gamma_2)^\kappa$ is an ordered pair
$(\lambda_1,\lambda_2)$. 
Define
$\wh\lambda_1^+,\wh\lambda_2^+$ by (a): $(\wh\gamma_1^+)^\kappa=\{\wh\lambda_1^+,\lambda_2^+\}$.
Furthermore define $\wh\gamma_2^+$ by the requirement: 
$(\wh\gamma_2^+)^\kappa=\{\wh\lambda_1^+,\wh\lambda_2^-\}$ (exactly one of
the two extensions of $\wh\gamma_2$ satisfy this). 

Clearly the choice of these extensions  depends on the fact that $(\gamma_1,\gamma_2)$ and
$(\lambda_1,\lambda_2)$ are ordered pairs.

There might be an issue of consistency here:
if we've already chosen $\wh\gamma_2^+$, it may conflict with the one
just made. (Similar issues possibly could arise elsewhere.) 
Let's ignore this issue for now, and hope it works.
If so, we have chosen a preferred extension of each parameter, and all
formulas are in terms of this extension. See the Desideratum
\ref{desideratum}.

\bigskip

\subsec{A Table}

It is possible that the ordering we've chosen in the previous section, while natural,
isn't the right one. Hopefully this table will never be needed, but it
shows the affect of different choices.

Assume we've decided on an ordering of $\gamma_1,\gamma_2$, 
and extensions of these two parameters, labelled $+$. This uniquely
determines an ordering of $\lambda_1,\lambda_2$, and  extensions of
these. This is the first row of the table.

The subsequent rows show the affect of the choices. For example,
suppose we keep the same order of $\gamma_1,\gamma_2$, but choose the
other extension of $\gamma_2$. 
Then since
$$
\begin{aligned}
&\wh\gamma_1^+\rightarrow\wh\lambda_1^+,\wh\lambda_2^+\\
&\wh\gamma_2^+\rightarrow\wh\lambda_1^+,\wh\lambda_2^-
\end{aligned}
$$
with our new choices we have
$$
\begin{aligned}
&\wh\gamma_1^+\rightarrow\wh\lambda_1^+,\wh\lambda_2^+\\
&\wh\gamma_2^-\rightarrow\wh\lambda_1^-,\wh\lambda_2^+
\end{aligned}
$$
meaning the sign has changed on the first member of the target pair.
So we should change their order:
$$
\begin{aligned}
&\wh\gamma_1^+\rightarrow\wh\lambda_2^+,\wh\lambda_1^+\\
&\wh\gamma_2^-\rightarrow\wh\lambda_2^+,\wh\lambda_1^-
\end{aligned}
$$
This amounts to switching $\lambda_1,\lambda_2$,
giving the second row of the table.

\bigskip

\begin{tabular}{|c|c||c|c|}
\hline
$\wh\gamma_1^+$&  $\wh\gamma_2^+$&  $\wh\lambda_1^+$&  $\wh\lambda_2^+$\\\hline
\hline
$\wh\gamma_1^+$&  $\wh\gamma_2^-$&  $\wh\lambda_2^+$&  $\wh\lambda_1^+$\\\hline
$\wh\gamma_1^-$&  $\wh\gamma_2^+$&  $\wh\lambda_2^-$&  $\wh\lambda_1^-$\\\hline
$\wh\gamma_1^-$&  $\wh\gamma_2^-$&  $\wh\lambda_1^-$&  $\wh\lambda_2^-$\\\hline
$\wh\gamma_2^+$&  $\wh\gamma_1^+$&  $\wh\lambda_1^+$&  $\wh\lambda_2^-$\\\hline
$\wh\gamma_2^+$&  $\wh\gamma_1^-$&  $\wh\lambda_2^-$&  $\wh\lambda_1^+$\\\hline
$\wh\gamma_2^-$&  $\wh\gamma_1^+$&  $\wh\lambda_2^+$&  $\wh\lambda_1^-$\\\hline
$\wh\gamma_2^-$&  $\wh\gamma_1^-$&  $\wh\lambda_1^-$&  $\wh\lambda_2^+$\\\hline
\end{tabular}

\subsec{The sign $\epsilon(\gamma,\lambda)$}

The sign which arises in the {\tt 2i12/2r21} cases appears frequently, so we need some notation for it. 
We need to refer forward to the definition of $\kappaarrow\gamma\lambda$ (Definition \ref{d:kappaarrow}).

\begin{definition}
\label{d:epsilon}
Suppose $\gamma,\lambda\in\D^\sigma$, $\kappa\in\overline S$, and $\kappaarrow\gamma\lambda$. 

If $t_\gamma(\kappa)\ne${\tt 2r21} define $\epsilon(\gamma,\lambda)=1$. 

Assume $t_\gamma(\kappa)=${\tt 2r21}, so $t_\lambda(\kappa)=${\tt
  2i12}. By the discussion at the beginning of this section,
$\gamma$ and $\lambda$ are members of  ordered pairs $(\gamma_1,\gamma_2)$ and 
$(\lambda_1,\lambda_2)$, respectively. Define:
\begin{equation}
\label{e:epsilon}
\epsilon(\gamma_i,\lambda_j)=
\begin{cases}
-1&i=j=2\\
1&\text{otherwise}
\end{cases}
\end{equation}

\end{definition}

\sec{Formulas for the $\bH$ action on $M$}
\label{s:formulas}

Implicit in the following formulas is the fact that we have chosen an
extension of each parameter as discussed in Section \ref{s:extended}.
For each $\gamma\in \D^\sigma$ we have
chosen an extension $\wh\gamma^+$;
in the following formulas each
$a_{\gamma}$ is really $a_{\wh\gamma^+}$. 

\medskip

\subsec{Length 1}

Suppose  $\sigma(\alpha)=\alpha$, and  $\gamma\in\mathcal D^\sigma$. Then $T_{w_\kappa}(\gamma)$ is given
by the usual formulas, taking the quotients in types {\tt 1i2s,1r1s}
into account. The first column is $t_\gamma(\kappa)$, the type of
$\kappa$ with respect to $\gamma$.

\begin{enumerate}

\item[{\tt1C+}:] $T_{w_\kappa}(a_\gamma)=a_{w_\kappa\times\gamma}$

\item[{\tt 1C-:}] $T_{w_\kappa}(a_\gamma)=(u-1)a_\gamma+ua_{w_\kappa\times\gamma}$

\item[{\tt 1i1:}] $T_{w_\kappa}(a_\gamma)=a_{w_\kappa\times\gamma}+a_{\gamma^\kappa}$

\item[{\tt 1i2f:}] $T_{w_\kappa}(a_\gamma)=a_\gamma+(a_{\gamma^\kappa_1}+a_{\gamma^\kappa_2})$

\item[{\tt 1i2s:}] $T_{w_\kappa}(a_\gamma)=-a_\gamma\qquad$ 

\item[{\tt 1ic:}] $T_{w_\kappa}(a_\gamma)=ua_\gamma$

\item[{\tt 1r1f:}] $T_{w_\kappa}(a_\gamma)=(u-2)a_\gamma+(u-1)(a_{\gamma_\kappa^1} +
  a_{\gamma_\kappa^2})$ 

\item[{\tt 1r1s:}] $T_{w_\kappa}(a_\gamma)=ua_\gamma \qquad$ 

\item[{\tt 1r2:}] $T_{w_\kappa}(a_\gamma)=(u-1)a_\gamma-a_{w_\kappa\times\gamma}+(u-1)a_{\gamma_\kappa}$

\item[{\tt 1rn:}] $T_{w_\kappa}(a_\gamma)=-a_\gamma$
\end{enumerate}

\subsec{Length 2}

\medskip

\begin{enumerate}

\item[{\tt 2C+:}] $T_{w_\kappa}( a_\gamma)=a_{w_\kappa\times\gamma}$

\item[{\tt 2C-:}] $T_{w_\kappa}( a_\gamma)=(u^2-1) a_\gamma+u^2a_{w_\kappa\times \gamma}$

\item[{\tt 2Ci:}] $T_{w_\kappa}( a_\gamma)=u a_\gamma+(u+1)a_{\gamma^\kappa}$ 

\item[{\tt 2Cr:}] $T_{w_\kappa}( a_\gamma)=(u^2-u-1) a_\gamma+(u^2-u)a_{\gamma_\kappa}$

\item[{\tt 2i11:}] $T_{w_\kappa}( a_\gamma)=a_{w_\kappa\times \gamma}+a_{\gamma^\kappa}$

\item[{\tt 2i12:}] $T_{w_\kappa}( a_{\gamma})= a_{\gamma}+\displaystyle\sum_{\lambda|\kappaarrow\lambda\gamma}\epsilon(\lambda,\gamma)a_\lambda$

\item[{\tt 2i22:}] $T_{w_\kappa}( a_\gamma)= a_\gamma+ (a_{\gamma^\kappa_1}+ a_{\gamma^\kappa_2})$
\item[{\tt 2r22:}] $T_{w_\kappa}( a_\gamma)=(u^2-1) a_\gamma-a_{w_\kappa\times \gamma}+(u^2-1)a_{ \gamma_\kappa}$ 

\item[{\tt 2r21:}] $T_{w_\kappa}( a_\gamma)=(u^2-2)a_\gamma+(u^2-1)\displaystyle\sum_{\lambda|\kappaarrow\gamma\lambda}\epsilon(\gamma,\lambda)a_\lambda$
\item[{\tt 2r11:}] $T_{w_\kappa}( a_\gamma)=(u^2-2) a_\gamma+(u^2-1)( a_{\gamma_\kappa^1}+ a_{\gamma_\kappa^2})$

\item[{\tt 2rn:}] $T_{w_\kappa}( a_\gamma)=- a_\gamma$ 

\item[{\tt 2ic:}] $T_{w_\kappa}( a_\gamma)=u^2 a_\gamma$
\end{enumerate}

\begin{remarkplain}
In the {\tt 2i12} case, if $\kappaarrow\lambda\gamma$ (Definition \ref{d:kappaarrow})
recall $\lambda,\gamma$ occur in ordered pairs  $(\lambda_1,\lambda_2)$
and  $(\gamma_1,\gamma_2)$
(Section \ref{s:2i12}). With $\epsilon(\lambda,\gamma)$ given by Definition \ref{d:epsilon}, the 
stated formula in this case is shorthand for:
\begin{enumerate}
\item[] $T_{w_\kappa}( a_{\gamma_1})= a_{\gamma_1}+ (a_{\lambda_1}+ a_{\lambda_2})$ 
\item[] $T_{w_\kappa}( a_{\gamma_2})= a_{\gamma_2}+ (a_{\lambda_1}- a_{\lambda_2})$,
\end{enumerate}
We could state the other formulas using $\epsilon(\gamma,\lambda)$ as well, for example if $t_\gamma(\kappa)=${\tt 2r22}. 
But this doesn't seem worth it.
\end{remarkplain}

\subsec{Length 3}




\begin{enumerate}
\item[{\tt 3C+:}] $T_{w_\kappa}(a_\gamma)=w_\kappa\times a_\gamma$ 
\item[{\tt 3C-:}] $T_{w_\kappa}(a_\gamma)=(u^3-1)a_\gamma+u^3(a_{w_\kappa\times a_\gamma})$ 
\item[{\tt 3Ci:}]  $T_{w_\kappa}(a_\gamma)=ua_\gamma+(u+1)a_{\gamma^\kappa}$
\item[{\tt 3Cr:}]  $T_{w_\kappa}(a_\gamma)=(u^3-u-1)a_\gamma+(u^3-u)a_{\gamma_\kappa}$
\item[{\tt 3i:}]  $T_{w_\kappa}(a_\gamma)=ua_\gamma+(u+1)a_{\gamma^\kappa}$
\item[{\tt 3r:}]  $T_{w_\kappa}(a_\gamma)=(u^3-u-1)a_\gamma+(u^3-u)a_{\gamma_\kappa}$ 
\item[{\tt 3rn:}] $T_{w_\kappa}(a_\gamma)=-a_\gamma$ 
\item[{\tt 3ic:}] $T_{w_\kappa}(a_\gamma)=u^3a_\gamma$
\end{enumerate}

\begin{remarkplain}
In terms of extended parameters, the {\tt 2i12/2r21} cases are
simpler. 
Suppose $\kappa$ is of type {\tt 2i12} with respect to an 
ordinary parameter $\gamma$. 

Suppose $\wh\gamma$ is an extension of $\gamma$.
On the level of extended parameters $\wh\gamma$ has a well defined Cayley
transform 
$(\wh \gamma)^\kappa$, which is an unordered pair of extended
parameters. 
Then:
$$
T_{w_\kappa}(a_{\wh\gamma})= a_{\wh\gamma}+\sum_{\wh\lambda\in(\wh\gamma)^\kappa}a_{\wh\lambda}
$$
If we write $-\wh\gamma$ for the opposite extension then
$(-\wh\gamma)^\kappa$ is the same set of two elements, with the sign
changed on one of them. There are no choices involved here.
\end{remarkplain}

We have the usual notion of the W-graph associated to this Hecke
algebra action.

\begin{definition}
\label{d:kappaarrow}
Suppose $\kappa\in\overline S,\gamma,\lambda\in\D^\sigma$, and
$\kappa\in\tau(\gamma)$. Then we say $\kappaarrow\gamma\lambda$ if
$\kappa\not\in\tau(\lambda)$, and $a_\lambda$ appears in
$T_{w_\kappa}(a_\gamma)$. 
\end{definition}

\sec{Kazhdan-Lusztig-Vogan algorithm}
\label{s:klv}

We use a hybrid notation combining Fokko's notes {\em Implementation of Kazhdan-Lusztig Algorithm},
and \cite{lv2012b}.

Recall $\bH$ is an algebra over $\Z[u,u\inv]$, with generators 
parametrized by $\overline S$. Also $M$ is a $\bH$-module, 
with $\Z[u,u\inv]$-basis $\{a_\gamma\mid\gamma\in\D^\sigma\}$.

Write $\bD$ for the canonical involution of $M$. It satisfies
\begin{equation}
\bD(um)=u\inv\bD(m).
\end{equation}

The order on $\D^\sigma$ is defined in \cite[5.1]{lv2012b}, and length $\ell(\gamma)$
is inherited from $\D$. 

\begin{theorem}[\cite{lv2012b}, Theorem 5.2]
\label{t:Cdelta}
There is a unique basis $\{C_\delta\mid\delta\in \D^\sigma\}$ of $M$ 
satisfying the following conditions. There are polynomials $P^\sigma(\gamma,\delta)\in\Z[u]$ such that
\begin{equation}
C_\delta=\sum_{\gamma}P^\sigma(\gamma,\delta)a_\gamma,
\end{equation}
and:
\begin{enumerate}
\item $\bD(C_\gamma)=u^{-\ell(\gamma)}C_\gamma$;
\item $P^\sigma(\gamma,\delta)\ne 0$ implies $\gamma\le\delta$;
\item $P^\sigma(\gamma,\gamma)=1$
\item $\deg(P^\sigma(\gamma,\delta))\le \frac12(\ell(\delta)-\ell(\gamma)-1)$.
\end{enumerate}
\end{theorem}

Introduce a new variable $v$ satisfying
$v^2=u$, and tensor everything with $\Z[v,v\inv]$, so $\bH$ becomes an
algebra over $\Z[v,v\inv]$. Define
\begin{subequations}
\renewcommand{\theequation}{\theparentequation)(\alph{equation}}  
\begin{equation}
\wh a_\gamma=v^{-\ell(\gamma)}a_\gamma
\end{equation}
and
\begin{equation}
\wh C_\delta=v^{-\ell(\delta)}C_\delta
\end{equation}
\end{subequations}

With this notation Theorem \ref{t:Cdelta} can be written as

\begin{theorem}
\label{t:Chatdelta}
There is a unique basis $\{\wh C_\delta\mid\delta\in \D^\sigma\}$ of $M$ 
satisfying the following conditions. There are polynomials
$\P(\gamma,\delta)(v)\in \Z[v\inv]$ such that
\begin{equation}
\label{e:P}
\wh C_\delta=\sum_{\gamma}\P(\gamma,\delta)(v)\wh a_\gamma,
\end{equation}
and:
\begin{enumerate}
\item $\bD(\wh C_\gamma)=\wh C_\gamma$;
\item $\P(\gamma,\delta)\ne 0$ implies $\gamma\le\delta$;
\item $\P(\gamma,\gamma)=1$;
\item if $\gamma\ne \delta$, then $\P(\gamma,\delta) \in
  v\inv\Z[v\inv]$; and
\item $\deg(\P(\gamma,\delta))(v\inv)\le \ell(\gamma)-\ell(\delta)$
\end{enumerate}
\end{theorem}
It is easy to see that
\begin{equation}
\P(\gamma,\delta)(v)=
v^{\ell(\gamma)-\ell(\delta)}P^\sigma(\gamma,\delta)(v^2).
\end{equation}

Fix $\gamma < \delta$, and suppose 
\begin{subequations}
\renewcommand{\theequation}{\theparentequation)(\alph{equation}}  
\begin{equation}
P^\sigma(\gamma,\delta)=c_0+c_1u+\dots +c_nu^n
\end{equation}
with
\begin{equation}
n= \begin{cases}
(\ell(\delta)-\ell(\gamma)-1)/2&\ell(\delta)-\ell(\gamma)\text{ odd}\\
(\ell(\delta)-\ell(\gamma)-2)/2&\ell(\delta)-\ell(\gamma)\text{ even}\\  
\end{cases}
\end{equation}
\end{subequations}
Then 
\begin{equation}
\label{e:Ppoly}
\P(\gamma,\delta)=
\begin{cases}
c_nv\inv +c_{n-1}v^{-3}+\dots +c_0v^{\ell(\gamma)-\ell(\delta)=-(2n+1)}&\ell(\delta)-\ell(\gamma)\text{ odd}\\  
c_nv^{-2}+c_{n-1}v^{-4}+\dots +c_0v^{\ell(\gamma)-\ell(\delta)=-(2n+2)}&\ell(\delta)-\ell(\gamma)\text{ even}
\end{cases}
\end{equation}
or alternatively
$$
\P(\gamma,\delta)=
\begin{cases}
v\inv[c_n+c_{n-1}v^{-2}+\dots +c_0v^{\ell(\gamma)-\ell(\delta)+1}]&\ell(\gamma)-\ell(\delta)\text{ odd}\\  
v\inv[c_nv^{-1}+c_{n-1}v^{-3}+\dots +c_0v^{\ell(\gamma)-\ell(\delta)+1}]&\ell(\gamma)-\ell(\delta)\text{ even}
\end{cases}
$$


\sec{Action of $T_{w_\kappa}+1$}

An important role is played by the operator $T_{w_\kappa}+1$, which we renormalize.
For $\kappa\in\overline S$ define
\begin{equation}
\wh T_\kappa=v^{-\ell(\kappa)}(T_{w_\kappa}+1)
\end{equation}

Note that Fokko has both $T_s$ and $t_s$. One can deduce they are related 
by 
$t_s=v\inv T_s$. These correspond to our $T_{w_\kappa}$ and $v^{-\ell(w_\kappa)}T_{w_\kappa}$, respectively. 
Also Fokko has an operator $c_s$, which can be seen to be 
$t_s+v\inv=v\inv(T_s+1)$, which is our $\wh T_\kappa$.

\subsection{Formulas for $\wh T_\kappa$}
\label{s:formulasfortkappa}
Here are $30$ formulas, for $\T_\kappa(\a_\gamma)$,
depending on the type of $\kappa$ with respect to $\gamma$.

\medskip

\noindent{\bf Type 1}: $\wh T_\kappa=v\inv(T_{w_\kappa}+1)=v\inv(T_{s_\alpha}+1)$

These are  copied from \cite{implementation}
Section 1.

\medskip

\begin{tabular}{|l|l|}
\hline
$t_\gamma(\kappa)$& $\T_\kappa(\a_\gamma)$\\
\hline
{\tt 1C+}& $v\inv\a_\gamma+\a_{w_\kappa\times\gamma}$\\
\hline
{\tt 1C-}& $v\a_\gamma+\a_{w_\kappa\times\gamma}$\\
\hline
{\tt 1i1}& $v\inv(\a_\gamma+\a_{w_\kappa\times\gamma})+\a_{\gamma^\kappa}$\\
\hline
{\tt 1i2f}& $2v\inv\a_\gamma+(\a_{\gamma^\kappa_1}+\a_{\gamma^\kappa_2})$\\
\hline{\tt 1i2s}& $0$ \\
\hline
{\tt 1r1f}&$(v-v\inv)\a_\gamma+(1-v^{-2})(\a_{\gamma_\kappa^1}+\a_{\gamma_\kappa^2})$\\
\hline
{\tt 1r1s}& $(v+v\inv)\a_\gamma$\\
\hline
{\tt 1r2}& $v\a_\gamma-v\inv\a_{w_\kappa\times\gamma}+(1-v^{-2})\a_{\gamma_\kappa}$\\
\hline
{\tt 1rn}& $0$\\
\hline
{\tt 1ic}& $(v+v\inv)\a_\gamma$\\
\hline
\end{tabular}

\bigskip


\noindent{\bf Type 2}: $\wh T_\kappa=v^{-2}(T_{w_\kappa}+1)=v^{-2}(T_{s_\alpha s_\beta}+1)$

\medskip

\begin{tabular}{|l|l|}
\hline
$t_\gamma(\kappa)$& $\T_\kappa(\a_\gamma)$\\
\hline
{\tt 2C+}& $v^{-2}\a_\gamma+\a_{w_\kappa\times \gamma}$\\
\hline
{\tt 2C-}&$v^2\a_\gamma+\a_{w_\kappa\times\gamma}$\\
\hline
{\tt 2Ci}&$(v+v^{-1})[v^{-1}\a_\gamma+\a_{\gamma^\kappa}]$\\
\hline
{\tt 2Cr}& $(v^2-1)\a_\gamma+(v-v^{-1})\a_{\gamma_\kappa}$\\
\hline
{\tt 2i11}&
$v^{-2}(\a_\gamma+\a_{w_\kappa\times\gamma})+\a_{\gamma^{\kappa}}$\\
\hline
{\tt 2i12}& $2v^{-2}\a_{\gamma}+\displaystyle\sum_{\gamma'|\kappaarrow{\gamma'}\gamma}\epsilon(\gamma',\gamma)\a_{\gamma'}$\\
\hline
{\tt 2i22}& $2v^{-2} \a_\gamma+\a_{\gamma^\kappa_1}+\a_{\gamma^\kappa_2}$\\
\hline
{\tt 2r22}&$v^2\a_\gamma-v^{-2}\a_{w\times\gamma}+(1-v^{-4})\a_{\gamma_\kappa}$ \\
\hline
{\tt 2r21}& $(v^2-v^{-2})\a_\gamma+(1-v^4)
\displaystyle\sum_{\substack{\gamma'|\kappaarrow{\gamma}{\gamma'}}}\epsilon(\gamma,\gamma')\a_{\gamma'}$\\
\hline
{\tt 2r11}& $(v^2-v^{-2})\a_{\gamma}+(1-v^{-4})(\a_{\gamma_\kappa^1}+a_{\gamma_\kappa^2})$\\
\hline
{\tt 2rn}&$0$\\
\hline
{\tt 2ic}& $(v^2+v^{-2})\a_\gamma$\\
\hline
\end{tabular}

\bigskip

\noindent{\bf Type 3}: $\wh T_\kappa=v^{-3}(T_{w_\kappa}+1)=v^{-3}(T_{s_\alpha s_\beta s_\alpha}+1)$

\medskip

\begin{tabular}{|l|l|}
\hline
$t_\gamma(\kappa)$& $\T_\kappa(\a_\gamma)$\\\hline
{\tt 3C+}& $v^{-3}\a_\gamma+\a_{w_\kappa\times\gamma}$\\
\hline
{\tt 3C-}& $v^{3}\a_\gamma+\a_{w_\kappa\times\gamma}$\\
\hline
{\tt 3Ci}&  $(v+v\inv)v^{-2}\a_\gamma+(v+v\inv)\a_{\gamma^\kappa}$\\
\hline
{\tt 3Cr}& $(v^2-v^{-2})v\a_\gamma+(v^2-v^{-2})v\inv\a_{\gamma_\kappa}$\\
\hline
{\tt 3i}&  $(v+v\inv)v^{-2}\a_\gamma+(v+v\inv)\a_{\gamma\kappa}$\\
\hline
{\tt 3r}& $(v^2-v^{-2})v\a_\gamma+(v^2-v^{-2})v\inv\a_{\gamma\kappa}$\\
\hline
{\tt 3rn}& $0$\\
\hline
{\tt 3ic}& $(v^3+v^{-3})\a_\gamma$\\
\hline
\end{tabular}

\subsec{Summary}
\label{s:summary}

We write some of these formulas in a slightly different form 
in the following table of all $\T_\kappa\a_\gamma$.

\newpage

\begin{mytable}
\hfil\break\newline
\label{table:Tagamma}
{\small
\hskip-2cm
\begin{tabular}{|l|l|l|l|}
\hline
$t_\gamma(\kappa)$&$\T_\kappa\a_\gamma$&$t_\gamma(\kappa)$ & $T_\kappa\a_\gamma$\\
\hline\hline
{\tt 1C+}&$[\a_{w_\kappa\times\gamma}+v\inv a_\gamma]$&{\tt 1C-}&$v[\a_\gamma +
v\inv\a_{w_\kappa\times\gamma}]$\\
\hline
{\tt1i1}&$[\a_{\gamma^\kappa}+v\inv(\a_\gamma+\a_{w_\kappa\times\gamma})]$
&{\tt 1r1f}&$(v-v\inv)[\a_\gamma+v\inv(\a_{\gamma^1_\kappa}+\a_{\gamma^2_\kappa})]$\\
\hline
{\tt1i2f}&$[\a_{\gamma^\kappa_1}+v\inv\a_\gamma]
+[\a_{\gamma^\kappa_2}+v\inv\a_\gamma]$
&{\tt 1r2}& \parbox{.3\textwidth}{$v[\a_\gamma +v\inv\a_{\gamma_\kappa}]$ $-
v\inv[\a_{w_\kappa\times\gamma} + v\inv\a_{\gamma_\kappa}]$} \\
\hline
{\tt1i2s}&0&{\tt 1r1s}&$(v+v\inv)[\a_\gamma]$\\
\hline
{\tt1rn}&$0$&{\tt 1ic}&$(v+v\inv)[\a_\gamma]$\\
\hline
{\tt 2C+}&$[\a_{w_\kappa\times\gamma}+v^{-2}\a_\gamma$] &{\tt 2C-}&
$v^2[\a_\gamma+v^{-2}\a_{w_\kappa\times \gamma}]$\\
\hline
{\tt2Ci}&$(v+v\inv)[\a_{\gamma^\kappa}+v\inv\a_\gamma]$ &{\tt 2Cr}&$v(v-v\inv)[\a_\gamma+v\inv\a_{\gamma_\kappa}]$\\
\hline
{\tt2i11}&$[\a_{\gamma^\kappa}+v^{-2}(\a_\gamma+\a_{w_\kappa\times\gamma})]$
&{\tt 2r11}&$(v^2-v^{-2})[\a_{\gamma} + v^{-2}(\a_{\gamma_\kappa^1} +
a_{\gamma_\kappa^2})]$\\ 
\hline
{\tt 2i12}&$\displaystyle\sum_{\gamma'|\kappaarrow{\gamma'}\gamma}\epsilon(\gamma',\gamma)[
\a_{\gamma'}+v^{-2}\displaystyle\sum_{\mu|\kappaarrow{\gamma'}\mu}\epsilon(\gamma',\mu)\a_\mu]$
&{\tt 2r21}&$(v^2-v^{-2})[\a_{\gamma}+v^{-2}\displaystyle\sum_{\gamma'|\kappaarrow{\gamma'}\gamma}\epsilon(\gamma',\gamma)\a_{\gamma'}$
\\
\hline
{\tt2i22}&$[\a_{\gamma^\kappa_1} + v^{-2}\a_\gamma] +
[\a_{\gamma^\kappa_2}+ v^{-2}\a_\gamma]$ &
{\tt 2r22} &\parbox{.3\textwidth}{$v^2[\a_\gamma +v^{-2}\a_{\gamma_\kappa}]$
$-v^{-2}[\a_{w_\kappa\times\gamma}+v^{-2}\a_{\gamma_\kappa}]$} \\
\hline
{\tt2rn}&$0$&{\tt 2ic}&$(v^2+v^{-2})\a_\gamma$\\
\hline
{\tt 3C+}&$[\a_{w_\kappa\times\gamma}+v^{-3}\a_\gamma]$&{\tt3C-}&$v^{3}[\a_\gamma+
v^{-3}\a_{w_\kappa\times\gamma}]$\\
\hline
{\tt3Ci},
{\tt3i}&$(v+v\inv)[\a_{\gamma^\kappa}+v^{-2}\a_\gamma]$ & {\tt3Cr},
{\tt3r} &$v(v^2-v^{-2})[\a_\gamma+v^{-2}\a_{\gamma_\kappa}]$\\ 
%
%
%
\hline
{\tt3rn}&$0$&{\tt3ic}&$(v^3+v^{-3})[\a_\gamma]$\\
\hline
\end{tabular}}
\end{mytable}

We're going to simplify this table - see Table \ref{table:Tagammakappa}.

\begin{remarkplain}
The identity in the {\tt 2i12} case is tricky, let's write it out. 
Suppose $t_\gamma(\kappa)=${\tt 2i12}, and $\gamma$ is one member of
the ordered pair $(\gamma_1,\gamma_2)$. 
Simlarly $\gamma^\kappa$ is an ordered pair $(\gamma'_1,\gamma'_2)$. 

The formula from Section \ref{s:formulasfortkappa} is:
\begin{subequations}
\renewcommand{\theequation}{\theparentequation)(\alph{equation}}  
\begin{equation}
2v^{-2}a_\gamma+\epsilon(\gamma'_1,\gamma)a_{\gamma'_1}+\epsilon(\gamma'_2,\gamma)a_{\gamma'_2}
\end{equation}

whereas Table \ref{table:Tagamma} gives:
$$
\begin{aligned}
&
\epsilon(\gamma'_1,\gamma)[a_{\gamma'_1}+v^{-2}(\epsilon(\gamma'_1,\gamma_1)a_{\gamma_1}
+
\epsilon(\gamma'_1,\gamma_2)a_{\gamma_2})]+\\
&
\epsilon(\gamma'_2,\gamma)[a_{\gamma'_2}+v^{-2}(\epsilon(\gamma'_2,\gamma_1)a_{\gamma_1}
+
\epsilon(\gamma'_2,\gamma_2)a_{\gamma_2})].
\end{aligned}
$$
Using the definition of $\epsilon$ this equals:
\begin{equation}
[a_{\gamma'_1}+v^{-2}(a_{\gamma_1}
+
a_{\gamma_2})]+
\epsilon(\gamma'_2,\gamma)[a_{\gamma'_2}+v^{-2}(a_{\gamma_1}
-a_{\gamma_2})].
\end{equation}
Plugging  $\gamma=\gamma_1$ or  $\gamma_2$ in to (a) and (b) and
comparing confirms the identity.
\end{subequations}

\end{remarkplain}
\sec{Image of $\T_\kappa$}

\subsec{$\T_\kappa(\a_\gamma)$}

\begin{lemma}
\label{l:kappadescents} Fix $\kappa\in\overline S$.

\begin{enumerate} 
\item The image of $\T_\kappa$ is equal to the $(v^{\ell(\kappa)} +
  v^{-\ell(\kappa)})$ eigenspace of $\T_\kappa$.  This is also equal to
  the kernel of $T_\kappa - v^{\ell(\kappa)}$. 
\item
Suppose $\kappa\in\tau(\lambda)$.
For each $\lambda'$ satisfying $\kappaarrow\lambda{\lambda'}$,
the sign $\epsilon(\lambda,\lambda')=\pm1$ (Definition \ref{d:epsilon}) is the unique integer
such that 
\begin{equation}
\label{e:akappalambda}
\a^\kappa_\lambda = \a_\lambda + v^{\ell(\lambda')-\ell(\lambda)}
\sum_{\substack{\lambda'|\kappaarrow\lambda{\lambda'}}}\epsilon(\lambda,\lambda')\a_{\lambda'}
\end{equation}
belongs to the image of $\T_\kappa$.
\item The elements
$$
\{\a_\lambda^\kappa\mid \gamma\in\D^\sigma, \kappa\in\tau(\gamma)\}
$$
form a basis of the image of $\T_\kappa$.
\end{enumerate}
\end{lemma}

Part (1) follows from the quadratic relation \eqref{e:quadratic}. 
Statements (2) and (3) follow from an examination
of Table \ref{table:Tagamma}.
In each entry of the table the terms in square brackets are the $\a_\delta$ 
(not including the $v^{\defect_\delta(\kappa)}$ term). 
This amounts to the fact that we can rewrite 
Table \ref{table:Tagamma} as in Table \ref{table:Tagammakappa}.

Recall  $\epsilon(\delta,\delta')=1$ except in cases {\tt 2i12/2r21}.

Table \ref{table:Tagamma} now simplifies.

\begin{mytable}
\hfil\break\newline
\label{table:Tagammakappa}
\begin{tabular}{|l|l||l|l|}
\hline
$t_\gamma(\kappa)$&$\T_\kappa(\a_\gamma)$&$t_\gamma(\kappa)$ & $\T_\kappa(\a_\gamma)$\\
\hline\hline
{\tt1C+}&$\a^\kappa_{w_\kappa\times\gamma}$&{\tt1C-}&$v\a^\kappa_\gamma$\\
\hline
{\tt1i1}&$\a^\kappa_{\gamma^\kappa}$
&{\tt1r1f}&$(v-v\inv)\a^\kappa_\gamma$\\
\hline
{\tt1i2f}&$\a^\kappa_{\gamma^\kappa_1} +\a^\kappa_{\gamma^\kappa_2}$
&{\tt1r2}& $v\a^\kappa_\gamma- v\inv\a^\kappa_{w_\kappa\times\gamma}$ \\
\hline
{\tt1i2s}&0&{\tt1r1s}&$(v+v\inv)\a^\kappa_\gamma$\\
\hline
{\tt1rn}&$0$&{\tt1ic}&$(v+v\inv)\a^\kappa_\gamma$\\
\hline
{\tt 2C+}&$\a^\kappa_{w_\kappa\times\gamma}$ &{\tt2C-}&
$v^2\a^\kappa_\gamma$\\
\hline
{\tt2Ci}&$(v+v\inv)\a^\kappa_{\gamma^\kappa}$ &{\tt2Cr}&$v(v-v\inv)\a^\kappa_\gamma$\\
\hline
{\tt2i11}&$\a_{\gamma^\kappa}$ &{\tt2r22} &$v^2\a^\kappa_\gamma - v^{-2}\a^\kappa_{w_\kappa\times\gamma}$ \\
\hline
{\tt2i12}&$\displaystyle{\sum_{\kappaarrow{\gamma'}\gamma}}\epsilon(\gamma',\gamma)\a^\kappa_{\gamma'}$
&{\tt2r21}& $(v^2-v^{-2})\a^\kappa_\gamma$\\
\hline
{\tt2i22}&$\a^\kappa_{\gamma^\kappa_1}+\a^\kappa_{\gamma^\kappa_2}$ &
{\tt2r11}&$(v^2-v^{-2})\a^\kappa_{\gamma}$\\ 
\hline
{\tt2rn}&$0$&{\tt2ic}&$(v^2+v^{-2})\a^\kappa_\gamma$\\
\hline
{\tt3C+}&$\a^\kappa_{w_\kappa\times\gamma}$&{\tt3C-}&$v^{3}\a^\kappa_\gamma$\\
\hline
{\tt3Ci},
{\tt3i}&$(v+v\inv)\a^\kappa_{\gamma^\kappa}$ &{\tt3Cr},
{\tt3r} &$v(v^2-v^{-2})\a^\kappa_\gamma$\\ 
%
%
%
\hline
{\tt 3rn}&$0$&{\tt 3ic}&$(v^3+v^{-3})\a^\kappa_\gamma$\\
\hline
\end{tabular}
\end{mytable}

\medskip

Those extra powers of $v$ in cases {\tt 2CR,3Cr,3r} are important.
Suppose $\kappa\in\tau(\lambda)$. 
Then $\ell(\lambda')$ is the same for all $\kappaarrow\lambda{\lambda'}$; typically
(always in the classical case) $\ell(\lambda)-\ell(\lambda')=\ell(\kappa)$.
In general $\ell(\lambda)-\ell(\lambda')\le\ell(\kappa)$.
\begin{definition}
\label{d:defect}
Suppose $\kappaarrow\lambda{\lambda'}$.
Define  the $\kappa$-defect of $\lambda$ and $\lambda'$ to be
\begin{equation}
\label{e:defect}
\defect_\lambda(\kappa)=\defect_\kappa(\lambda')=\ell(\kappa)-\ell(\lambda)+\ell(\lambda').
\end{equation}

(If $\{\lambda'\mid\kappaarrow\lambda{\lambda'}\}=\emptyset$ define $\defect(\kappa,\lambda)=0$). 
\end{definition}
Checking the cases gives:
\begin{lemma}
\label{l:casesofd}
$$
\defect_\lambda(\kappa) =
\begin{cases}
1&t_\lambda(\kappa)={\tt 2Ci,3Ci,3i;\,2Cr,3Cr,3r}\\  
0&\text{else}
\end{cases}
$$
\end{lemma}

We can now write Table \ref{table:Tagammakappa} more concisely.
For  $\kappa\in\tau(\gamma)$ define:
$$
\zeta_\kappa(\gamma)=
\begin{cases}
1&t_\gamma(\kappa)=\text{\tt 1ic,2ic,3ic,1r1s}\\
0&t_\gamma(\kappa)=\text{\tt 1C-,2C-,3C-}\\
-1&\text{otherwise}
\end{cases}
$$



  

\begin{lemma}
\label{l:Tkappaagamma1}
Fix $\kappa\in\overline S$, $\gamma\in\D^\sigma$, and set $d=\defect_\gamma(\kappa)$. 
Then
\begin{equation}
\T_\kappa(\a_\gamma)=
\begin{cases}
\displaystyle(v+v\inv)^{d}\sum_{\kappaarrow{\gamma'}\gamma}\epsilon(\gamma',\gamma)\a^\kappa_{\gamma'}
&\kappa\not\in\tau(\gamma)\\
v^d[v^{\ell(\kappa)-d}\a^\kappa_\gamma+\zeta_\kappa(\gamma)v^{-\ell(\kappa)+d}\a^\kappa_{w_\kappa\times\gamma}]
&\kappa\in\tau(\gamma)
\end{cases}
\end{equation}

\end{lemma}

In the second case:
\begin{enumerate}
\item if $t_\gamma(\kappa)=${\tt 1r2,2r22} then $w_\kappa\times\gamma\ne\gamma$, and $\kappa\in\tau(w_\kappa\times\gamma)$ 
 - there are two terms;
\item  if $t_\gamma(\kappa)=${\tt 1C-,2C-,3C-} then $w_\kappa\times\gamma\ne\gamma$, but $\kappa\not\in\tau(w_\kappa\times\gamma)$ 
- since $\zeta=0$ there is only one term;
\item in all other
cases $w_\kappa\times\gamma=\gamma$ (there is one term with a
coefficient of $v^d(v^{\ell(\kappa)-d}\pm v^{-\ell(\kappa)+d}$).
\end{enumerate}

\subsec{$\T_\kappa(\Chat_\lambda)$ in terms of $\a^\kappa_\gamma$}
\label{s:TChatintermsofa}

We can now compute $\T_\kappa(\Chat_\lambda)$ in  the 
basis of $\a^\kappa_\gamma$ (and the unkown $\P(\gamma,\lambda)$). 

Write $\Chat_\lambda=\sum_{\gamma\le\lambda}\P(\gamma,\lambda)\a_\gamma$,
$\T_\kappa(\Chat_\lambda)=\sum_{\gamma\le\lambda}\P(\gamma,\lambda)\T_\kappa(\a_\gamma)$.
The condition $\gamma\le\lambda$ is superfluous because of the $\P(\gamma,\lambda)$ term.
Apply Lemma \ref{l:Tkappaagamma1}.

$$
\begin{aligned}
\T_\kappa(\Chat_\lambda)&=\sum_{\gamma}P(\gamma,\lambda)\T_\kappa(\a_\gamma)\\
&=\sum_{\gamma\mid\kappa\not\in\tau(\gamma)}P(\gamma,\lambda)\T_\kappa(\a_\gamma)+
\sum_{\gamma\mid\kappa\in\tau(\gamma)}P(\gamma,\lambda)\T_\kappa(\a_\gamma)\\
&=\sum_{\gamma\mid\kappa\not\in\tau(\gamma)}\big[P(\gamma,\lambda)(v+v\inv)^{\defect_\gamma(\kappa)}
\sum_{\gamma'|\kappaarrow{\gamma'}\gamma}\epsilon(\gamma',\gamma)\a^\kappa_{\gamma'}\big]+\\
&\quad\sum_{\gamma\mid\kappa\in\tau(\gamma)}v^{\defect_\gamma(\kappa)}\big[P(\gamma,\lambda)
(v^{\ell(\kappa)-\defect_\gamma(\kappa)}\a^\kappa_\gamma+\zeta_\kappa(\gamma)v^{-\ell(\kappa)+\defect_\gamma(\kappa)}\a^\kappa_{w_\kappa\times\gamma})\big]\\
&=\sum_{\gamma'\mid\kappa\in\tau(\gamma')}\big[(v+v\inv)^{\defect_\kappa(\gamma')}
\sum_{\gamma\mid\kappaarrow{\gamma'}\gamma}P(\gamma,\lambda)\epsilon(\gamma',\gamma)\big]\a^\kappa_{\gamma'}+\\
&\quad
\sum_{\gamma\mid\kappa\in\tau(\gamma)}v^{\defect_\gamma(\kappa)}\big[P(\gamma,\lambda)
v^{\ell(\kappa)-\defect_\gamma(\kappa)}\a^\kappa_{\gamma}+
P(\gamma,\lambda)
\zeta_\kappa(\gamma)v^{-\ell(\kappa)+\defect_\gamma(\kappa)}\a^\kappa_{w_\kappa\times\gamma}\big]\\
\end{aligned}
$$
Interchange $\gamma,\gamma'$ in the first sum to conclude:
$$
\begin{aligned}
\T_\kappa(\Chat_\lambda)&=\sum_{\gamma\mid\kappa\in\tau(\gamma)}\big[(v+v\inv)^{\defect_\gamma(\kappa)}
\sum_{\gamma'\mid\kappaarrow{\gamma}\gamma'}P(\gamma',\lambda)\epsilon(\gamma,\gamma')\big]\a^\kappa_{\gamma}+\\
&\quad
\sum_{\gamma\mid\kappa\in\tau(\gamma)}
v^{\defect_\gamma(\kappa)}\big[v^{\ell(\kappa)-\defect_\gamma(\kappa)}P(\gamma,\lambda)
+
\zeta_\kappa(\gamma)v^{-\ell(\kappa)+\defect_\gamma(\kappa)}
P(w_\kappa\times\gamma,\lambda)\big]
\a^\kappa_{\gamma}
\end{aligned}
$$

\begin{lemma}
\label{l:coefficient}
Fix $\gamma$ with $\kappa\in\tau(\gamma)$. 
  The coefficient of $\a^\kappa_\gamma$ in $\T_\kappa(\Chat_\lambda)$ is
\begin{equation}
\label{e:coefficient}
\begin{aligned}
v^{\defect_\gamma(\kappa)}&\big[v^{\ell(\kappa)-\defect_\gamma(\kappa)}\P(\gamma,\lambda)+
\\&\qquad\zeta_\kappa(\gamma)v^{-\ell(\kappa)+\defect_\gamma(\kappa)}
\P(w_\kappa\times\gamma,\lambda)]+\\
&(v+v\inv)^{\defect_\gamma(\kappa)}
\sum_{\gamma'\mid\kappaarrow{\gamma}\gamma'}\P(\gamma',\lambda)\epsilon(\gamma',\gamma)
\end{aligned}
\end{equation}
\end{lemma}

Here is the information needed to make this explicit. Assume $\kappa\in\tau(\gamma)$.

If $t_\gamma(\kappa)=${\tt 1r2,2r22,1C-,2C-,3C-} then  $w_\kappa\times\gamma\ne\gamma$. In all other cases $w_\kappa\times\gamma=\gamma$. 

We need $\{\gamma'\mid\kappaarrow\gamma{\gamma'}\}$:

\begin{enumerate}
\item $t_\gamma(\kappa)=${\tt 1C-,2C-,3C-}: $w_\kappa\times\gamma$;
\item $t_\gamma(\kappa)=${\tt 1r2,2Cr,2r22,3Cr,3r}: $\gamma_\kappa$ (single valued);
\item $t_\gamma(\kappa)=${\tt 1r1f,2r21,2r11}:  $\{\gamma^1_\kappa,\gamma^2_\kappa\}$ (double valued);
\item $t_\gamma(\kappa)=${\tt 1r1s,1ic,2ic,3ic}: none
\end{enumerate}

The defect $\defect_\gamma(\kappa)$ is $1$ if $t_\gamma(\kappa)=${\tt 2Cr,3Cr,3r}, and $0$ otherwise.

$$
\zeta_\kappa(\gamma)=
\begin{cases}
1&t_\gamma(\kappa)=\text{\tt 1ic,2ic,3ic}\\
0&t_\gamma(\kappa)=\text{\tt 1C-,2C-,3C-}\\
-1&\text{otherwise}
\end{cases}
$$

\begin{mytable}
\hfil\break\newline
\label{table:akappagammainTClambda}
\begin{tabular}{|l|l|l|l|}
\hline
\multicolumn{3}{|c|}{coefficient of $\a^\kappa_\gamma$ in $\T_\kappa(\Chat_\lambda)$}\\\hline
$t_\gamma(\kappa)$ & first term on the RHS of \eqref{e:coefficient}&second term on RHS of \eqref{e:coefficient}\\
\hline
{\tt1C-}&$v\P(\gamma,\lambda)$&$\P(w_\kappa\times\gamma,\lambda)$\\
\hline
{\tt1r1f}&$(v-v\inv)\P(\gamma,\lambda)$&$\P(\gamma_\kappa^1,\lambda)+\P(\gamma_\kappa^2,\lambda)$\\
\hline
{\tt1r1s}&$(v+v\inv)\P(\gamma,\lambda)$&\\
\hline
{\tt1r2}&$v\P(\gamma,\lambda)-v\inv\P(w_\kappa\times\gamma,\lambda)$&$\P(\gamma_\kappa,\lambda)$\\
\hline
{\tt1ic}&$(v+v\inv)\P(\gamma,\lambda)$&\\
\hline
{\tt2C-}&$v^2\P(\gamma,\lambda)$&$\P(w_\kappa\times\gamma,\lambda)$\\
\hline
{\tt2Cr}&$v(v-v\inv)\P(\gamma,\lambda)$&$(v+v\inv)\P(\gamma_\kappa,\lambda)$\\
\hline
{\tt2r22}&$v^2\P(\gamma,\lambda)-v^{-2}\P(w_\kappa\times\gamma,\lambda)$&$\P(\gamma_{\kappa},\lambda)$\\
\hline
{\tt2r21}&$(v^2-v^{-2})\P(\gamma,\lambda)$&$\displaystyle\sum_{\gamma'|\kappaarrow\gamma{\gamma'}}
\epsilon(\gamma,\gamma')\P(\gamma',\lambda)$\\
\hline
{\tt2r11}&$(v^{2}-v^{-2})\P(\gamma,\lambda)$&$\P(\gamma_\kappa^1,\lambda)+\P(\gamma_\kappa^2,\lambda)$\\
\hline
{\tt2ic}&$(v^2+v^{-2})\P(\gamma,\lambda)$&\\
\hline
{\tt3C-}&$v^3\P(\gamma,\lambda)$&$\P(w_\kappa\times\gamma,\lambda)$\\
\hline
{\tt3Cr}&$v(v^2-v^{-2})\P(\gamma,\lambda)$&$(v+v\inv)\P(\gamma_\kappa,\lambda))$\\
\hline
{\tt3r}&$v(v^2-v^{-2})\P(\gamma,\lambda)$&$(v+v\inv)\P(\gamma_\kappa,\lambda)$\\
\hline
{\tt3ic}&$(v^3+v^{-3})\P(\gamma,\lambda)$&\\
\hline
\end{tabular}
\end{mytable}

Here is a condensed version of this table. Let $k=\ell(\kappa)$. 

\begin{mytable}
\hfil\break\newline
\label{table:akappagammainTClambda2}
\begin{tabular}{|l|l|l|l|}
\hline
\multicolumn{3}{|c|}{coefficient of $\a^\kappa_\gamma$ in $\T_\kappa(\Chat_\lambda)$}\\\hline
$t_\gamma(\kappa)$ & first term on the RHS of \eqref{e:coefficient}&second term on RHS of \eqref{e:coefficient}\\
\hline
{\tt 1C-,2C-,3C-}&$v^k\P(\gamma,\lambda)$&$\P(w_\kappa\times\gamma,\lambda)$\\
\hline
{\tt 1ic,2ic,3ic,1r1s}&$(v^k+v^{-k})\P(\gamma,\lambda)$&\\
\hline
{\tt 2Cr,3Cr,3r}&$v(v^{k-1}-v^{-k+1})\P(\gamma,\lambda)$&$(v+v\inv)\P(\gamma_\kappa,\lambda)$\\
\hline
{\tt1r1f,2r11}&$(v^k-v^{-k})\P(\gamma,\lambda)$&$\P(\gamma_\kappa^1,\lambda)+\P(\gamma_\kappa^2,\lambda)$\\
\hline
{\tt1r2,2r22}&$v^k\P(\gamma,\lambda)-v^{-k}\P(w_\kappa\times\gamma,\lambda)$&$\P(\gamma_\kappa,\lambda)$\\
\hline
{\tt2r21}&$(v^2-v^{-2})\P(\gamma,\lambda)$&$\displaystyle\sum_{\gamma'|\kappaarrow\gamma{\gamma'}}
\epsilon(\gamma,\gamma')\P(\gamma',\lambda)$\\\hline
\end{tabular}
\end{mytable}

\subsec{$\T_\kappa(\Chat_\mu)$ in terms of $\Chat_\gamma$}

\begin{lemma}
\label{l:tauinv}
Suppose $\gamma\in\D^\sigma$, $\kappa\in\overline S$. 
Then $\kappa\in \tau(\lambda)$ iff 
\begin{equation*}
\T_\kappa\Chat_\lambda =
(v^{\ell(\kappa)} + v^{-\ell(\kappa)})\Chat_\lambda.
\end{equation*}
\end{lemma}
This is the way ``descent'' is defined. In the geometric language of
\cite{lv2012b}, the condition means that the corresponding perverse
sheaf is pulled back from the partial flag variety of type
$\kappa$.
Compare
 {\cite[Theorem 4.4(c)]{lv2012a}} and \cite[Lemma 6.7]{voganparkcity}.

Recall the image of $\T_\kappa$ has $\{\a^\kappa_\gamma\mid\gamma\in\D^\sigma,\kappa\in\tau(\gamma)\}$ as a basis. 
We can use 
 $\{\Chat_\gamma\mid\gamma\in\D^\sigma,\kappa\in\tau(\gamma)\}$ instead.

\begin{lemma}\label{l:basis}
Fix $\kappa\in\overline S$.
\begin{enumerate}
\item Suppose $\mu\in\D^\sigma$, and  $\kappa\in\tau(\mu)$. Then
\begin{equation}
\label{e:Clambdaakappalambda}
\Chat_\mu=\sum_{\gamma|\kappa\in\tau(\gamma)}
\P(\gamma,\mu)\a^\kappa_\gamma;  
\end{equation}
The coefficient polynomials are exactly the ones from Theorem \ref{t:Chatdelta}.
In particular $\Chat_\mu$ is in the image of $\T_\kappa$.

\item The elements
$$
\{\Chat_\mu \mid \mu \in \D^\sigma, \ \kappa\in \tau(\mu)\}$$
form a basis of the image of $\T_\kappa$. 
\end{enumerate}
\end{lemma}

\begin{proof}
For (1), write
\begin{subequations}
\label{e:easy}
\renewcommand{\theequation}{\theparentequation)(\alph{equation}}  
\begin{equation}
\Chat_\mu=\sum_{\gamma|\kappa\in\tau(\gamma)}\P(\gamma,\mu)\a_\gamma+
\sum_{\gamma|\kappa\not\in\tau(\gamma)}\P(\gamma,\mu)\a_\gamma.
\end{equation}
By Lemmas \ref{l:tauinv} and \ref{l:kappadescents}(3)
we can also write
$\Chat_\mu=\displaystyle\sum_{\gamma|\kappa\in\tau(\gamma)}
\wh{\mathcal R}^\sigma(\gamma,\mu)\a^\kappa_\gamma$
for some $\wh{\mathcal R}^\sigma(\gamma,\mu)\in\Z[v,v\inv]$. 
Plugging in the definition of $\a_\gamma$ 
(Lemma  \ref{l:kappadescents}(2)) gives
\begin{equation}
\label{e:easy1}
\Chat_\mu=\sum_{\gamma|\kappa\in\tau(\gamma)}
\wh{\mathcal R}^\sigma(\gamma,\mu)[\a_\gamma+
v^{\ell(\gamma')-\ell(\gamma)}
\sum_{\gamma'|\kappaarrow{\gamma}{\gamma'}}\epsilon(\gamma,\gamma')\a_{\gamma'}].
\end{equation}
Since $\kappa\not\in\tau(\gamma')$ for each term in the last sum,
comparing coefficients of $\a_\gamma$ ($\kappa\in\tau(\gamma)$) in (a) and (b)  gives 
$\wh{\mathcal R}(\gamma,\mu)=\P(\gamma,\mu)$. This gives (1), and (2) follows.
\end{subequations}
\end{proof}

Comparing the coefficients of $\a_\gamma$ with
$\kappa\not\in\tau(\gamma)$ 
gives the ``easy'' recurrence relations for the $\P(\gamma,\lambda)$.
See  Section \ref{s:recursion}.

Next we want to compute $\T_\kappa\Chat_\lambda$ in the basis
$\Chat_\gamma$. When $\kappa\in\tau(\lambda)$ this is given in Lemma
\ref{l:tauinv}. 
We turn now to the case $\kappa\not\in\tau(\lambda)$.


\begin{lemma}[compare {\cite[Theorem 4.4(a,b)]{lv2012a}}]
\label{l:sumtau}
Suppose $\kappa\not\in\tau(\lambda)$. Then
\begin{equation}
\label{e:sumtau}
\T_\kappa\Chat_\lambda=\sum_{\gamma|\kappa\in\tau(\gamma)}m_\kappa(\gamma,\lambda)\Chat_\gamma
\end{equation}
for some $m_\kappa(\gamma,\lambda)\in\Z[v,v\inv]$. 
Each $m_\kappa(\gamma,\lambda)$ is self-dual, and is of the form

$$
m_\kappa(\gamma,\lambda)=
\begin{cases}
m_{\kappa,0}(\gamma,\lambda)&\ell(\kappa)=1\\  
m_{\kappa,0}(\gamma,\lambda)+m_{\kappa,1}(\gamma,\lambda)(v+v\inv)&\ell(\kappa)=2\\  
m_{\kappa,0}(\gamma,\lambda)+m_{\kappa,1}(\gamma,\lambda)(v+v\inv)+m_{\kappa,2}(\gamma,\lambda)(v^2+v^{-2})&\ell(\kappa)=3
\end{cases}
$$
for some integers $m_{\kappa,i}(\gamma,\lambda)$.

\end{lemma}

\begin{proof} 
The existence of $m_\kappa(\gamma,\delta)$ is (2) of Lemma \ref{l:basis}.
That the left side of \eqref{e:sumtau} is self-dual is
\cite[4.8(e)]{lv2012a}, and since $\Chat_\delta$ is self-dual this implies $m_\kappa(\gamma,\delta)$ is self-dual. 

The highest order term of $m_\kappa(\gamma,\lambda)$ is $v^{\ell(\kappa)-1}$. 
This follows by downward induction on $\ell(\gamma)$.
See \cite[pg. 17 (Proof of Theorem 4.4)]{lv2012a}.
This gives the remaining assertion.
\end{proof}

\begin{remarkplain}
It is easy to see $m_\kappa(\gamma,\lambda)\ne 0$ implies
$\kappaarrow\gamma\lambda$, or
$\gamma<\lambda$, or  $\kappaarrow\gamma{\gamma'}$ for some 
$\gamma'<\lambda$. We make this more precise in Lemmas \ref{l:mnotzero} and \ref{l:kappalesslength}. 
\end{remarkplain}


In the classical setting  $\mu(\gamma,\lambda)$ is defined to be the coefficient of 
the top degree term in $P(\gamma,\lambda)$, i.e. 
$q^{\frac12(\ell(\lambda)-\ell(\gamma)-1)}$. 
Furthermore if $\gamma<\lambda$ then 
$m_\kappa(\gamma,\lambda)=\mu(\gamma,\lambda)$.
 
With our normalization  the top degree term in $\P(\gamma,\lambda)$
is $v\inv$, which is zero unless
 $\ell(\lambda)-\ell(\gamma)$ is odd. We need a generalization to take $\kappa$ of length $2,3$ into account. 

\begin{definition}
For $i=-1,-2,-3$ let 
 $\Mu_i(\gamma,\lambda)$ be the coefficient of    $v^i$ in $\P(\gamma,\lambda)$. 
\end{definition}

So
$$
\P(\gamma,\lambda)=
\Mu_{-3}(\gamma,\lambda)v^{-3}+
\Mu_{-2}(\gamma,\lambda)v^{-2}+
\Mu_{-1}(\gamma,\lambda)v^{-1}\mypmod{v^{-4}}.
$$
It is clear that
\begin{equation}
\label{e:lambdazero}
\Mu_{-k}(\gamma,\lambda)=0\quad\text{ unless }\ell(\gamma)-\ell(\lambda)=k\mypmod2.
\end{equation}



We can now state the main result of this section.

\begin{theorem}
\label{t:T_kappa}
Suppose $\kappa\not\in\tau(\lambda)$. Then
\begin{equation*}
\T_\kappa\Chat_\lambda=
\sum_{\gamma|\kappa\in\tau(\gamma)}m_\kappa(\gamma,\lambda)\Chat_\gamma.
\end{equation*}
for coefficients 
 $m_\kappa(\gamma,\lambda)$ given as follows.

\begin{enumerate}
\item  If $\kappaarrow\gamma\lambda$ then
\begin{equation}
m_\kappa(\gamma,\lambda)=(v+v\inv)^{\defect_\kappa(\gamma)}\epsilon(\gamma,\lambda)=
\begin{cases}
\epsilon(\gamma,\lambda)&\defect_\lambda(\kappa)=0\\
(v+v\inv)&\defect_\lambda(\kappa)=1
\end{cases}
\end{equation}
(recall  $\epsilon(\lambda,\gamma)=\pm1$, and is $1$
unless $t_\gamma(\kappa)=${\tt 2r21}, cf. Definition \ref{d:epsilon}). 

\item Assume $\gamma\overset\kappa{\not\rightarrow}\lambda$, and $\ell(\kappa)=1$. 
Then 
$$
m_\kappa(\gamma,\lambda)=\Mu_{-1}(\gamma,\lambda).
$$
\item Assume $\gamma\overset\kappa{\not\rightarrow}\lambda$, and $\ell(\kappa)=2$. 
\begin{enumerate}
\item If $\ell(\gamma)\not\equiv\ell(\lambda)\mypmod2$ then
$$
m_\kappa(\gamma,\lambda)=\Mu_{-1}(\gamma,\lambda)(v+v\inv).
$$
\item If  $\ell(\gamma)\equiv\ell(\lambda)\mypmod2$ then
$$
\begin{aligned}
m_\kappa(\gamma,\lambda)&=
\Mu_{-2}(\gamma,\lambda)-
\sum_{\substack{\delta\\\kappa\in\tau(\delta)\\\gamma<\delta<\lambda}}
\Mu_{-1}(\gamma,\delta)\Mu_{-1}(\delta,\lambda)
\\
&-
\left\{\begin{aligned} 
&\Mu_{-1}(\gamma,\lambda^\kappa)\ &&t_\lambda(\kappa)=2Ci\\
&0\ &&\text{else}\end{aligned}\right\}
+
\left\{\begin{aligned} 
&\Mu_{-1}(\gamma_\kappa,\lambda)\ &&t_\gamma(\kappa)=2Cr\\
&0\ &&\text{else} \end{aligned}\right\} 
\end{aligned}
$$
\end{enumerate}

\item Assume $\gamma\overset\kappa{\not\rightarrow}\lambda$, and $\ell(\kappa)=3$. 

  \begin{enumerate}
  \item If $\ell(\gamma)\equiv\ell(\lambda)\mypmod2$ then
$$
m_\kappa(\gamma,\lambda)=
\big[\Mu_{-2}(\gamma,\lambda)-\sum_{\substack{\delta\\\kappa\in\tau(\delta)\\\gamma<\delta<\lambda\\}}
 \Mu_{-1}(\gamma,\delta)\Mu_{-1}(\delta,\lambda)\big](v+v\inv)
$$
\item If $\ell(\gamma)\not\equiv\ell(\lambda)\mypmod2$ then
$$
\begin{aligned}
m_\kappa(\gamma,\lambda)&=\Mu_{-1}(\gamma,\lambda)(v^2+v^{-2})+\Mu_{-3}(\gamma,\lambda)\\
&+\sum_{\substack{\delta,\phi\\\kappa\in\tau(\delta),\kappa\in\tau(\phi)\\\gamma<\delta<\phi<\lambda\\}} 
\Mu_{-1}(\gamma,\delta)
\Mu_{-1}(\delta,\phi)\Mu_{-1}(\phi,\lambda)+\\
&-
\sum_{\substack{\delta\\\kappa\in\tau(\delta)\\\gamma<\delta<\lambda\\}} 
\big[
\Mu_{-1}(\gamma,\delta)\Mu_{-2}(\delta,\lambda)+
\Mu_{-2}(\gamma,\delta)\Mu_{-1}(\delta,\lambda)\big]\\
\\
&-\begin{cases}
\Mu_{-1}(\gamma,\lambda^\kappa)&t_\lambda(\kappa)=3Ci\text{ or }3i\\
0&else
\end{cases}\\
&
+
\begin{cases}
\Mu_{-1}(\gamma_\kappa,\lambda)&t_\gamma(\kappa)=3Cr\text{ or }3r\\
0&\text{else}
\end{cases}
\end{aligned}
$$
\end{enumerate}
\end{enumerate}
\end{theorem}

I've written the proof
in great length for the sake of finding errors. See the Appendix.

Theorem \ref{t:T_kappa} gives us a basic identity which we use repeatedly.

Suppose $\kappa\not\in\tau(\lambda)$. By Theorem \ref{t:T_kappa}
and \eqref{e:P}:
\begin{equation}
\begin{aligned}
\T_\kappa\Chat_\lambda&=
\sum_{\delta|\kappa\in\tau(\delta)}
m_\kappa(\delta,\lambda)\Chat_\delta\\
&=\sum_{\delta|\kappa\in\tau(\delta)}
m_\kappa(\delta,\lambda)\sum_{\gamma}\P(\gamma,\lambda)\a_\gamma\\
&=\sum_\gamma\big[\sum_{\delta|\kappa\in\tau(\delta)}\P(\gamma,\delta)m_\kappa(\delta,\lambda)\big]\a_\gamma
\end{aligned}
\end{equation}

\begin{proposition}
\label{p:basicidentity}
Fix $\kappa\in\overline S$,  $\gamma,\lambda\in\D^\sigma$, with $\kappa\not\in\tau(\lambda)$. Then
\begin{equation}
\label{e:basicidentity}
\sum_{\delta|\kappa\in\tau(\delta)}\P(\gamma,\delta)m_\kappa(\delta,\lambda)=
\text{multiplicity of $\a_\gamma$ in }\T_\kappa(\Chat_\lambda)
\end{equation}
If $\kappa\in\tau(\gamma)$ the same equality holds with $\a^\kappa_\gamma$ on the right hand side.
\end{proposition}

\subsec{Nonvanishing of $m_\kappa(\gamma,\lambda)$}

It is important to know when $m_\kappa(\gamma,\lambda)$ can be nonzero.

\begin{lemma}
\label{l:mnotzero}
Assume $\kappa\in\tau(\gamma),\kappa\not\in\tau(\lambda)$, and $m_\kappa(\gamma,\lambda)\ne 0$. Then
one of the following  conditions  holds:
\begin{enumerate}
  \item[(a)] $\kappaarrow\gamma\lambda$
\item[(b)] $\gamma<\lambda$
\item[(c)] $\defect_\gamma(\kappa)=1$, $\gamma\not<\lambda$, and
  $\Mu_{-1}(\gamma_\kappa,\lambda)\ne 0$.
\item[(c$'$)] $\defect_\lambda(\kappa)=1$, $\gamma\not<\lambda$, and
  $\Mu_{-1}(\gamma,\lambda^\kappa)\ne 0$.
\end{enumerate}
\end{lemma}
See Definition \ref{d:defect} for  $\defect_\gamma(\kappa)$.
Compare \cite[Lemma 6.7]{voganparkcity}, and \cite[Section 3,II]{implementation}.

\begin{remarkplain}
In the classical case either $\kappaarrow\gamma\lambda$, or
$m_\kappa(\gamma,\lambda)=\mu(\gamma,\lambda)$ (the top degree term of
$P(\gamma,\lambda)$), which is nonzero only if $\gamma<\lambda$. 
So cases (c), (c$'$) don't occur. Since they allow
$m_\kappa(\gamma,\lambda)\ne 0$ for some $\gamma\not<\lambda$, these
cause some trouble.  
\end{remarkplain}

\begin{proof}
Consulting the cases in the Theorem, if  $m_\kappa(\gamma,\lambda)\ne 0$ 
then either:
\begin{enumerate}
\item $\kappaarrow\gamma\lambda$ (Case (1) of the Theorem)
\item Some $\Mu_{-k}(\gamma,\delta)\ne0$ with $\delta\le\lambda$. This implies $\gamma<\lambda$ ($\gamma\ne\lambda$ since they have opposite $\tau$-invariants).
\item One of the terms in braces in Cases (3b) or (4b) is nonzero.
\end{enumerate}
The cases {\tt 2Cr, 3Cr, 3r, 2Ci, 3Ci, 3i} are exactly the ones in which the defect is $1$,
and since 
$\Mu_{-1}(\gamma,\lambda^\kappa)\ne0$ or 
$\Mu_{-1}(\gamma_\kappa,\lambda)\ne0$  this gives the result.
\end{proof}

\begin{definition}
\label{d:kappaless}
Suppose $\kappa\in\tau(\gamma),\kappa\not\in\tau(\lambda)$. 
We say $\gamma\kappaless\lambda$ if one of conditions (b,c,c$'$)  of the Lemma hold:
\begin{enumerate}
\item[(b)] $\gamma<\lambda$
\item[(c)] $\defect_\gamma(\kappa)=1$, $\gamma\not<\lambda$, and
  $\Mu_{-1}(\gamma_\kappa,\lambda)\ne 0$.
\item[(c$'$)] $\defect_\lambda(\kappa)=1$, $\gamma\not<\lambda$, and
  $\Mu_{-1}(\gamma,\lambda^\kappa)\ne 0$.
\end{enumerate}
\end{definition}
Thus \eqref{e:sumtau} becomes
\begin{equation}
\label{e:sumtau1}
\T_\kappa\Chat_\lambda=
\sum_{\substack{\gamma|\kappa\in\tau(\gamma)\\\kappaarrow\gamma\lambda}}
m_\kappa(\gamma,\lambda)\Chat_\gamma+
\sum_{\substack{\gamma|\kappa\in\tau(\gamma)\\\gamma\kappaless\lambda}}
m_\kappa(\gamma,\lambda)\Chat_\gamma
\end{equation}

We want to replace (b),(c),(c$'$) with (weaker) conditions in terms of
length.
Obviously (b) implies $\ell(\gamma)<\ell(\lambda)$. Suppose
$\ell(\gamma)\ge\ell(\lambda)$,
and (c) or (c$'$) holds. This is quite rare.

Consider Case (c).
We're assuming $\ell(\gamma_\kappa)<\ell(\lambda)\le\ell(\gamma)$.
It is hard to satisfy this. 
Subtract $\ell(\gamma_\kappa)$ from each term,
and use 
$\ell(\gamma_\kappa)=\ell(\gamma)-\ell(\kappa)+1$, to see 
$$
0<\ell(\lambda)-\ell(\gamma_\kappa)\le\ell(\kappa)-1\in\{1,2\}
$$
But $\Mu_{-1}(\gamma_\kappa,\lambda)\ne 0$ implies
$\ell(\lambda)-\ell(\gamma_\kappa)$ is odd, so it equals $1$, and 
$$
\ell(\gamma)=\ell(\lambda)+\ell(\kappa)-2,\, \ell(\gamma_\kappa)=\ell(\lambda)-1.
$$
Case (c$'$) is similar: $t_\lambda(\kappa)=${\tt 2Cr,3Cr,3r},
$\ell(\gamma)=\ell(\lambda)+\ell(\kappa)-2$, and 
$\ell(\lambda^\kappa)=\ell(\gamma)+1$.

These are  illustrated by the following pictures.
An arrow with a  label: $\alpha\underset{k}\rightarrow\beta$ indicates $k=\ell(\alpha)-\ell(\beta)$.
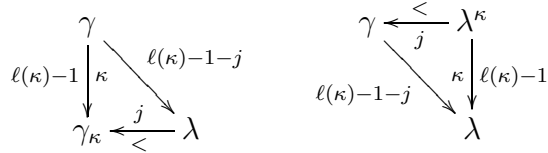
$$
\xymatrix{
&\gamma\ar[rd]^{\ell(\kappa)-1-j}\ar[d]_{\ell(\kappa)-1}^\kappa&&&
\gamma\ar[rd]_{\ell(\kappa)-1-j}&\lambda^\kappa\ar[l]_{<}^{j}\ar[d]^{\ell(\kappa)-1}_\kappa\\
&\gamma_\kappa&\lambda\ar[l]^{<}_j&&&\lambda
}
$$
The preceding argument shows that $j=1$ in both cases.

This gives a nonvanishing criterion in terms of length.

\begin{lemma}
\label{l:kappalesslength}
Assume $\kappa\in\tau(\gamma),\kappa\not\in\tau(\lambda)$, and
$\gamma\kappaless\lambda$. Then
one of the following  conditions  holds:
\begin{enumerate}
\item[(b)] $\ell(\gamma)<\ell(\lambda)$
\item[(c)]  $\ell(\gamma)=\ell(\lambda)+\ell(\kappa)-2$, $\defect_\gamma(\kappa)=1$
\item[(c$'$)] $\ell(\gamma)=\ell(\lambda)+\ell(\kappa)-2$, $\defect_\lambda(\kappa)=1$
\end{enumerate}
\end{lemma}

\begin{remarkplain}
\label{r:rare}
Explicitly, Cases (c) and (c$'$) of the Lemma are:
\begin{enumerate}
\item $\ell(\kappa)=2$,  $t_\gamma(\kappa)=${\tt 2Cr} or  $t_\lambda(\kappa)=${\tt 2Ci}, and $\ell(\gamma)=\ell(\lambda)$;
\item $\ell(\kappa)=3$,  $t_\gamma(\kappa)=${\tt 3Cr,3r} or  $t_\lambda(\kappa)=${\tt 3Ci,3i}, and $\ell(\gamma)=\ell(\lambda)+1$.
\end{enumerate}

\begin{danger}Cases (a,b,c,c$'$) in Lemmas \ref{l:mnotzero} and \ref{l:kappalesslength} don't precisely line up. 
There can be $\gamma$ in case (c) or (c$'$) of Lemma  \ref{l:mnotzero}, so $\gamma\not<\lambda$, but 
$\ell(\gamma)<\ell(\lambda)$, putting it in case (b) of Lemma \ref{l:kappalesslength}.
\end{danger}

\end{remarkplain}

\sec{Computing $\P(\gamma,\mu)$}

\label{s:recursion}

\subsec{Easy recurrence relations}
\label{s:easy}

Recall \eqref{e:easy}(a) and (b)
$$
\Chat_\mu=\sum_{\gamma|\kappa\in\tau(\gamma)}\P(\gamma,\mu)\a_\gamma+
\sum_{\gamma|\kappa\not\in\tau(\gamma)}\P(\gamma,\mu)\a_\gamma.
$$
and
$$
\Chat_\mu=
\sum_{\gamma|\kappa\in\tau(\gamma)}
\P(\gamma,\mu)\a_\gamma+
\sum_{\gamma'|\kappa\not\in\tau(\gamma')}
\big[\sum_{\gamma|\kappaarrow{\gamma}{\gamma'}}
v^{\ell(\gamma')-\ell(\gamma)}\P(\gamma,\mu)
\epsilon(\gamma,\gamma')\big]\a_{\gamma'}.
$$
Equate the coefficients of $\a_\gamma$ ($\kappa\not\in\tau(\gamma)$),
and use \eqref{e:defect}  to conclude the ``easy'' relations:

\begin{lemma}
\label{l:easyrecursion}
Suppose $\kappa\not\in\tau(\gamma),\kappa\in\tau(\mu)$. Then
\begin{equation}
\label{e:easyrecursion}
\boxed{\P(\gamma,\mu)=v^{-\ell(\kappa)+\defect_\gamma(\kappa)}\sum_{\gamma'|\kappaarrow{\gamma'}\gamma}\epsilon(\gamma',\gamma)\P(\gamma',\mu)}
\end{equation}
\end{lemma}

We will compute $\P(\gamma,\mu)$ (and $m_\kappa(\gamma,\mu)$)
by induction on length as follows.
To compute $\P(\gamma,\mu)$ we may assume we know:
\begin{equation}
\label{e:inductive}
\begin{aligned}
\P(*,\mu')&\text{ if }\ell(\mu')<\ell(\mu)\\
\P(\gamma',\mu)&\text{ if }\ell(\gamma')>\ell(\gamma).
\end{aligned}
\end{equation}

We know the right hand side of \eqref{e:easyrecursion}  by the inductive assumption. 
If there is only one term on the right hand side $\P(\gamma,\mu)$ is equal (up to a power of $v$) to a polynomial we have already computed. 
Otherwise
 $\P(\gamma,\mu)$ is the sum of two terms.

\begin{definition} 
Suppose $\gamma<\mu$. Then $(\gamma,\mu)$ is:
\begin{itemize}
\item extremal if $\kappa\in\tau(\mu)\Rightarrow\kappa\in\tau(\gamma)$
\item primitive if 
$\kappa\in\tau(\mu)\Rightarrow\kappa\in\tau(\gamma)$ or  $\kappa\not\in\tau(\gamma)$, $|\{\gamma'\mid\kappaarrow{\gamma'}{\gamma}\}|=2$.
\end{itemize}
\end{definition}

I find the converses  more natural:
\begin{definition}
Suppose $\gamma<\mu$. Then $(\gamma,\mu)$ is:  

\begin{itemize}
\item non-extremal if there exists $\kappa\in\tau(\mu),\kappa\not\in\tau(\gamma)$.
\item non-primitive if there exists $\kappa\in\tau(\mu),\kappa\not\in\tau(\gamma)$ and
$|\{\gamma'\mid\kappaarrow{\gamma'}{\gamma}\}|<2$.
\end{itemize}
Explicitly $(\gamma,\mu)$ is:
\begin{itemize}
\item non-primitive if there exists $\kappa\in\tau(\mu),\kappa\not\in\tau(\gamma)$ and
$t_\gamma(\kappa)\ne${\tt 1i2f,2i12}.
\end{itemize}
\end{definition}

Thus {\it extremal} $\subset$ {\it primitive} 
and
{\it non-primitive}  $\subset$  {\it non-extremal}.

\bigskip

If $(\gamma,\mu)$ is non-primitive,  \eqref{e:easyrecursion} writes $P(\gamma,\mu)=v^k P(\gamma',\mu)$. 

\subsec{Direct Recursion Relations}

Recall Proposition \ref{p:basicidentity}:

\begin{subequations}
\renewcommand{\theequation}{\theparentequation)(\alph{equation}}  
\label{e:direct}
\begin{equation}
\sum_{\delta|\kappa\in\tau(\delta)}\P(\gamma,\delta)m_\kappa(\delta,\lambda)=
\text{multiplicity of $\a_\gamma$ in }\T_\kappa(\Chat_\lambda),
\end{equation}
and if $\kappa\in\tau(\gamma)$ the same equality holds with $\a^\kappa_\gamma$ on the right hand side.

By \eqref{e:sumtau1} the left hand side is:

\begin{equation}
\sum_{\substack{\delta|\kappa\in\tau(\delta)\\\kappaarrow\delta\lambda}}
\P(\gamma,\delta)m_\kappa(\delta,\lambda)
+
\sum_{\substack{\delta|\kappa\in\tau(\delta)\\\delta\kappaless\lambda}}
\P(\gamma,\delta)m_\kappa(\delta,\lambda)
\end{equation}
\end{subequations}

We introduce some notation for the final sum.
See \cite[after Lemma 6.7]{voganparkcity}.
\begin{definition}
\label{d:U}
For $\kappa\not\in\tau(\lambda)$ define:
$$
\U(\gamma,\lambda)=\sum_{\substack{\delta|\kappa\in\tau(\delta)\\\delta\kappaless\lambda}}
\P(\gamma,\delta)
m_\kappa(\delta,\lambda).
$$
\end{definition}

By \eqref{e:direct}(a) and (b) we see:
\begin{lemma}
\label{l:murecursion1}
Fix $\gamma,\lambda$, with  $\kappa\not\in\tau(\lambda)$. Then
\begin{equation}
\label{e:recursion1}
\sum_{\delta|\kappaarrow{\delta}\lambda}
\P(\gamma,\delta)m_\kappa(\delta,\lambda)
=[\text{coefficient of $\a_\gamma$ in $\T_\kappa\Chat_\lambda$}]-
\U(\gamma,\lambda).
\end{equation}
If $\kappa\in\tau(\gamma)$ we can replace the term in brackets with
$$
\text{coefficient of $\a^\kappa_\gamma$ in $\T_\kappa\Chat_\lambda$}
$$
\end{lemma}

We first dispense with a case which won't be used until Section \ref{s:new}.
In the setting of the Lemma, if $t_\lambda(\kappa)=${\tt 1i2s,1ic,2ic,3ic} then,
even though $\kappa\in\tau(\lambda)$, there are no $\delta$ occuring in the sum on the left hand side.

\begin{lemma}
\label{l:lhsempty}
Assume $t_\lambda(\kappa)=$ {\tt 1i2s,1rn,2rn,3rn}. Equivalently  $\kappa\not\in\tau(\lambda)$
but there does not exist $\delta$ satisfying $\kappaarrow{\delta}{\lambda}$. 
Then
$$
\text{coefficient of $\a_\gamma$ in $\T_\kappa\Chat_\lambda$}=
\U(\gamma,\lambda)
$$
\end{lemma}

We turn now to the main case, in which the left hand side of \eqref{e:recursion1} is nonempty. 
In this case this sum has $1$ or $2$
terms. We're mainly interested when it has $1$ term, in which case it
gives a formula for $\P(\gamma,\mu)$. For this reason it is
convenient to change variables. This gives the main result.

\begin{proposition}
\label{p:recursion}
Suppose $\kappa\in\tau(\gamma)$, $\kappa\in\tau(\mu)$, and 
$t_\mu(\kappa)\ne${\tt 1r1s, 1ic, 2ic, 3ic}. Choose $\lambda$ satisfying
$\kappaarrow{\mu}\lambda$. Then 
\begin{equation}
\label{e:recursion2}
\boxed{\sum_{\mu'|\kappaarrow{\mu'}\lambda}
\P(\gamma,\mu')m_\kappa(\mu',\lambda)
=[\text{coefficient of $\a^\kappa_\gamma$ in $\T_\kappa\Chat_\lambda$}]-
\U(\gamma,\lambda)}
\end{equation}
The sum on the left hand side is over $\{\mu,w_\kappa\times\mu\}$ if
$t_\mu(\kappa)=${\tt 1r2, 2r22, 2r21},
and just $\mu$ otherwise.
\end{proposition}

This is our main recursion relation. We analyse its effectiveness in
the next section.

Since we are assuming $\kappa\in\tau(\gamma)$ we have replaced $\a_\gamma$ with $\a^\kappa_\gamma$ as in Lemma \ref{l:murecursion1}. 

Note that the term $m_\kappa(\mu,\lambda)$ is computed by the first case of Theorem \ref{t:T_kappa},
i.e. $\kappaarrow\mu\lambda$, 
and either equals $\epsilon(\mu,\lambda)=\pm1$ (see Definition \ref{d:epsilon}) 
or $(v+v\inv)$.

Consulting Theorem \ref{t:T_kappa} we make the left  side of the Proposition explicit.

\bigskip
\begin{mytable}
\hfil\break\newline
\begin{tabular}{|l|l|}
  \hline
$t_\mu(\kappa)$&LHS of \eqref{e:recursion2}\\\hline
{\tt 1C-,1r1f,2C-,2r11,3C-}&$\P(\gamma,\mu)$\\\hline
{\tt 1r2,2r22}&$\P(\gamma,\mu)+\P(\gamma,w_\kappa\times\mu)$\\\hline
{\tt 2r21}&$\displaystyle\sum_{\mu'|\kappaarrow{\mu}{\mu'}}\epsilon(\mu,\mu')\P(\gamma,\mu')$\\\hline
{\tt 2Cr,3Cr,3r}&$(v+v\inv)\P(\gamma,\mu)$\\\hline
\end{tabular}
\end{mytable}
\medskip

In the {\tt 2r21} case, recall $\mu$ comes in an ordered pair $(\mu_1,\mu_2)$.
Then $\kappaarrow{\mu}{\lambda}$ says that 
$\lambda$ is one member of the ordered pair
$\mu_\kappa=(\lambda_1,\lambda_2)$.

We turn to the right hand side of \eqref{e:recursion2}. The first term  is given by Lemma \ref{l:coefficient}.
Let $k=\ell(\kappa)$. This is copied from Table \ref{table:akappagammainTClambda2},
which is a condensed version of Table  \ref{table:akappagammainTClambda}.
\begin{mytable}
\label{table:condensed2}
\hfil\break\newline
\begin{tabular}{|l|l|l|}
\hline
$t_\gamma(\kappa)$ & first term on RHS of \eqref{e:recursion2}\\
\hline
{\tt 1C-,2C-,3C-}&$v^k\P(\gamma,\lambda)+\P(w_\kappa\times\gamma,\lambda)$\\
\hline
{\tt 1ic,2ic,3ic}&$(v^k+v^{-k})\P(\gamma,\lambda)$\\
\hline
{\tt 2Cr,3Cr,3r}&$v(v^{k-1}-v^{-k+1})\P(\gamma,\lambda)+(v+v\inv)\P(\gamma_\kappa,\lambda)$\\
\hline
{\tt1r1f,2r11}&$(v^k-v^{-k})\P(\gamma,\lambda)+\P(\gamma_\kappa^1,\lambda)+\P(\gamma_\kappa^2,\lambda)$\\
\hline
{\tt1r1s}&$(v-v\inv)\P(\gamma,\lambda)$\\
\hline
{\tt1r2,2r22}&$v^k\P(\gamma,\lambda)-v^{-k}\P(w_\kappa\times\gamma,\lambda)+\P(\gamma_\kappa,\lambda)$\\
\hline
{\tt2r21}&$(v^2-v^{-2})\P(\gamma,\lambda)$+$\displaystyle\sum_{\gamma'|\kappaarrow\gamma{\gamma'}}
\epsilon(\gamma,\gamma')\P(\gamma',\lambda)$\\
\hline
\end{tabular}
\end{mytable}

\subsec{Analysis of the Recursion}
\label{s:analysis}
Fix $\gamma,\mu$ with $\kappa\in\tau(\gamma)$ and 
$t_\mu(\kappa)=${\tt 1C-, 1r1f, 2C-, 2r11, 3C-, 2Cr, 3Cr, 3r}. 
Choose $\lambda$ with $\kappaarrow\mu\lambda$.
Then Proposition \ref{p:recursion} says
\begin{equation}
\label{e:pgammamu}
\P(\gamma,\mu)=\text{coefficient of $\a^\kappa_\gamma$ in $\T_\kappa\Chat_\lambda$}-\U(\gamma,\lambda).
\end{equation}
Recall $\U(\gamma,\lambda)$ is given in Definition \ref{d:U}.
We analyse how this fits in our recursive scheme.


\begin{lemma}
\label{l:U}
Fix $\kappa$. For all $\gamma,\lambda$ with $\kappa\in\tau(\gamma)$,  $\kappa\not\in\tau(\lambda)$:

$$
\begin{aligned}
\U(\gamma,\lambda)=
&\sum_{\substack{\delta|\kappa\in\tau(\delta)\\\ell(\gamma)\le\ell(\delta)<\ell(\lambda)}}
\P(\gamma,\delta)m_\kappa(\delta,\lambda)+\\
&\sum_{\substack{\delta|t_\delta(\kappa)={\tt 2Cr,3Cr,3r}\\\ell(\delta)=\ell(\lambda)+\ell(\kappa)-2}}
\P(\gamma,\delta)\Mu_{-1}(\delta_\kappa,\lambda)+\\
\defect_\lambda(\kappa)\hskip-15pt
&\sum_{\substack{\delta|\kappa\in\tau(\delta)\\\ell(\delta)=\ell(\lambda)+\ell(\kappa)-2}}
\P(\gamma,\delta)\Mu_{-1}(\delta,\lambda^\kappa)\\
\end{aligned}
$$
\end{lemma}

\begin{proof}
Recall $\U(\gamma,\lambda)$ is defined by the sum over
$\delta\kappaless\lambda$. 
Of course $\P(\gamma,\delta)\ne 0$ implies
$\gamma\le\delta$. Therefore, by the definition of $\kappaless$,
$\U(\gamma,\lambda)$ is the sum over $\delta$ satisfying
$\kappa\in\tau(\delta)$, and one of:
\begin{enumerate}
\item $\gamma\le\delta<\lambda$
\item $\defect_\delta(\kappa)=1,
  \delta\not<\lambda,\delta_\kappa<\lambda$
\item $\defect_\lambda(\kappa)=1, \delta\not<\lambda,\delta<\lambda^\kappa$
\end{enumerate}
Every such term appears in Lemma \ref{l:U}. Conversely any nonzero term 
in Lemma \eqref{l:U} appears in one of these cases.
\end{proof}

Recall to compute $P(\gamma,\mu)$ we may assume we know:
\begin{subequations}
\renewcommand{\theequation}{\theparentequation)(\alph{equation}}  
\label{e:indhypo}
\begin{align}
\P(*,\mu')&\text{ if }\ell(\mu')<\ell(\mu)\\
\P(\gamma',\mu)&\text{ if }\ell(\gamma')>\ell(\gamma).
\end{align}
\end{subequations}

\begin{lemma}
Suppose $\kappa\in\tau(\gamma),\tau(\mu)$, and $\kappaarrow\mu\lambda$, and assume 
we know the terms \eqref{e:indhypo}(a) and (b). Then 
$$
\U(\gamma,\lambda)=(*)+
\defect_\lambda(\kappa)(-1)^{[\ell(\gamma)<\ell(\lambda)]}
\Mu_{-1}(\gamma,\mu)
$$
where $(*)$ is known.
\end{lemma}
To be explicit: the extra term is $0$ unless 
$\defect_\lambda(\kappa)=1$ in which case it equals:
\begin{subequations}
\renewcommand{\theequation}{\theparentequation)(\alph{equation}}  
\begin{equation}
\begin{cases}
\Mu_{-1}(\gamma,\mu)&\ell(\gamma)=\ell(\mu)-1\\
-\Mu_{-1}(\gamma,\mu)&\ell(\gamma)<\ell(\mu)-1
\end{cases}
\end{equation}
Equivalently if $\defect_\lambda(\kappa)=1$ it is equal to 
\begin{equation}
\begin{cases}
\Mu_{-1}(\gamma,\mu)&\ell(\gamma)=\ell(\mu)-1\\
-\Mu_{-1}(\gamma,\mu)&\ell(\gamma)=\ell(\mu)-3,5,7\dots
\end{cases}
\end{equation}
and is $0$ otherwise.

Note that
$$
\defect_\lambda(\kappa)=1\Leftrightarrow
\defect_\mu(\kappa)=1\Leftrightarrow
t_\lambda(\kappa)={\tt 2Ci,3Ci,3i}\Leftrightarrow
t_\mu(\kappa)={\tt 2Cr,3Cr,3r}. 
$$
\end{subequations}

\begin{proof}
According to Lemma \ref{l:U} $\U(\gamma,\lambda)$ is equal to:
\begin{subequations}
\renewcommand{\theequation}{\theparentequation)(\alph{equation}}  
\label{e:Urecap}
\begin{align}
&[\ell(\gamma)<\ell(\lambda)]m_\kappa(\gamma,\lambda)+\\
&\sum_{\substack{\delta|\kappa\in\tau(\delta)\\\ell(\gamma)<\ell(\delta)<\ell(\lambda)}}
\P(\gamma,\delta)m_\kappa(\delta,\lambda)+\\
&\hskip-10pt\sum_{\substack{\delta|t_\delta(\kappa)={\tt 2Cr,3Cr,3r}\\\ell(\delta)=\ell(\mu)-2+\defect_\lambda(\kappa)}}
\P(\gamma,\delta)\Mu_{-1}(\delta_\kappa,\lambda)+\\
\defect_\lambda(\kappa)\hskip-15pt
&\sum_{\substack{\delta|\kappa\in\tau(\delta)\\\ell(\delta)=\ell(\mu)-1}}
\P(\gamma,\delta)\Mu_{-1}(\delta,\mu)
\end{align}
\end{subequations}
For (a) we have separated out the term $\ell(\gamma)=\ell(\delta)$ from the first sum in Lemma \ref{l:U}.
Since then $\P(\gamma,\delta)=0$ unless $\gamma=\delta$ this gives (a) (the bracket is due to 
the summation condition).
Also, we have replaced $\ell(\lambda)+\ell(\kappa)$ with $\ell(\mu)+\defect_\lambda(\kappa)$
(Definition \ref{d:defect}) and 
have used $\defect_\lambda(\kappa)=1$ in the final summand.

We analyse each of the terms appearing, starting with the $\P$ and $\mu_{-1}$ terms.

\begin{itemize}
\item In (b,c,d) we know $\P(\gamma,\delta)$ since $\ell(\delta)<\ell(\mu)$ (by \eqref{e:indhypo}(a))
\item In (c) we know $\Mu_{-1}(\delta_\kappa,\lambda)$ since
  $\ell(\lambda)<\ell(\mu)$ (\eqref{e:indhypo}(a))
\item In (d) we know $\Mu_{-1}(\delta,\mu)$ if $\ell(\delta)>\ell(\gamma)$ (by \eqref{e:indhypo}(b))
\item In (d) suppose $\ell(\delta)<\ell(\gamma)$. Then $\P(\gamma,\delta)=0$.
\item In (d) suppose $\ell(\delta)=\ell(\gamma)$. Then $\P(\gamma,\delta)=0$ unless $\gamma=\delta$.
\end{itemize}

So the only additional contribution to (d) is
\begin{equation}
\label{e:additional}
\defect_\lambda(\kappa)[\ell(\gamma)=\ell(\mu)-1]\Mu_{-1}(\gamma,\mu)
\end{equation}

Now consider $m_\kappa(\delta,\lambda)$ in (b). By Theorem
\ref{t:T_kappa} this requires various $\mu_{-k}(*,\lambda)$ ($k=1,2$),
and also $\mu_{-k}(*,\lambda')$ with $\ell(\lambda')<\ell(\lambda)$,
all of which we know.  The only potential problems are the terms
$\Mu_{-1}(\delta,\lambda^\kappa)$ in Theorem \ref{t:T_kappa} (3)(a) and (4)(a). But
$\lambda^\kappa=\mu$, and $\ell(\delta)>\ell(\gamma)$ (from the
summation condition in \eqref{e:Urecap}(b)), so we know these terms by
\eqref{e:indhypo}(b).

The only remaining term is $[\ell(\gamma)<\ell(\lambda)]m_\kappa(\gamma,\lambda)$ in \eqref{e:Urecap}(a). This is the
same as the previous case $m_\kappa(\delta,\lambda)$, except for the
very last case: there may be a term $\mu_{-1}(\gamma,\mu)$ from
Theorem \ref{t:T_kappa}(3)(a) or (4)(a). These arise when
$\defect_\lambda(\kappa)=1$ and also certain length conditions hold.
There are two cases.

\bigskip

\noindent $\ell(\kappa)=2$. 
By   \ref{t:T_kappa}(3)(b) we get an extra term $-\Mu_{-1}(\gamma,\mu)$ 
if $\ell(\gamma)=\ell(\lambda)\mypmod 2$ and $\ell(\gamma)<\ell(\lambda)$.
On the other hand
$\ell(\kappa)=2$ implies 
$\ell(\mu)-\ell(\lambda)=1$, so  \eqref{e:additional} 
is nonzero  only if $\ell(\gamma)=\ell(\lambda)$. So the combined extra contribution is
($\defect_\lambda(\kappa)$ times)
$$
\begin{cases}
\Mu_{-1}(\gamma,\mu)&\ell(\gamma)=\ell(\lambda)\\
-\Mu_{-1}(\gamma,\mu)&\ell(\gamma)=\ell(\lambda)-2k\quad(k=1,2,3,\dots)  
\end{cases}
$$
\begin{subequations}
If $\ell(\lambda)-\ell(\gamma)$ is odd then $\ell(\mu)-\ell(\lambda)$ is even, 
so  $\Mu_{-1}(\gamma,\mu)=0$. Therefore this can be written
\begin{equation}
\begin{cases}
\Mu_{-1}(\gamma,\mu)&\ell(\gamma)=\ell(\lambda)\\
-\Mu_{-1}(\gamma,\mu)&\ell(\gamma)<\ell(\lambda).
\end{cases}
\end{equation}

\renewcommand{\theequation}{\theparentequation)(\alph{equation}}  

\bigskip

\noindent $\ell(\kappa)=3$. 
By   \ref{t:T_kappa}(4)(b) we get an extra term if 
$\ell(\lambda)-\ell(\gamma)=1\mypmod2$
and $\ell(\gamma)<\ell(\lambda)$.
On the other hand $\ell(\kappa)=2$ implies $\ell(\mu)-1=\ell(\lambda)+1$,
 so  \eqref{e:additional} 
is nonzero only if $\ell(\gamma)=\ell(\lambda)+1$.
So the extra contribution is ($\defect_\lambda(\kappa)$ times)
\begin{equation}
  \begin{cases}
\Mu_{-1}(\gamma,\mu)&\ell(\gamma)=\ell(\lambda)+1\\    
-\Mu_{-1}(\gamma,\mu)&\ell(\gamma)=\ell(\lambda)+1-2k\quad (k=1,2,\dots)
  \end{cases}
\end{equation}
If $\ell(\lambda)-\ell(\gamma)$ is even then  so is $\ell(\mu)-\ell(\lambda)$
so  $\Mu_{-1}(\gamma,\mu)=0$. Therefore this term is
\begin{equation}
\begin{cases}
\Mu_{-1}(\gamma,\mu)&\ell(\gamma)=\ell(\lambda)+1\\
-\Mu_{-1}(\gamma,\mu)&\ell(\gamma)<\ell(\lambda)
\end{cases}
\end{equation}
We can combine the two cases:
$$
\defect_\lambda(\kappa)(-1)^{[\ell(\gamma)<\ell(\lambda)]}
\Mu_{-1}(\gamma,\mu)
$$
\end{subequations} 
\end{proof}

\begin{proposition}[Direct Recursion]
\label{p:directrecursion}
Suppose $\kappa\in\tau(\gamma)$, $\kappa\in\tau(\mu)$, and 
$t_\mu(\kappa)=${\tt 1C-, 1r1f, 2C-, 2r11, 3C-, 2Cr, 3Cr, 3r}.

\begin{enumerate}
\item[(a)]  
We know all terms necessary to compute $\P(\gamma,\mu)$ 
unless $\defect_\mu(\kappa)=1$ and $\ell(\mu)-\ell(\gamma)$ is odd (and positive). 
\item[(b)]  Suppose $\defect_\mu(\kappa)=1$, i.e. 
$t_\mu(\kappa)=${\tt 2Cr,3Cr,3r}, and $\ell(\mu)-\ell(\gamma)$ is odd.
Then
$$
(v+v\inv)\P(\gamma,\mu)=c\Mu_{-1}(\gamma,\mu)+(*)
$$
where all terms in (*), and $c=1,2$,  are known.  This can be solved for $\P(\gamma,\mu)$.
\end{enumerate}

\end{proposition}

\begin{proof}
This follows from the preceding Lemma, with the final  case  provided by the next Lemma.
\end{proof}

\begin{lemma}
\label{l:topterm}
Suppose $f(v)\in v\inv\Z[v^{-2}]$, and we know all terms of 
$f(v)(v+v\inv)$ except the constant term. Then we can compute $f(v)$.
\end{lemma}

\begin{proof}
Write $f(v)=c_nv\inv+c_{n-1}v^{-3}+\dots+c_0v^{-(2n+1)}$, and 
$$
(v+v\inv)f(v)=b_{n+1}+b_nv^{-2}+v_{n-1}v^{-4}+\dots+b_0v^{-2n-2}
$$
and we know $b_0,\dots,b_n$.
Starting at $v^{-2n-2}$ we see $c_0=b_0, (c_1+c_0)=b_1, (c_2+c_1)=b_2,
\dots$. This is easy to solve for $c_i$:
$c_0=b_0,c_1=b_0-b_1,c_2=b_2-b_1+b_0,\dots$. 
That is:
$$
c_k=(-1)^{k}\sum_0^k (-1)^jb_j\quad(0\le k\le n).
$$
\end{proof}

\begin{remarkplain}
It might be easier to think about this if we multiply by $v^{2n+2}$ and replace $v^2$ with $q$. 
This gives
$$
(1+q)(c_0+c_1q^2+\dots+c_nq^n)=b_0+b_1q+\dots+b_{n+1}q^{n+1}
$$
If we know all terms on the right hand side except for $b_{n+1}$, we can find all $c_i$.
\end{remarkplain}

\begin{remarkplain}
This computes $\P(\gamma,\mu)$ ($\kappa\in\tau(\gamma),\tau(\mu)$) unless:
\begin{enumerate}
\item there is no $\lambda$ with $\kappaarrow\mu\lambda$:
{\tt 1r1s,1ic,2ic,3ic}
\item $\kappaarrow\mu\lambda$ for some $\lambda$,
but $|\{\mu'|\kappaarrow{\mu'}{\lambda}\}|=2$:
{\tt 1r2,2r22,2r21}.
\end{enumerate}
See Section \ref{s:guide}.
\end{remarkplain}

\sec{New Recursion Relations}
\label{s:new}

We return now to the setting of Lemma \ref{l:murecursion1}, and the case we skipped earlier.

\begin{lemma}
\label{l:lhsempty2}
Assume
$t_\lambda(\kappa)=$ {\tt 1i2s,1rn,2rn,3rn}. Equivalently  $\kappa\not\in\tau(\lambda)$
but there does not exist $\lambda'$ satisfying $\kappaarrow{\lambda'}{\lambda}$. 
Then for any $\gamma$:

\begin{equation}
\label{e:hsempty2}
\sum_\mu\P(\mu,\lambda)(\text{multiplicity of $\a_\gamma$ in $\T_\kappa(\a_\mu)$})=\U(\gamma,\lambda)
\end{equation}
\end{lemma}

This is an immediate consequence of Lemma \ref{l:lhsempty}, which says
that under this assumption the $\text{coefficient of $\a_\gamma$ in
  $\T_\kappa\Chat_\lambda$}=\U(\gamma,\lambda)$.

The left hand side has at most $3$ terms, which can be read off from
the tables in Section \ref{s:formulasfortkappa}, or
Table \ref{table:Tagamma}. One of the terms is a polynomial times
$\P(\gamma,\lambda)$, and we wish to solve for $\P(\gamma,\lambda)$. 

If $\kappa\not\in\tau(\gamma)$ then 
the only possibilities for $\mu$ are  $\mu=\gamma, \mu=w_\kappa\times\gamma$,
or $\kappaarrow\mu\gamma$. These are all well suited to using
induction to computing $\P(\gamma,\lambda)$,
unless $t_\gamma(\kappa)=${\tt 1i2s,1rn,2rn,3rn}, in which case  the coefficient
of $\P(\gamma,\lambda)$ is $0$.

If $\kappa\in\tau(\gamma)$
then $\mu=\gamma,\mu=w_\kappa\times\gamma$ or $\kappaarrow\gamma\mu$.
If $\kappaarrow\gamma\mu$ then $\gamma>\mu$ and this is not well
suited to our inductive hypothesis. So this case is only effective if
there is no such $\mu$, i.e. $t_\gamma(\kappa)=${\tt 1r1s,1ic,2ic,3ic}.

Here are the resulting formulas.

\begin{mytable}
\label{table:newnotkappa}
\hfil\break\newline
\begin{tabular}{|l|l|}
\hline
\multicolumn{2}{|c|}{Formula  \eqref{e:hsempty2}: $\kappa\not\in\tau(\gamma)$}\\\hline
$t_\gamma(\kappa)$ & LHS\\\hline  
{\tt 1C+}&$v\inv\P(\gamma,\lambda)+\P(w_\kappa\times\gamma,\lambda)$\\\hline
{\tt 1i1}&
$v\inv\P(\gamma,\lambda)+v\inv\P(w_\kappa\times\gamma,\lambda)+(1-v^{-2})\P(\gamma^\kappa,\lambda)$\\\hline
{\tt 1i2f}&$2v\inv\P(\gamma,\lambda)+v\inv(v-v\inv)(\P(\gamma_1^\kappa,\lambda)+\P(\gamma_2^\kappa,\lambda))$\\\hline
{\tt 1i2s}&$0$\\\hline
{\tt 1rn}&$0$\\\hline
{\tt 2C+}&$v^{-2}\P(\gamma,\lambda)+\P(w_\kappa\times\gamma,\lambda)$\\\hline
{\tt 2Ci}&$v\inv(v+v^{-1})\P(\gamma,\lambda)+(v-v\inv)\P(\gamma^\kappa,\lambda)$\\\hline
{\tt 2i11}&$v^{-2}(\P(\gamma,\lambda)+\P(w_\kappa\times\gamma,\lambda))+(1-v^{-4})\P(\gamma^\kappa,\lambda)$
\\\hline
{\tt 2i12}&$2v^{-2}\P(\gamma,\lambda)+(1-v^{-4})\displaystyle\sum_{\gamma'|\kappaarrow{\gamma'}\gamma}\epsilon(\gamma',\gamma)\P(\gamma',\lambda)$
\\\hline
{\tt 2i22}&$2v^{-2}\P(\gamma,\lambda)+(1-v^{-4})(\P(\gamma^\kappa_1,\lambda)+\P(\gamma^\kappa_2,\lambda))$\\\hline
{\tt 2rn}&$0$\\\hline
{\tt
  3C+}&$v^{-3}\P(\gamma,\lambda)+\P(w_\kappa\times\gamma,\lambda)$\\\hline
{\tt 3Ci,3i}&$v^{-2}(v+v\inv)\P(\gamma,\lambda)+(v^2-v^{-2})v\inv\P(\gamma^\kappa,\lambda)$\\\hline
{\tt 3rn}&$0$\\\hline
\end{tabular}
\end{mytable}

\begin{mytable}
\label{table:newkappa}
\hfil\break\newline
\begin{tabular}{|l|l|}
\hline
\multicolumn{2}{|c|}{Formula  \eqref{e:hsempty2}: $\kappa\in\tau(\gamma)$}\\\hline
$t_\gamma(\kappa)$ & LHS\\\hline  
{\tt
  1r1s}&$(v+v\inv)\P(\gamma,\lambda)$\\\hline
{\tt 1ic}&$(v+v\inv)\P(\gamma,\lambda)$\\\hline
{\tt 2ic}&$(v^2+v^{-2})\P(\gamma,\lambda)$\\\hline
{\tt 3ic}&$(v^3+v^{-3})\P(\gamma,\lambda)$\\\hline
\end{tabular}
\end{mytable}

Solve these for $\P(\gamma,\lambda)$.

\begin{lemma}
Assume $t_\lambda(\kappa)=${\tt 1i2s,1rn,2rn,3rn}. 
Then:

\begin{mytable}
\label{table:new}
\hfil\break\newline
\begin{tabular}{|l|l|}
\hline
\multicolumn{2}{|c|}{Formula for $\P(\gamma,\lambda)$, $t_\lambda(\kappa)=${\tt 1i2s,1rn,2rn,3rn}. }\\\hline
$t_\gamma(\kappa)$ & Formula for $\P(\gamma,\lambda)$
\\\hline  
$\kappa\not\in\tau(\gamma)$&\\\hline
{\tt 1C+}&$\P(\gamma,\lambda)=-v\P(w_\kappa\times\gamma,\lambda)+v\U(\gamma,\lambda)$\\\hline
{\tt 1i1}&$\P(\gamma,\lambda)+\P(w_\kappa\times\gamma,\lambda)=-(v-v^{-1})\P(\gamma^\kappa,\lambda)+v\U(\gamma,\lambda)$\\\hline
{\tt 1i2f}&$\P(\gamma,\lambda)=-2(v-v\inv)(\P(\gamma_1^\kappa,\lambda)+\P(\gamma_2^\kappa,\lambda))+2v\U(\gamma,\lambda)$\\\hline
{\tt 1i2s}&none\\\hline
{\tt 1rn}&none\\\hline
{\tt 2C+}&$\P(\gamma,\lambda)=-v^2\P(w_\kappa\times\gamma,\lambda)+v^2\U(\gamma,\lambda)$\\\hline
{\tt 2Ci}&$(v+v^{-1})\P(\gamma,\lambda)=-v(v-v\inv)\P(\gamma^\kappa,\lambda)+v\U(\gamma,\lambda)$\\\hline
{\tt 2i11}&$\P(\gamma,\lambda)+\P(w_\kappa\times\gamma,\lambda)=-(v^2-v^{-2})\P(\gamma^\kappa,\lambda)+v^2\U(\gamma,\lambda)$
\\\hline
{\tt 2i12}&$2v^{-2}\P(\gamma,\lambda)=-\frac12(v^2-v^{-2})\displaystyle\sum_{\gamma'|\kappaarrow{\gamma'}\gamma}\epsilon(\gamma',\gamma)\P(\gamma',\lambda)+\tfrac12v^2\U(\gamma,\lambda)$
\\\hline
{\tt 2i22}&$\P(\gamma,\lambda)=-\frac12(v^2-v^{-2})(\P(\gamma^\kappa_1,\lambda)+\P(\gamma^\kappa_2,\lambda))+\frac12v^2\U(\gamma,\delta)$\\\hline
{\tt 2rn}&none\\\hline
{\tt
  3C+}&$\P(\gamma,\lambda)=-v^3\P(w_\kappa\times\gamma,\lambda)+v^3\U(\gamma,\delta)$\\\hline
{\tt 3Ci,3i}&$(v+v\inv)\P(\gamma,\lambda)=-v(v^2-v^{-2})\P(\gamma^\kappa,\lambda)+v^2\U(\gamma,\lambda)$\\\hline
{\tt 3rn}&none\\\hline\hline
$\kappa\in\tau(\gamma)$&\\\hline
{\tt 1r1s}&$(v+v\inv)\P(\gamma,\lambda)=\U(\gamma,\lambda)$\\\hline
{\tt 1ic}&$(v+v\inv)\P(\gamma,\lambda)=\U(\gamma,\lambda)$\\\hline
{\tt 2ic}&$(v^2+v^{-2})\P(\gamma,\lambda)=\U(\gamma,\lambda)$\\\hline
{\tt 3ic}&$(v^3+v^{-3})\P(\gamma,\lambda)=\U(\gamma,\lambda)$\\\hline
\end{tabular}
\end{mytable}
\end{lemma}

In every case we know all terms on the RHS by induction, with the possible exception of $\U(\gamma,\lambda)$. 
By Lemma \ref{l:U}, since $t_\lambda(\kappa)\ne${\tt 2Ci,3Ci,3i},
\begin{equation}
\begin{aligned}
\U(\gamma,\lambda)=\sum_{\substack{\delta|\kappa\in\tau(\delta)\\\ell(\gamma)\le\ell(\delta)<\ell(\lambda)}}
\P(\gamma,\delta)m_\kappa(\delta,\lambda)+
\sum_{\substack{\delta|t_\delta(\kappa)={\tt 2Cr,3Cr,3r}\\\ell(\delta)=\ell(\lambda)+\ell(\kappa)-2}}
\P(\gamma,\delta)\Mu_{-1}(\delta_\kappa,\lambda)
\end{aligned}
\end{equation}
The second sum is  problematic, so we give it a name.
\begin{definition}
\label{d:Udagger}
Suppose $\kappa\not\in\tau(\lambda)$. Define:
\begin{equation}
\U^\dagger(\gamma,\lambda)
\sum_{\substack{\delta|t_\delta(\kappa)={\tt 2Cr,3Cr,3r}\\\ell(\delta)=\ell(\lambda)+\ell(\kappa)-2}}
\P(\gamma,\delta)\Mu_{-1}(\delta_\kappa,\lambda)
\end{equation}
\end{definition}
This term doesn't fit well in our recursion scheme, since
$\ell(\delta)>\ell(\lambda)$. 
(In the direct recursion section this wasn't an issue, since we were
computing 
$\P(\gamma,\mu)$, and $\ell(\mu)>\ell(\delta)$.) 
We're hoping this term is usually $0$.

\begin{definition}
Fix $\kappa\in\overline S,\gamma,\lambda\in \D^\sigma$, with 
$\kappa\not\in\tau(\lambda)$. 

We say condition A holds for $(\kappa,\gamma,\lambda)$ if 
\begin{equation}
\label{e:A}
\tag{A}
\kappa\in\tau(\delta),\,\defect_\delta(\kappa)=1,\,\ell(\delta)=\ell(\lambda)+\ell(\kappa)+2\Rightarrow
\P(\gamma,\delta)\Mu_{-1}(\delta_\kappa,\lambda)=0
\end{equation}
This is automatic if $\ell(\kappa)=1$.

We say condition $(\dagger)$ holds for  $(\kappa,\gamma,\lambda)$  if 
\begin{equation}
\tag{$\dagger$}\U^\dagger(\gamma,\lambda)=0
\end{equation}
\end{definition}

Obviously (A)$\Rightarrow(\dagger)$, although (A) is easier to check.
It often holds simply because the set in   (A) is empty.

So now we need to compute:
\begin{equation}
\label{e:Ugammalambda}
\U(\gamma,\lambda)=
m_\kappa(\gamma,\lambda)+\sum_{\substack{\delta|\kappa\in\tau(\delta)\\\ell(\gamma)<\ell(\delta)<\ell(\lambda)}}
\P(\gamma,\delta)m_\kappa(\delta,\lambda)+\U^\dagger(\gamma,\lambda)
\end{equation}
with the first term present only if $\kappa\in\tau(\gamma)$.

As in the discussion in Section
\ref{s:analysis} we know all of the terms in the first sum. 
There is one crucial difference with the direct recursion of Proposition \ref{p:directrecursion}.
In that lemma we were computing $\P(\gamma,\mu)$, and we needed
$\P(\gamma,\lambda)$, where $\kappaarrow\mu\lambda$. This gave us a
little extra room when applying the inductive hypothesis that we know
$\P(*,\mu')$ with $\mu'<\mu$. 

We consider the terms appearing in the sum in \eqref{e:Ugammalambda}.

If $\kappa\in\tau(\gamma)$ we don't know the term
$m_\kappa(\gamma,\lambda)$. 

By the inductive hypothesis we know all terms $\P(\gamma,\delta)$
occuring in \eqref{e:Ugammalambda}. 
Consider $m_\kappa(\delta,\lambda)$. 
This requires various $\P(*,\lambda')$ with
$\ell(\lambda)'<\ell(\lambda)$, and
$\P(\delta',\lambda)$ with $\ell(\delta')\ge
\ell(\delta)>\ell(\gamma)$, all of which we know.
The only potential problems are the terms 
$$
\P(\gamma,\delta)\Mu_{-1}(\delta_\kappa,\lambda)\text{ if } \defect_\gamma(\kappa)=1
$$
which occur in the formulas for $m_\kappa(\delta,\lambda)$
(the opposite case $\P(\delta,\lambda^\kappa)$ does not occur since
$t_\lambda(\kappa)\ne${\tt 2Ci,3Ci,3i}).
If $\ell(\delta_\kappa)>\ell(\gamma)$ we know this by induction. 
This leaves:
$$
\P(\gamma,\delta)\Mu_{-1}(\delta_\kappa,\lambda)\quad
\ell(\delta_\kappa)\le \ell(\gamma)<\ell(\delta), \defect_\gamma(\kappa)=1.
$$
This might be an issue. (In the Direct Recursion this was
taken care of by the fact that $\lambda<\mu$). 
As in the discussion after Definition \ref{d:kappaless} this 
term is nonzero only if:
$$
\ell(\delta)=\ell(\gamma)+\ell(\kappa)-2,\quad
\ell(\gamma_\kappa)=\ell(\delta)-1
$$
\begin{definition}
Fix $\kappa\in\overline S,\gamma,\lambda\in \D^\sigma$, with 
$\kappa\not\in\tau(\lambda)$. We say condition
(B) holds for $(\kappa,\gamma,\lambda)$ if 
\begin{equation}
\tag{B}
\kappa\in\tau(\delta),\,\defect_\delta(\kappa)=1,\,\ell(\delta)=\ell(\gamma)+\ell(\kappa)+2\Rightarrow
\P(\gamma,\delta)\Mu_{-1}(\delta_\kappa,\lambda)=0
\end{equation}
This is automatic if $\ell(\kappa)=1$.

\end{definition}

Here is the conclusion.

\newpage

\begin{proposition}
\label{p:newrecursion}
Fix $\kappa\in\overline S,\gamma,\lambda\in \D^\sigma$, satisfying:
\begin{enumerate}
\item $t_\gamma(\kappa)=${\tt 1C+, 1i2f, 2C+, 2Ci, 2i12,2i22, 3C+, 3Ci, 3i,
1r1s, 1ic, 2ic, 3ic}.
\item $t_\lambda(\kappa)=${\tt 1i2s,1rn,2rn,3rn}
\end{enumerate}
Assume conditions (A) and (B) hold. 
Then the formulas of Table \ref{table:new} give an effective recursion 
relation for $\P(\gamma,\lambda)$.
It is sufficient to assume Conditions ($\dagger$) and (B). 

If $t_\gamma(\kappa)=${\tt 1i1,2i11} the same result holds, except
that we get a formuls for
$\P(\gamma,\lambda)+\P(w_\kappa\times\gamma,\lambda)$. 
\end{proposition}

\begin{remarkplain}
\hfil
\begin{enumerate}
\item We allow $\lambda$ if  $\kappa\not\in\tau(\lambda)$ but
there is no $\lambda'$ with $\kappaarrow{\lambda'}{\lambda}$ (see
Lemma \ref{l:lhsempty2}): {\tt 1i2s,1rn,2rn,3rn}. This gives the
$\lambda$ of the Proposition.
\item If $\kappa\not\in\tau(\gamma)$ we exclude $\gamma$ if the
  formula has two terms on the LHS: {\tt 1i1, 2i11}
\item If $\kappa\not\in\tau(\gamma)$ we exclude $\gamma$  
if  there is no $\gamma'$ with
  $\kappaarrow{\gamma'}{\gamma}$: type {\tt 1i2s, 1rn, 2rn, 3rn}.
Together (2) and (3) leave:
{\tt 1C+, 1i2f, 2C+, 2Ci, 2i12,2i22, 3C+, 3Ci, 3i}
\item If $\kappa\in\tau(\gamma)$ we include $\gamma$ if  there is no $\gamma'$ with
  $\kappaarrow\gamma{\gamma'}$: {\tt 1r1s,1ic,2ic, 3ic}. (2-4) give
  the $\gamma$ of the Proposition.
\end{enumerate}
\end{remarkplain}

If $\kappa\not\in\tau(\gamma)$ the term $m_\kappa(\gamma,\lambda)$
does not appear in Lemma \ref{l:U}, 
and the Lemma is evident from the preceding discussion.
If $\kappa\in\tau(\gamma)$  we need a generalization of 
Lemma \ref{l:topterm}.
Here are the cases.

In the column
$m_\kappa(\gamma,\lambda)$ $\alpha$ is an unknown constant. 
We have written $\P(\gamma,\lambda)=v\inv f(v^{-2})$ or
$v^{-2}f(v^{-2})$, depending on the parity of
$\ell(\lambda)-\ell(\gamma)$, where $f$ is a polynomial.
In the last column
$g$ is a  polynomial which is known, and $\alpha,\beta,\gamma$ are
unkown constants. We want to solve for $f$.

\newpage

\begin{tabular}{|l|l|l|l|}
\hline
$\ell(\kappa)$&$\ell(\lambda)-\ell(\gamma)$&
$m_\kappa(\gamma,\lambda)$ & equation\\\hline
$1$& even & 0 &   $(v\pm v\inv)v^{-2}f(v^{-2})=v\inv g(v^{-2})$\\\hline
$1$& odd & $\alpha$ &   $(v\pm v\inv)v^{-1}f(v^{-2})=\alpha+v^{-2} g(v^{-2})$\\\hline
$2$& even & $\alpha$ &   $(v^2- v^{-2})v^{-2}f(v^{-2})=\alpha+v^{-2} g(v^{-2})$\\\hline
$2$& odd & $\alpha(v+v^{-1})$ &   $(v^2- v^{-2})v^{-1}f(v^{-2})=\alpha v+\beta v\inv+v^{-3} g(v^{-2})$\\\hline
$3$& even & $\alpha(v+v^{-1})$ &   $(v^3- v^{-3})v^{-2}f(v^{-2})=\alpha v+\beta v\inv+v^{-3} g(v^{-2})$\\\hline
$3$& odd & $\alpha$ &   $(v^3- v^{-3})v^{-1}f(v^{-2})=\beta v^2+\alpha+\gamma v^{-2}+v^{-4} g(v^{-2})$\\\hline
\end{tabular}

\begin{lemma}
\label{l:topterm2}
In each case in the table we can solve for $f$.
\end{lemma}

After multiplying by the appropriate power of $v$, these all come down
to:

\begin{lemma}
Suppose $(1\pm q^k)f(q)=g(q)$ where $g(q)$ is a polyomial. If we know
all but the top $k$ coefficients of $g$, then we can solve for $f$.
\end{lemma}

\sec{Guide}
\label{s:guide}
In the following tables, we've indicated which formulas to use in various cases.

\begin{enumerate}
\item *: not primitive, easy recursion
\item NE: not extremal (but primitive): easy recursion, but has a sum on the right hand side
\item *0: not primitive, necessarily $0$
\item DR: direct recursion (Proposition \ref{p:directrecursion})
\item DR+: new type of direct recursion (Proposition
  \ref{p:newrecursion}, $\ell(\kappa)=1$)
\item DR+?: new type of direct recursion, but the recursion may
  not work. See Proposition \ref{p:newrecursion}, $\ell(\kappa)=2,3$;
  we need conditions (A) and (B).
\end{enumerate}

In (2), \{{\it non-primitive}\}$\subset$\{{\it non-extremal}\}; the
pairs marked NE are in the second set, but not the first
(they are 
{\it not non-primitive}).

\newpage

\centerline{\bf\large Type 1}
\bigskip

{\small
\tt\begin{tabular}{|c|c|c|c|c|c||c|c|c|c|c|}
\hline
&1C-&1r1f&1r1s&1r2&1ic&1C+&1i1&1i2f&1i2s&1rn\\
\hline
1C-&DR&DR&&&&&&&&\\
\hline
1r1f&DR&DR&&&&&&&&\\
\hline
1r1s&DR&DR&&&&&&&DR+&DR+\\
\hline
1r2&DR&DR&&&&&&&&\\
\hline
1ic&DR&DR&&&&&&&DR+&DR+\\
\hline\hline
1C+&*&*&*&*&*&&&&DR+&DR+\\
\hline
1i1&*&*&*&*&*&&&&&\\
\hline
1i2f&NE&NE&NE&NE&NE&&&&DR+&DR+\\
\hline
1i2s&*0&*0&*0&*0&*0&&&&&\\
\hline
1rn&*0&*0&*0&*0&*0&&&&&\\
\hline
\end{tabular}}

\bigskip

\centerline{\bf\large Type 2}
\bigskip
{\small
{\tt\begin{tabular}{|c|c|c|c|c|c|c||c|c|c|c|c|c|}
\hline
&2C-&2Cr&2r22&2r21&2r11&2ic&2C+&2Cif&1i11&1i12&2i22&2rn\\
\hline
2C-&DR&DR&&&DR&&&&&&&\\
\hline
2Cr&DR&DR&&&DR&&&&&&&\\
\hline
2r22&DR&DR&&&DR&&&&&&&\\
\hline
2r21&DR&DR&&&DR&&&&&&&\\
\hline
2r11&DR&DR&&&DR&&&&&&&\\
\hline
2ic&DR&DR&&&DR&&&&&&&DR+?\\
\hline\hline
2C+&*&*&*&*&*&*&&&&&&DR+?\\
\hline
2Ci&*0&*0&*0&*0&*0&*0&&&&&&DR+?\\
\hline
2i11&*&*&*&*&*&*&&&&&&\\
\hline
2i12&NE&NE&NE&NE&NE&NE&&&&&&DR+?\\
\hline
2i22&*&*&*&*&*&*&&&&&&DR+?\\
\hline
2rn&*0&*0&*0&*0&*0&*0&&&&&&\\
\hline
\end{tabular}}
}

\newpage

\centerline{\bf\large Type 3}
\bigskip

{\small
\tt\begin{tabular}{|c|c|c|c|c||c|c|c|c|c|c|}
\hline
&3C-&3Cr&3r&3ic&3C+&3Ci&3i&3rn\\
\hline
3C-&DR&DR&DR&&&&&\\
\hline
3Cr&DR&DR&DR&&&&&\\
\hline
3r&DR&DR&DR&&&&&\\
\hline
3ic&DR&DR&DR&&&&&DR+?\\
\hline\hline
3C+&*&*&*&*&&&&DR+?\\
\hline
3Ci&*&*&*&*&&&&DR+?\\
\hline
3i&*&*&*&*&&&&DR+?\\
\hline
3rn&*0&*0&*0&*0&&&&\\
\hline
\end{tabular}}

\bigskip

\newpage










\sec{Appendix I: Proof of Theorem \ref{t:T_kappa}}
Throughout this section we assume $\kappa\in\tau(\gamma),\kappa\not\in\tau(\lambda)$.
Recall we are trying to find $m_\kappa(\gamma,\lambda)$ 
such that:
$$
\T_\kappa\Chat_\lambda=\sum_{\gamma|\kappa\in\tau(\gamma)}m_\kappa(\gamma,\lambda)\Chat_\gamma
$$
The main tool is the identity \eqref{e:basicidentity} (for $\kappa\not\in\tau(\lambda)$):

\begin{equation}
\label{e:recursion3}
\sum_{\mu|\kappa\in\tau(\mu)}\P(\gamma,\mu)m_\kappa(\mu,\lambda)=
\text{multiplicity of $\a_\gamma$ in }\T_\kappa(\Chat_\lambda)
\end{equation}
The right hand side is given by Table \ref{table:condensed2}, which we reproduce here.

\begin{mytable}
\label{table:condensed3}
\hfil\break\newline
\begin{tabular}{|l|l|l|}
\hline
$t_\gamma(\kappa)$ & RHS  of \eqref{e:recursion3}\\
\hline
{\tt 1C-,2C-,3C-}&$v^k\P(\gamma,\lambda)+\P(w_\kappa\times\gamma,\lambda)$\\
\hline
{\tt 1ic,2ic,3ic,1r1s}&$(v^k+v^{-k})\P(\gamma,\lambda)$\\
\hline
{\tt 2Cr,3Cr,3r}&$v(v^{k-1}-v^{-k+1})\P(\gamma,\lambda)+(v+v\inv)\P(\gamma_\kappa,\lambda)$\\
\hline
{\tt1r1f,2r11}&$(v^k-v^{-k})\P(\gamma,\lambda)+\P(\gamma_\kappa^1,\lambda)+\P(\gamma_\kappa^2,\lambda)$\\
\hline
{\tt1r2,2r22}&$v^k\P(\gamma,\lambda)-v^{-k}\P(w_\kappa\times\gamma,\lambda)+\P(\gamma_\kappa,\lambda)$\\
\hline
{\tt2r21}&$(v^2-v^{-2})\P(\gamma,\lambda)$+$\displaystyle\sum_{\gamma'|\kappaarrow\gamma{\gamma'}}
\epsilon(\gamma,\gamma')\P(\gamma',\lambda)$\\\hline
\end{tabular}
\end{mytable}
We are going to look at the top degree terms of both sides.
\newpage
Write any element of $\Z[v,v\inv]$ as $f=f^++f^-$ where $f^+\in\Z[v]$
and $f^-\in v\inv\Z[v\inv]$. 
We make frequent use of:

\begin{lemma}
If $\gamma=\mu$ then 
$$
[\P(\gamma,\mu)m_\kappa(\mu,\lambda)]^+=m_{\kappa,0}(\mu,\lambda)+m_{\kappa,1}(\mu,\lambda)v+m_{\kappa,2}(\mu,\lambda)v^2
$$
If $\gamma<\mu$ then
$$
\begin{aligned}
\P(\gamma,\mu)m_\kappa(\mu,\lambda)^+=&
[\Mu_{-2}(\gamma,\mu)m_{\kappa,2}(\mu,\lambda)+
\Mu_{-1}(\gamma,\mu)m_{\kappa,1}(\mu,\lambda)]+\\
&\Mu_{-1}(\gamma,\mu)m_{\kappa,2}(\mu,\lambda)v
\end{aligned}
$$
\end{lemma}

\begin{lemma}
\label{l:length1}
If $\ell(\kappa)=1$ then
\begin{equation}
m_\kappa(\gamma,\lambda)=m_{\kappa,0}(\gamma,\lambda)=
\begin{cases}
1&\kappaarrow\gamma\lambda\\
\Mu_{-1}(\gamma,\lambda)&\gamma<\lambda\\
0&\text{else}
\end{cases}
\end{equation}
\end{lemma}

\begin{proof}
Since $\ell(\kappa)=1$, $m_\kappa(\mu,\lambda)=m_{\kappa,0}(\mu,\lambda)$ for
all $\mu$, and
on the left hand side of 
\eqref{e:recursion3}  the maximal degree is $0$.
In each term on the right hand side,
$[v\P(\gamma,\lambda)]^+=\Mu_{-1}(\gamma,\lambda)$
and $[v\inv\P(\gamma,\lambda)]^+=0$.
On the other hand a term
$\P(\mu,w_\kappa\times\gamma),\P(\mu,\gamma_\kappa)$  or
$\P(\mu,\gamma_\kappa^i)$ contributes $1$ if 
and only if the two arguments are equal. 
So, $[\quad]^+$ of both 
sides gives (the last column is $t_\gamma(\kappa)$):
\begin{equation}
m_\kappa(\gamma,\lambda)=\Mu_{-1}(\gamma,\lambda)+
\begin{cases}
\delta_{w_\kappa\times\gamma,\lambda}&{\tt 1C-}\\
\delta_{\gamma_\kappa^,\lambda}+\delta_{\gamma_\kappa^2,\lambda}&{\tt 1r1f}\\
0&{\tt 1r1s}\\
\delta_{\gamma_\kappa,\lambda}&{\tt 1r2}\\
0&{\tt 1ic}
\end{cases}
\end{equation}
Each Kronecker $\delta$ after the brace is
non-zero precisely when 
$\kappaarrow\gamma\lambda$, in which case $\gamma>\lambda$,
so $\Mu_{-1}(\gamma,\lambda)=0$.
\end{proof}

This proves Cases (1) $(\ell(\kappa)=1)$ and (2) of Theorem \ref{t:T_kappa}.
Note that $m_\kappa(\gamma,\lambda)=0$ unless $\kappaarrow\gamma\lambda$ or $\gamma<\lambda$.

\subsec{$\ell(\kappa)=2$}

Take the $+$ part of both sides of \eqref{e:recursion3}. The left hand
side is
\begin{subequations}
\renewcommand{\theequation}{\theparentequation)(\alph{equation}}  
\begin{equation}
[m_{\kappa,0}(\gamma,\lambda)+
\sum_{\mu|\kappa\in\tau(\mu)} \Mu_{-1}(\gamma,\mu)m_{\kappa,1}(\mu,\lambda)]+m_{\kappa,1}(\gamma,\lambda)v
\end{equation}
The first and last terms are from $\mu=\gamma$, and the second sum is from all other terms $\mu\ne\gamma$. 
(Note that the summand is $0$ if $\mu=\gamma$.) 

The right hand side is
\begin{equation}
\begin{aligned}
&\Mu_{-2}(\gamma,\lambda)+\Mu_{-1}(\gamma,\lambda)v+
\begin{cases}
\delta_{w_\kappa\times\gamma,\lambda}&{\tt 2C-}\\
\Mu_{-1}(\gamma_\kappa,\lambda)+\delta_{\gamma_\kappa,\lambda}v&{\tt 2Cr}\\
\delta_{\gamma_{\kappa},\lambda}&{\tt 2r22}\\
\displaystyle\sum_{\gamma'|\kappaarrow\gamma{\gamma'}}\epsilon(\gamma,\gamma')\delta_{\gamma',\lambda}
&{\tt 2r21}\\
\delta_{\gamma_\kappa^1,\lambda}
+
\delta_{\gamma_\kappa^2,\lambda}&{\tt 2r11}\\
0&2ic
\end{cases}
\end{aligned}
\end{equation}

Equating the coefficient of $v$ in (a) and (b) gives
\begin{equation}
m_{\kappa,1}(\gamma,\lambda)=
\begin{cases}
\Mu_{-1}(\gamma,\lambda)&t_\gamma(\kappa)\ne {\tt 2Cr}\\
\Mu_{-1}(\gamma,\lambda)&t_\gamma(\kappa)={\tt 2Cr},\text{ and
}\gamma<\lambda\\
1&t_\gamma(\kappa)={\tt 2Cr},\text{ and }\kappaarrow\gamma\lambda\\
0&t_\gamma(\kappa)={\tt 2Cr},\text{ otherwise }
\end{cases}
\end{equation}
We can rewrite this
\begin{equation}
\label{e:m1gammadeltacase2}\boxed{
m_{\kappa,1}(\gamma,\lambda)=\Mu_{-1}(\gamma,\lambda)+
\begin{cases}
\delta_{\gamma_\kappa,\lambda}&t_\gamma(\kappa)={\tt 2Cr}\\
0&\text{otherwise}
\end{cases}}
\end{equation}
In particular $m_{\kappa,1}(\gamma,\lambda)=0$ unless $\gamma<\lambda$ or
$\kappaarrow\gamma\lambda$. 

Now (d) holds for all $\gamma$ with $\kappa\in\tau(\gamma)$, so we can
apply it to all $\gamma$ occuring the sum in \eqref{e:recursion3}.

So plug this back in to  (a),  keep only the constant term, 
and set this equal to the constant term of (b)
\begin{equation}
\begin{aligned}
m_{\kappa,0}(\gamma,\lambda)+\sum_{\mu|\kappa\in\tau(\mu)}&
\Mu_{-1}(\gamma,\mu)\bigg[\Mu_{-1}(\mu,\lambda)+
\begin{cases}
\delta_{\mu_\kappa,\lambda}&t_\mu(\kappa)={\tt 2Cr}\\
0&\text{else}  
\end{cases}
\bigg]
=\\
&\Mu_{-2}(\gamma,\lambda)+
\begin{cases}
\delta_{w_\kappa\times\gamma,\lambda}&{\tt2C-}\\
\Mu_{-1}(\gamma_\kappa,\lambda)&{\tt 2Cr}\\
\delta_{\gamma_{\kappa},\lambda}&{\tt 2r22-}\\
\displaystyle\sum_{\gamma'|\kappaarrow\gamma{\gamma'}}\epsilon(\gamma,\gamma')\delta_{\gamma',\lambda}&{\tt 2r21}\\
\delta_{\gamma_\kappa^1,\lambda}
+
\delta_{\gamma_\kappa^2,\lambda}&{\tt 2r11}\\
0&{\tt 2ic}
\end{cases}
\end{aligned}
\end{equation}
\end{subequations}

Note that each Kronecker $\delta$ after the brace  is $1$ iff
$\kappaarrow\gamma\lambda$.
Also we can put $\epsilon(\gamma,\lambda)$ in front of each such term
without 
changing anything (these are $1$ unless $t_\gamma(\kappa)=${\tt 2r21}).
Therefore
\begin{equation}
\begin{aligned}
m_{\kappa,0}(\gamma,\lambda)&=\Mu_{-2}(\gamma,\lambda)\\
&-\sum_\mu\Mu_{-1}(\gamma,\mu)\Mu_{-1}(\mu,\lambda)\\
&-
\sum_\mu
\Mu_{-1}(\gamma,\mu)
*
\begin{cases}
\delta_{\mu_\kappa,\delta}&t_\mu(\kappa)={\tt 2Cr}\\
0&\text{else}  
\end{cases}\\
&+\begin{cases}
\Mu_{-1}(\gamma_\kappa,\lambda)&t_\gamma(\kappa)={\tt 2Cr}\\
\epsilon(\gamma,\lambda)&\kappaarrow\gamma\lambda, t_\gamma(\kappa)\ne{\tt 2Cr}\\
0&\text{otherwise}  
\end{cases}
\end{aligned}
\end{equation}
Then

\newpage

\begin{equation}
\sum_\mu
\Mu_{-1}(\gamma,\mu)
*
\begin{cases}
\delta_{\mu_\kappa,\lambda}&t_\mu(\kappa)={\tt 2Cr}\\
0&\text{else}  
\end{cases}
\end{equation}
is equal to
\begin{equation}
\begin{aligned}
\begin{cases}
\Mu_{-1}(\gamma,\lambda^\kappa)&t_\lambda(\kappa)={\tt 2Ci}\\
0&\text{else}
\end{cases}
\end{aligned}
\end{equation}

\begin{lemma}
\label{l:length2}
Assume $\kappa\in\tau(\gamma), \kappa\not\in\tau(\lambda)$, $\ell(\kappa)=2$. 
Then
\begin{subequations}
\renewcommand{\theequation}{\theparentequation)(\alph{equation}}  
\begin{equation}
\boxed{\begin{aligned}
m_{\kappa,0}(\gamma,\lambda)&=\Mu_{-2}(\gamma,\lambda)
-\sum_{\substack{\mu|\kappa\in\tau(\mu)}}
\Mu_{-1}(\gamma,\mu)\Mu_{-1}(\mu,\lambda)\\
&-
\begin{cases}
\Mu_{-1}(\gamma,\lambda^\kappa)&t_\lambda(\kappa)={\tt 2Ci}\\
0&\text{else}
\end{cases}\\
&+\begin{cases}
\Mu_{-1}(\gamma_\kappa,\lambda)&t_\gamma(\kappa)={\tt 2Cr}\\
\epsilon(\gamma,\lambda)&\kappaarrow\gamma\lambda, t_\gamma(\kappa)\ne {\tt 2Cr}\\
0&\text{else}  
\end{cases}
\end{aligned}}
\end{equation}

and

\begin{equation}
\boxed{
m_{\kappa,1}(\gamma,\lambda)=\Mu_{-1}(\gamma,\lambda)+
\begin{cases}
\delta_{\gamma_\kappa,\lambda}&t_\gamma(\kappa)={\tt 2Cr}\\
0&\text{otherwise}
\end{cases}}
\end{equation}
\end{subequations}
\end{lemma}

Let's look at some cases.
First assume $\kappaarrow\gamma\lambda$. In particular
$\gamma>\delta$, so the first two terms are $0$. 
A little checking gives
\begin{equation}
m_{\kappa,0}(\gamma,\lambda)=
\begin{cases}
\epsilon(\lambda,\gamma)&t_\lambda(\kappa)\ne {\tt 2Ci}\\  
0&t_\lambda(\kappa)={\tt 2Ci}\\  
\end{cases}
\end{equation}
Putting this together with formula \eqref{e:m1gammadeltacase2} for
$m_{\kappa,1}$ we get:
\begin{equation}
\boxed{
\kappaarrow\gamma\lambda\Rightarrow
m_\kappa(\gamma,\lambda)=
\begin{cases}
v+v\inv&t_\lambda(\kappa)={\tt 2Cr}\\
\epsilon(\gamma,\lambda)&else  
\end{cases}
}
\end{equation}

Assume $\gamma\overset\kappa{\not\rightarrow}\lambda$. We see:
$$
\begin{aligned}
m_\kappa(\gamma,\lambda)&=
\Mu_{-2}(\gamma,\lambda)-
\sum_{\substack{\mu|\kappa\in\tau(\mu)}}
\Mu_{-1}(\gamma,\mu)\Mu_{-1}(\mu,\lambda)
\\
&-
\begin{cases}
\Mu_{-1}(\gamma,\lambda^\kappa)&t_\lambda(\kappa)={\tt 2Ci}\\
0&\text{else}  
\end{cases}
+
\begin{cases}
\Mu_{-1}(\gamma_\kappa,\lambda)&t_\gamma(\kappa)={\tt 2Cr}\\
0&\text{else}  
\end{cases}\\
&+\Mu_{-1}(\gamma,\lambda)(v+v\inv)
\end{aligned}
$$

If $\ell(\lambda)\not\equiv\ell(\gamma)\mypmod2$
all terms but the last are $0$,
so
\begin{equation}
\boxed{
\gamma\overset\kappa{\not\rightarrow}\lambda,\ell(\gamma)\not\equiv \ell(\lambda)\mypmod 2\Rightarrow
m_\kappa(\gamma,\lambda)=\Mu_{-1}(\gamma,\lambda)(v+v\inv)
}\end{equation}
On the other hand $\ell(\gamma)=\ell(\lambda)\mypmod2$ implies
the last term is $0$, and
\begin{equation}
\boxed{\begin{aligned}
m_\kappa(\gamma,\lambda)&=
\Mu_{-2}(\gamma,\lambda)-
\sum_{\substack{\mu|\kappa\in\tau(\mu)}}
\Mu_{-1}(\gamma,\mu)\Mu_{-1}(\mu,\lambda)
\\
&-
\begin{cases}
\Mu_{-1}(\gamma,\lambda^\kappa)&t_\lambda(\kappa)={\tt 2Ci}\\
0&\text{else}  
\end{cases}
+
\begin{cases}
\Mu_{-1}(\gamma_\kappa,\lambda)&t_\gamma(\kappa)={\tt 2Cr}\\
0&\text{else}  
\end{cases}
\end{aligned}}
\end{equation}

I believe this agrees with \cite[Theorem 4.4]{lv2012a}.

Note that all terms of $m_\kappa(\gamma,\delta)$ are $0$ unless $\kappaarrow\gamma\delta$ or
$\gamma<\delta$, except possibly the last two. 

\subsec{$\ell(\kappa)=3$}
We continue to assume
$\kappa\in\tau(\gamma), \kappa\not\in\tau(\lambda)$. In particular
$\gamma\ne\lambda$.

Take the $+$ part of both sides of \eqref{e:recursion3}. The left hand
side is:
\begin{subequations}
\renewcommand{\theequation}{\theparentequation)(\alph{equation}}  
\label{e:3main}
\begin{equation}
\begin{aligned}
&[m_{\kappa,0}(\gamma,\lambda)+
\sum_{\mu|\kappa\in\tau(\mu)}
\Mu_{-1}(\gamma,\mu)m_{\kappa,1}(\mu,\lambda)
+\sum_{\mu|\kappa\in\tau(\mu)} \Mu_{-2}(\gamma,\mu)m_{\kappa,2}(\mu,\lambda)]\\
&+[m_{\kappa,1}(\gamma,\lambda)+
\sum_{\mu|\kappa\in\tau(\mu)} \Mu_{-1}(\gamma,\mu)m_{\kappa,2}(\mu,\lambda)]v+m_{\kappa,2}(\gamma,\lambda)v^2
\end{aligned}
\end{equation}

The right hand side is
\begin{equation}
\begin{aligned}
\Mu_{-3}(\gamma,\lambda)+&\Mu_{-2}(\gamma,\lambda)v+
\Mu_{-1}(\gamma,\lambda)v^2+
\begin{cases}
\delta_{w_\kappa\times\gamma,\lambda}&{\tt 3C-}\\
\Mu_{-1}(\gamma_\kappa,\lambda)+\delta_{\gamma_\kappa,\lambda}v&{\tt 3Cr}\\
\Mu_{-1}(\gamma_\kappa,\lambda)+\delta_{\gamma_\kappa,\lambda}v&{\tt 3r}\\
0&3ic
\end{cases}
\end{aligned}
\end{equation}
Comparing the coefficient of $v^2$ gives
\begin{equation}
\boxed{m_{\kappa,2}(\gamma,\lambda)=\Mu_{-1}(\gamma,\lambda)}
\end{equation}
Plugging this in to \eqref{e:3main}, the coefficient of $v$ gives
\begin{equation}
m_{\kappa,1}(\gamma,\lambda)+\sum_{\mu|\kappa\in\tau(\mu)} \Mu_{-1}(\gamma,\mu)\Mu_{-1}(\mu,\lambda)
=
\Mu_{-2}(\gamma,\lambda)+
\begin{cases}
0&{\tt 3C-}\\
\delta_{\gamma_\kappa,\lambda}&{\tt 3Cr}\\
\delta_{\gamma_\kappa,\lambda}&{\tt 3r}\\
0&3ic
\end{cases}
\end{equation}
i.e.
\begin{equation}
\boxed{m_{\kappa,1}(\gamma,\lambda)=
\Mu_{-2}(\gamma,\lambda)-\sum_{\mu|\kappa\in\tau(\mu)}
 \Mu_{-1}(\gamma,\mu)\Mu_{-1}(\mu,\lambda)
+
\begin{cases}
0&t_\gamma(\kappa)={\tt3C-,3ic}\\
\delta_{\gamma_\kappa,\lambda}&t_\gamma(\kappa)={\tt 3Cr,3r}\\
\end{cases}}
\end{equation}

Turn the crank one more time, plugging this in, to compute  the constant term:
\begin{equation}
\begin{aligned}
m_{\kappa,0}(\gamma,\lambda)&=
-\sum_{\mu|\kappa\in\tau(\mu)}
\Mu_{-1}(\gamma,\mu)m_{\kappa,1}(\mu,\lambda)
-
\sum_{\mu|\kappa\in\tau(\mu)}
\Mu_{-2}(\gamma,\mu)m_{\kappa,2}(\mu,\lambda)+\\
&\Mu_{-3}(\gamma,\lambda)+
\begin{cases}
\delta_{w_\kappa\times\gamma,\lambda}&{\tt 3C-}\\
\Mu_{-1}(\gamma_\kappa,\lambda)&{\tt 3Cr}\\
\Mu_{-1}(\gamma_\kappa,\lambda)&{\tt 3r}\\
0&3ic
\end{cases}
\end{aligned}
\end{equation}
Plug in (c) and (e):
\begin{equation}
\begin{aligned}
&m_{\kappa,0}(\gamma,\lambda)=\\
-\sum_{\mu|\kappa\in\tau(\mu)}
&\Mu_{-1}(\gamma,\mu)
\Bigg[\Mu_{-2}(\mu,\lambda)-\sum_{\phi|\kappa\in\tau(\phi)}
\Mu_{-1}(\mu,\phi)\Mu_{-1}(\phi,\lambda)\\
+
&\begin{cases}
0&t_\lambda(\kappa)=3C-,{\tt 3ic}\\
\delta_{\lambda_\kappa,\lambda}&t_\lambda(\kappa)={\tt 3Cr, 3r}\\
\end{cases}\Bigg]
\\
-\bigg[\sum_{\mu|\kappa\in\tau(\mu)}
&\Mu_{-2}(\gamma,\mu)\Mu_{-1}(\mu,\lambda)\bigg]+
\Mu_{-3}(\gamma,\lambda)+
\begin{cases}
\delta_{w_\kappa\times\gamma,\lambda}&t_\gamma(\kappa)={\tt 3C-}\\
\Mu_{-1}(\gamma_\kappa,\lambda)&t_\gamma(\kappa)={\tt 3Cr}\\
\Mu_{-1}(\gamma_\kappa,\lambda)&t_\gamma(\kappa)={\tt 3r}\\
0&t_\gamma(\kappa)=3ic
\end{cases}
\end{aligned}
\end{equation}
Note that if $t_\gamma(\kappa)=3C-$,  $\delta_{w_\kappa\times\gamma,\lambda}=1$ 
if $\kappaarrow\gamma\lambda$, and $0$ otherwise.
Also, evaluating
\begin{equation}
\sum_\mu\Mu_{-1}(\gamma,\mu)*
\begin{cases}
0&t_\lambda(\kappa)=3C-,{\tt 3ic}\\
\delta_{\lambda_\kappa,\lambda}&t_\lambda(\kappa)={\tt 3Cr,  3r}\\
\end{cases}
\end{equation}
as in the length $2$ case gives 
\begin{equation}
\begin{cases}
\Mu_{-1}(\gamma,\lambda^\kappa)&t_\lambda(\kappa)=3Ci\text{ or }{\tt 3i}\\
0&else
\end{cases}
\end{equation}
Inserting this information, moving a few terms around, and taking
$\lambda<\lambda$ in all sums as in the previous cases, gives

\begin{equation}\boxed{
\begin{aligned}
m_{\kappa,0}(\gamma,\lambda)&=\Mu_{-3}(\gamma,\lambda)\\
&+\sum_{\substack{\mu|\kappa\in\tau(\mu)\\\phi|\kappa\in\tau(\phi)}}
\Mu_{-1}(\gamma,\mu)
\Mu_{-1}(\mu,\phi)\Mu_{-1}(\phi,\lambda)+\\
&-
\sum_{\mu|\kappa\in\tau(\mu)}
\big[
\Mu_{-1}(\gamma,\mu)\Mu_{-2}(\mu,\lambda)+
\Mu_{-2}(\gamma,\mu)\Mu_{-1}(\mu,\lambda)\big]\\
\\
&-\begin{cases}
\Mu_{-1}(\gamma,\lambda^\kappa)&t_\lambda(\kappa)={\tt 3Ci}\text{ or }{\tt 3i}\\
0&else
\end{cases}\\
&
+
\begin{cases}
1&t_\gamma(\kappa)={\tt 3C-}, \kappaarrow\gamma\lambda\\
0&t_\gamma(\kappa)={\tt 3C-}, \gamma\overset\kappa{\not\rightarrow}\lambda\\
\Mu_{-1}(\gamma_\kappa,\lambda)&t_\gamma(\kappa)={\tt 3Cr}\\
\Mu_{-1}(\gamma_\kappa,\lambda)&t_\gamma(\kappa)={\tt 3r}\\
0&t_\gamma(\kappa)={\tt 3ic}
\end{cases}
\end{aligned}}
\end{equation}
\end{subequations}
Summarizing the length 3 case:

\begin{lemma}
\label{l:length3}
Assume $\kappa\in\tau(\gamma), \kappa\not\in\tau(\lambda)$, $\ell(\kappa)=3$. 
Then $m_{\kappa,2},m_{\kappa,1},m_{\kappa,0}$ are given by 
\ref{e:3main}(c),(e), and (j), respectively.


\end{lemma}

\newpage

Let's look at some cases.

Suppose $\kappaarrow\gamma\lambda$. All $\Mu_{-i}$ terms are $0$, and 
\begin{equation}
\kappaarrow\gamma\lambda\Rightarrow
m_\kappa(\gamma,\lambda)=
\begin{cases}
  1&t_\gamma(\kappa)={\tt 3C-}\\
(v+v\inv)&t_\gamma(\kappa)={\tt 3Cr,  3r}\\
0&t_\gamma(\kappa)=3ic
\end{cases}
\end{equation}

Now assume $\gamma\overset\kappa{\not\rightarrow}\lambda$, and $\ell(\lambda)\equiv\ell(\gamma)\mypmod2$.
Then $m_{\kappa,2}(\gamma,\lambda)=m_{\kappa,0}(\gamma,\lambda)=0$.
The formula for $m_{\kappa,1}$ doesn't simplify, except that last term is $0$ since $\gamma\overset\kappa{\not\rightarrow}\lambda$, 
so
\begin{equation}
m_\kappa(\gamma,\lambda)=
\big[\Mu_{-2}(\gamma,\lambda)-\sum_{\mu|\kappa\in\tau(\mu)}
 \Mu_{-1}(\gamma,\mu)\Mu_{-1}(\mu,\lambda)\big](v+v\inv)
\end{equation}

Finally assume $\gamma\overset\kappa{\not\rightarrow}\lambda$, and $\ell(\lambda)\not\equiv\ell(\gamma)\mypmod2$.
Then $m_{\kappa,1}(\gamma,\lambda)=0$, and $m_{\kappa,0}$ simplifies a little, to give:
\begin{equation}
\begin{aligned}
m_\kappa(\gamma,\lambda)&=\Mu_{-1}(\gamma,\lambda)(v^2+v^{-2})+\Mu_{-1}(\gamma,\lambda)\\
&+\sum_{\substack{\mu|\kappa\in\tau(\mu)\\\phi|\kappa\in\tau(\phi)}}
\Mu_{-1}(\gamma,\mu)
\Mu_{-1}(\mu,\phi)\Mu_{-1}(\phi,\lambda)+\\
&-
\sum_{\mu|\kappa\in\tau(\mu)}
\big[
\Mu_{-1}(\gamma,\mu)\Mu_{-2}(\mu,\lambda)+
\Mu_{-2}(\gamma,\mu)\Mu_{-1}(\mu,\lambda)\big]\\
\\
&-\begin{cases}
\Mu_{-1}(\gamma,\lambda^\kappa)&t_\lambda(\kappa)=3Ci\text{ or }{\tt 3i}\\
0&else
\end{cases}\\
&
+
\begin{cases}
\Mu_{-1}(\gamma_\kappa,\lambda)&t_\gamma(\kappa)=3Cr\text{ or }{\tt 3r}\\
0&\text{else}
\end{cases}
\end{aligned}
\end{equation}

As in the length $2$ case, all terms are zero unless $\kappaarrow\gamma\lambda$ or 
$\gamma<\lambda$, except possibly the last two.



\sec{Appendix II: Some supplementary material}
\subsec{Explanation of the {\tt 1i2s/1r1s} cases}

This was originally a separate note on these cases. 

We recall some notation from \cite{lv2012b}. There is a space $\D$ of
parameters (Langlands parameters for $G$) with an action of $\sigma$.
There is a space of extended parameter $\wt\D$ for the extended
group $\Gext$. Each $\gamma\in\D^\sigma$ gives rise to 
two parameters $(\gamma,\pm)$ in $\wt\D$ (a $\sigma$-fixed representation
extends in two ways to the extended group). 
If $\sigma\gamma=\gamma'\ne\gamma$ then both $\gamma,\gamma'$ give 
a single parameter $(\gamma')=(\gamma)\in\wt\D$.
See \cite[Section 2.3]{lv2012b}; these are the elements 
$(\mathcal L,\pm q^k\beta^{\mathcal L})$ and
$(\mathcal L,q^{2k}\bold t^{\mathcal L})^\theta$. 

We have a Hecke algebra  $\H$, and an $\H$-module $\mathcal M$ with
basis $\{a_\mu\mid \mu\in \wt\D\}$. This is the module 
$\mathfrak K(\mathcal C)$ of \cite[Section 2.3]{lv2012b}. 
(It isn't entirely clear from \cite{lv2012b} that $\M$ carries an
action of $\H$, but David assures me this is so.)

This is {\it not} the main module $M$ of loc. cit., which is a
quotient of $\mathcal M$: $M$ has basis $\{a_\mu\mid \mu\in
\wt\D\}$ modulo relations:
\begin{subequations}
\renewcommand{\theequation}{\theparentequation)(\alph{equation}}  
\label{e:relations}
\begin{equation}
a_{(\gamma,\mp)}=-a_{(\gamma,\pm)}\quad (\gamma\in\D^\sigma)
\end{equation}
\begin{equation}
a_{(\gamma)}=0\quad(\gamma\in\D-\D^\sigma)
\end{equation}
\end{subequations}
See the discussion of the image of the homomorphism $\theta$, and the
definition of $M$, in 
\cite[Section 2.3]{lv2012b}.

Now suppose $\kappa=\{\alpha\}$ where $\alpha$ is a $\sigma$-fixed root,
$\gamma\in\D^\sigma$, and $\alpha$ is imaginary or real with respect to
$\gamma$. Associated to $\kappa$ is a Hecke operator $T_\kappa$.
The problem is to compute the action of $T_\kappa$. 
There are two easy cases and two hard cases.

\medskip
\noindent (a) {\tt 1i2f/1r1}. This means $\alpha$ is imaginary for
$\gamma$, 
$c^\alpha(\gamma)=\{\gamma',\gamma''\}$ is double valued, and
$\gamma',\gamma''\in\D^\sigma$ (one discrete series and two principal
series).
The {\tt f} refers to the fact that $\gamma',\gamma''$ are $\sigma$-fixed.

There is a
$6$-dimensional subspace of $\mathcal M$ on which the Hecke operator
$T_\kappa$ acts, with basis $a_{(\gamma,\pm)}, a_{(\gamma',\pm)}, a_{(\gamma'',\pm)}$
In the quotient $M$ this becomes $3$-dimensional by
\eqref{e:relations}(a). The matrix of $T_\kappa$ on this space is
\begin{equation}
\label{e:matrix1}
\begin{pmatrix}
  1&q-1&q-1\\
1&q-1&-1\\
1&-1&q-1
\end{pmatrix}
\end{equation}
with eigenvalues $u,u,-1$, and corresponding eigenvectors
$(1,1,0), (1,0,1)$ (eigenvalue $u$) and $(u-1,-1,-1)$ (eigenvalue
$-1$). 

\medskip

\noindent (b) {\tt 1i1/1r1f}. This means $\alpha$ is imaginary for
$\gamma$, 
$\gamma'=s_\alpha\gamma\ne\gamma$ and
$c^\alpha(\gamma)=c^\alpha(\gamma')=\gamma''$ is single valued (two
discrete series and one principal series) (also
$\sigma(\gamma)=\gamma$).
Again there is 
$6$-dimensional subspace of $\mathcal M$ on which the Hecke operator
$T_\kappa$ acts, in the quotient $M$ this becomes $3$-dimensional,
and the matrix of $T_\kappa$ on this space is
\begin{equation}
\label{e:matrix2}
\begin{pmatrix}
0&1&q-1\\
1&0&q-1\\
1&1&q-2
\end{pmatrix}
\end{equation}
with eigenvalues $q,-1,-1$, and corresponding eigenvectors
$(1,1,1)$ (eigenvalue $q$) and $(q-1,0,-1),(0,q-1,-1)$ (eigenvalue
$-1$). 

Now the hard cases.

\medskip

\noindent (c) {\tt 1i2s}: Just as in the {\tt 1i2f} case we have three parameters $\gamma$ and
$c^\alpha(\gamma)=\{\gamma',\gamma''\}$, except now 
$\sigma$ switches $\gamma',\gamma''$ 
(hence the {\tt s}). 

Now we need to be careful. There is a $3$ dimensional space  $\mathcal
V$ invariant
by $T_\kappa$, with basis
$$
a_{(\gamma,+)},a_{(\gamma,-)},a_{(\gamma')}
$$
Note that $(\gamma')=(\gamma'')$. In the quotient $M$ we have
$$
a_{(\gamma,-)}=-a_{(\gamma,+)}\quad\text{by }\eqref{e:relations}(a)
$$
and
$$
a_{(\gamma')}=0\quad\text{by }\eqref{e:relations}(b).
$$
The subspace we are modding out by is spanned by
$a_{(\gamma,+)}+a_{(\gamma,-)}$ and $a_{(\gamma')}$, i.e. $(1,1,0)$ and $(0,0,1)$ in the given basis.
The quotient $V$ is $1$ dimensional. 

So to calculate the action of $T_\kappa$ on the one-dimensional space
$V$,  calculate it on the $3$-dimensional space $\mathcal
V$, and mod out by the span of $(1,1,0),(0,0,1)$ (which better be
$T_\kappa$ invariant). 

Now comes some guesswork. Recall we started with $1$ discrete series
$\gamma$, and two principal series $\gamma',\gamma''$. However on the
extended group this becomes two discrete series $(\gamma,\pm)$ and one
principal series $(\gamma')$. This suggests that the action of the 
Hecke operator $T_\kappa$ on $\mathcal M$ is not \eqref{e:matrix1}
(the {\tt 1i2f} case) but rather 
\eqref{e:matrix2} (from the {\tt 1i1} case). 

\begin{conjecture}
In the {\tt 1i2s} case $T_\kappa$ acts on $\mathcal M$, with basis
$a_{(\gamma,+)},a_{(\gamma,-)},a_{(\gamma')}$ with matrix
\begin{equation}
\begin{pmatrix}
0&1&q-1\\
1&0&q-1\\
1&1&q-2
\end{pmatrix}
\end{equation}
\end{conjecture}
Recall this matrix has eigenvalues $q,-1,-1$.

Assuming this, the subspace spanned by $(1,1,0)$ and $(0,0,1)$ is
$T_\kappa$-invariant: $(1,1,1)$ has eigenvalue $q$ and $(q-1,q-1,-2)$
has eigenvalue $-1$. Therefore there is one remaining eigenvalue $-1$,
and we conclude

\begin{lemma} Assuming the conjecture, in the {\tt 1i2s} case 
$T_\kappa$ acts with eigenvalue $-1$ on the one-dimensional space $V$.
\end{lemma}

\medskip
\noindent (d) {\tt 1r1s} As in the {\tt 1r1f} case there are two
discrete series $\gamma,\gamma'$, one principal series $\gamma''$,
except that now $\sigma(\gamma)=\gamma'$. There is a $3$-dimensional
space $\mathcal V$ spanned by
$$
a_{(\gamma)},a_{(\gamma'',+)},a_{(\gamma'',-)}
$$
(recall $(\gamma)=(\gamma')$). For the quotient we have relations
$$
a_{(\gamma)}=0
$$ 
and
$$
a_{(\gamma'',-)}=-a_{(\gamma'',+)}
$$
so the subspace is spanned by $(1,0,0)$ and $(0,1,1)$. 

\begin{conjecture}
In the {\tt 1r1s} case the matrix of $T_\kappa$, in the basis
$a_{(\gamma)},a_{(\gamma'',+)},a_{(\gamma'',-)}$, is
\begin{equation}
\label{e:matrix1}
\begin{pmatrix}
  1&q-1&q-1\\
1&q-1&-1\\
1&-1&q-1
\end{pmatrix}
\end{equation}
\end{conjecture}
Recall this matrix has eigenvalues $q,q,-1$.

Assuming the conjecture, the subspace spanned by $(1,0,0)$ and
$(0,1,1)$ is $T_\kappa$-invariant: $(2,1,1)$ has eigenvalue $q$ and
$(u-1,-1,-1)$ has eigenvalue $-1$. 

\begin{lemma}
Assuming the conjecture, in the {\tt 1r1s} case  $T_\kappa$ acts on the one-dimensional space
$V$ with eigenvalue $q$. 
\end{lemma}

\bigskip

\noindent{\bf\Large Conclusion}

Assume the conjectures. 

In case {\tt 1i2s} there is a single fixed discrete series parameter
$\gamma$, with two extensions $(\gamma,\pm)$. 
Recall $a_{(\gamma,+)}=-a_{(\gamma,-)}$.
Then (in $M$):
\begin{subequations}
\renewcommand{\theequation}{\theparentequation)(\alph{equation}}  
\begin{equation}
T_\kappa a_{(\gamma,+)}=-a_{(\gamma,+)}
\end{equation}
In case {\tt 1r1s} there is a single fixed principal series parameter
$\gamma$, with two extensions $(\gamma,\pm)$. 
Recall $a_{(\gamma,+)}=-a_{(\gamma,-)}$. Then (in $M$):
\begin{equation}
T_\kappa a_{(\gamma,+)}=qa_{(\gamma,+)}
\end{equation}
\end{subequations}
In the notation of my  notes {\it Computing Twisted KLV polynomials}, Section 6, these
would be written simply
$$
\boxed{\begin{aligned}
T_\kappa(a_\gamma)&=-a_\gamma\quad({\tt 1i2s}\text{ case})\\
T_\kappa(a_\gamma)&=qa_\gamma\quad({\tt 1r1s}\text{ case})
\end{aligned}}
$$

And here is an email from Marc completing the argument.

\begin{verbatim}
Date: Fri, 13 Dec 2013 17:05:09 +0100
From: Marc van Leeuwen <Marc.van-Leeuwen@math.univ-poitiers.fr>
To: Jeffrey Adams <jda@math.umd.edu>
CC: David Vogan <dav@math.mit.edu>, 
Subject: Confirmation: 1i2s must be an ascent (Re: 1i2s/1r1s)


On 06/12/13 04:18, Jeffrey Adams wrote:

> Marc raised a (yet another) valid objection to my formulas. After
> talking to David I arrived at the resolution explained in the attached
> file. I'm quite confident in the whole picture. However I'm not
> confident in my ability to calculate the 3x3 matrices in the two
> conjectures. I give a plausability argument for them, and hope that
> David is able to confirm or fix them.

I am not able to see how to compute those 3x3 matrices by sheer brain power
either. However, what I can do is compute braid relations. I needed a case
where one of these types occur, and fortunately there is an easy one:

> empty: type
> Lie type: A3 sc s
> main: extblock 
> (weak) real forms are:
> 0: sl(2,H)
> 1: sl(4,R)
> enter your choice: 1
> possible (weak) dual real forms are:
> 0: su(4)
> 1: su(3,1)
> 2: su(2,2)
> enter your choice: 1
> Name an output file (return for stdout, ? to abandon): 
> 0  1  [2C+ ,1rn ]  4  0   (*,*)  (*,*)  2^e
> 4  3  [2C- ,1i2s]  0  4   (*,*)  (*,*)  1x2^e

Now there is the first Hecke generator $T_{1,3}$, which acts by the companion
matrix $C$ of the quadratic relation $(X-q^2)(X+1)=X^2-(q^2-1)X-q^2$, and the
second Hecke generator $T_{2}$ which acts by a diagonal matrix $D$ with
diagonal coefficients $x=-1$ and $y\in\{-1,q\}$, the latter depending on
whether type 1i2s is an ascent ($y=-1$) or a descent ($y=q$). If it is an
ascent, then the two matrices obviously commute, and the required braid
relation $CDCD=DCDC$ is satisfied. If however $y=q$, then the matrices do not
commute, and a fairly easy computation shows that the off-diagonal coefficients
of CDCD and DCDC would not match up. Therefore 1i2s must be an ascent. I
suppose that by duality (more or less) 1r1s must be a descent, but I did not
really check.

This confirms what Jeff wrote. Cheers,

-- Marc

\end{verbatim}

\subsec{email from David Vogan regarding the outline}
\begin{verbatim}

Date: Thu, 07 Nov 2013 14:31:17 -0500
From: David Vogan <dav@math.mit.edu>
To: Marc.van-Leeuwen@math.univ-poitiers.fr, jeffreydavidadams@gmail.com
Subject: Re: twisted KLV

Dear Marc,

You're absolutely correct that none of the references gives a very
clear picture of how the recursions work.  The point of Jeff's notes
"Computing twisted KLV polynomials" ([CTKLP]) was to do that, and I
think he does a good job of writing down all the details properly; but
my understanding of your objection is that you want to know where the
details come from in order to be able to implement them reliably.

So here is the picture. We have the Coxeter group (W,S) with
involutive automorphism \sigma (defining an involutive automorphism of
S). Therefore W^\sigma is a Coxeter group with one generator \kappa
for each orbit (also called \kappa) of \sigma on S. The orbit \kappa
defines a Levi subgroup of W of type A_1 or A_1 x A_1 or A_2, which
has a long element w_\kappa of length 1, 2, or 3 accordingly.

Accordingly we get an unequal parameter Hecke algebra with one
generator T_\kappa for each orbit. The new-looking relation is (2.2)
in [CTKLP]:

(T_\kappa + 1)(T_\kappa - u^{\ell(w_\kappa)}) = 0. 

Of course this means that in any module, T_\kappa has eigenvalues -1
and u^{\ell(w_\kappa)}.  It's convenient to introduce v = u^{1/2} and
work instead with

\widehat T_{\kappa} = v^{-\ell(w_\kappa)}(T_\kappa + 1)

(Fokko's "c_s"). This element has eigenvalues 0 and
(v^{\ell(w_\kappa)} + v^{-\ell(w_\kappa)}).  Obviously

ker(\widehat T_{\kappa}) = zero eigenspace
im (\widehat T_{\kappa}) = v^{\ell(w_\kappa)}+v^{-\ell(w_\kappa)} eigenspace.

As you know, the module for the Hecke algebra has a "standard" basis
of various \widehat a_\gamma, and a KL basis of \widehat C_\gamma,
both indexed by the same "parameters" \gamma; and of course the
(twisted) KL polynomials are the transition matrix for expressing the
\widehat C_\delta in terms of the \widehat a_\gamma.  Here are the key
facts.

1. The action of the \widehat T_\kappa on the \widehat a_\gamma is
known, (involving more or less just Cayleys and crosses on \gamma by
\kappa). 

2. For each \kappa, the parameters divide more or less evenly into
those for which \kappa is a DESCENT and those for which it is an
ASCENT. These terms are defined by (3) below.

3. \kappa is a DESCENT for \gamma if and only if 

\widehat T_{\kappa} \widehat C_\gamma =
      (v^{\ell(w_\kappa)}+v^{-\ell(w_\kappa)}) \widehat C_\gamma.

The \widehat C_\gamma with \gamma a descent are a basis of this
eigenspace of \widehat T_{\kappa}.

4. If \kappa is an ASCENT for \lambda, then 

\widehat T_{\kappa} \widehat C_\lambda = combination of \widehat
                                        C_\gamma as in 3.

Statement 4. is a consequence of the last assertion in 3. and the
(obvious) statements about eigenspaces, ker, and im above. It's Lemma
9.3.5 in [CTKLP], made more explicit in Theorem 9.3.10.

Statement 3. leads to the "easy recursions," because it relates the
the KL polynomials for \gamma and \delta to those for \gamma and
[\kappa-Cayleys and crosses of \delta].

The way to use Statement 4. is in constructing \widehat C_\lambda by
induction on \lambda. Given a new and unknown \lambda', try to write
it as cross or Cayley of a shorter \lambda. If you can do this, then
\widehat C_\lambda' will appear on the right side of the formula in
4., probably with some very simple coefficient. The left side of 4. is
known. Try to see that all the other C_\gamma on the right in this
formula are already known, and their coefficients are already known;
then you can solve 4. for the unknown C_\lambda'.

I've written too many words, I'm afraid, but at least it's shorter
than the references. One last point: it might or might not help to
list my Park City paper as one of the references for the classical KLV
case. But my copy of that has a reasonable number of typos marked on
it, so don't read it too closely.

Take care,
David
\end{verbatim}

\subsec{Further explanation of the algorithm}

\begin{verbatim}
Date: Thu, 19 Dec 2013 21:08:13 +0100
From: Marc van Leeuwen <Marc.van-Leeuwen@math.univ-poitiers.fr>
To: Jeffrey Adams <jeffreydavidadams@gmail.com>
Subject: Re: modules
...
I've not stumbled on anything too difficult to understand, but I do
admit being puzzled about where this is going. I've checked table
9.1.3, which is OK except for a minus sign at 1r1s that (given the
mentioned resolution) should become '+', but I'm not entirely
convinced of the utility of the simplification (at the expense of the
new notation $\hat a_\lambda^\kappa$); even more so about the
notations 'def_\lambda' and \zeta_\kappa that follow. It looks like
there are (a lot) more rewritings of the same stuff before one comes
to actual recursion relations; maybe you could give an idea of the big
picture for these recursions that would help understanding why these
kind of reformulations are useful/necessary. It might help me speed up
my reading; I am not particularly good at digesting long and numerous
formulas .
\end{verbatim}

Some explanation\dots

The normalization $\hat a_\gamma=v^{-\ell(\gamma)}a_\gamma$ (which I got from Fokko) is just for convenience.
On the other hand the terminology 
$\hat a_\gamma^{\kappa}$ is more serious: this is designed 
to make Lemma \ref{l:kappadescents}(3) hold, i.e. 
$\hat a_\gamma^{\kappa}$ for $\kappa\in\tau(\gamma)$ are a basis of the image of $\wh T_\kappa$. 

The $\defect_\lambda$ and $\zeta_\kappa$ terminology are intended
partly to make coding easier: I assume it is easier to code (and
debug) a smaller table like \ref{table:akappagammainTClambda2} than
the bigger one \ref{table:akappagammainTClambda}. But whatever works
best is fine.

As far as the algorithm goes, on the one hand 
if $\kappa\not\in\tau(\lambda)$ then
Section \ref{s:TChatintermsofa} gives
\begin{subequations}
\renewcommand{\theequation}{\theparentequation)(\alph{equation}}  
\begin{equation}
\begin{aligned}
\T_\kappa(\Chat_\lambda)=\sum_{\gamma\mid\kappa\in\tau(\gamma)} c(\gamma,\lambda)\wh a_\gamma^\kappa
\end{aligned}
\end{equation}
where $c(\gamma,\lambda)$ can be computed provided we know various $\P(*,\lambda)$
,see Lemma \ref{l:coefficient}, which we will.

On the other hand Theorem \ref{t:T_kappa} says that 
\begin{equation}
\T_\kappa(\Chat_\lambda)=
\sum_{\gamma|\kappa\in\tau(\gamma)}m_\kappa(\gamma,\lambda)\Chat_\gamma.
\end{equation}
for certain coefficients 
 $m_\kappa(\gamma,\lambda)$.

Suppose $\kappa\in\tau(\gamma),\tau(\mu)$, and $\kappaarrow\mu\lambda$, 
and compare the coefficients of $\wh a^\kappa_\gamma$ in (a) and (b). 
We hope to get a formula for $\P(\gamma,\mu)$.

To compute $P(\gamma,\mu)$ we may assume we know:
$$
\begin{aligned}
\P(*,\mu')&\text{ if }\ell(\mu')<\ell(\mu)\\
\P(\gamma',\mu)&\text{ if }\ell(\gamma')>\ell(\gamma).
\end{aligned}
$$
So, since $\ell(\lambda)<\ell(\mu)$ we know $c(\gamma,\lambda)$ in (a).
On the other hand the coefficient of $\wh a_\gamma^\kappa$ in (b) 
is {\it roughly speaking}
\begin{equation}
\P(\gamma,\mu)+
\sum_{\substack{\delta|\kappa\in\tau(\delta)\\\ell(\delta)<\ell(\lambda)}}\P(\gamma,\delta)m_\kappa(\delta,\lambda)
\end{equation}
Setting this equal to $c(\gamma,\lambda)$ we conclude
\begin{equation}
\P(\gamma,\mu)=
c(\gamma,\lambda)-\sum_{\substack{\delta|\kappa\in\tau(\delta)\\\ell(\delta)<\ell(\lambda)}}\P(\gamma,\delta)m_\kappa(\delta,\lambda)
\end{equation}
Since $\ell(\delta)<\ell(\lambda)<\ell(\mu)$ we know $\P(\gamma,\delta)$. 
Again, {\it roughly speaking}, $m_\kappa(\delta,\lambda)$ is in terms of various $\P(*,\lambda')$ with 
$\ell(\lambda')\le \ell(\lambda)<\ell(\mu)$, and $\P(\delta,\mu)$ with $\ell(\delta)>\ell(\gamma)$, which we also know. 
So by induction we can compute $\P(\gamma,\mu)$. 

This argument works precisely as stated in some cases. 
However it can run into trouble in one or both {\it roughly speaking} clauses. 
For one thing the left hand side of (d) may have two terms.
(Actually the left hand side of (d) is multiplied by  $\pm1$ or $(v+v\inv)$, 
from Theorem \ref{t:T_kappa}(1), but this isn't serious). 
For another it isn't true that $\ell(\delta)<\ell(\lambda)$ in (c), only that this holds in most cases,
and for most terms. See Section \ref{s:analysis} for details.

\end{subequations}

\bibliographystyle{plain}
\def\cprime{$'$} \def\cftil#1{\ifmmode\setbox7\hbox{$\accent"5E#1$}\else
  \setbox7\hbox{\accent"5E#1}\penalty 10000\relax\fi\raise 1\ht7
  \hbox{\lower1.15ex\hbox to 1\wd7{\hss\accent"7E\hss}}\penalty 10000
  \hskip-1\wd7\penalty 10000\box7}
  \def\cftil#1{\ifmmode\setbox7\hbox{$\accent"5E#1$}\else
  \setbox7\hbox{\accent"5E#1}\penalty 10000\relax\fi\raise 1\ht7
  \hbox{\lower1.15ex\hbox to 1\wd7{\hss\accent"7E\hss}}\penalty 10000
  \hskip-1\wd7\penalty 10000\box7}
  \def\cftil#1{\ifmmode\setbox7\hbox{$\accent"5E#1$}\else
  \setbox7\hbox{\accent"5E#1}\penalty 10000\relax\fi\raise 1\ht7
  \hbox{\lower1.15ex\hbox to 1\wd7{\hss\accent"7E\hss}}\penalty 10000
  \hskip-1\wd7\penalty 10000\box7}
  \def\cftil#1{\ifmmode\setbox7\hbox{$\accent"5E#1$}\else
  \setbox7\hbox{\accent"5E#1}\penalty 10000\relax\fi\raise 1\ht7
  \hbox{\lower1.15ex\hbox to 1\wd7{\hss\accent"7E\hss}}\penalty 10000
  \hskip-1\wd7\penalty 10000\box7} \def\cprime{$'$} \def\cprime{$'$}
  \def\cprime{$'$} \def\cprime{$'$} \def\cprime{$'$} \def\cprime{$'$}
  \def\cprime{$'$} \def\cprime{$'$}

\enddocument
\end